\documentclass[12pt]{amsart}
\usepackage{amssymb}
\usepackage{amsmath,amscd}
\usepackage{amsfonts,latexsym,verbatim,amscd,mathrsfs,color,array}
\usepackage[all,cmtip]{xy}
\usepackage{mathrsfs}
\usepackage{multirow}
\usepackage{graphicx}
\usepackage{amsfonts, amsthm}
\usepackage{bbm}
\usepackage[hmargin=1in,vmargin=1in]{geometry}
\usepackage{mathdots}
\usepackage{stmaryrd}
\usepackage{MnSymbol}
\usepackage{bbm}
\usepackage[hidelinks]{hyperref}
\usepackage{enumitem}
\usepackage[all]{xy}
\usepackage{tikz-cd}
\usepackage{empheq}

\allowdisplaybreaks

\def\XXint#1#2#3{{\setbox0=\hbox{$#1{#2#3}{\int}$ }
		\vcenter{\hbox{$#2#3$ }}\kern-.6\wd0}}

\renewcommand{\O}{\mathrm{O}}

\newcommand{\Tr}{\operatorname{Tr}}
\newcommand{\transint}{\cap\kern-0.63em|\kern0.7em}

\DeclareMathSymbol{\intprod}{\mathbin}{MnSyC}{'270}

\newcommand{\YM}{\mathcal{YM}}

\newcommand{\LB}{\left[}
\newcommand{\RB}{\right]}

\newcommand{\LA}{\left\langle}
\newcommand{\RA}{\right\rangle}

\newcommand{\p}{{ \partial}}

\newcommand{\Z}{{\mathbb Z}}
\newcommand{\N}{{\mathbb N}}
\newcommand{\C}{{\mathbb C}}
\newcommand{\CP}{{\mathbb{CP}}}

\newcommand{\R}{{\mathbb R}}
\newcommand{\calB}{\mathcal{B}}

\newcommand{\SU}{{\mathrm{SU} }}

\newcommand{\Sp}{{\mathrm{Sp} }}

\DeclareMathOperator{\Ad}{Ad}
\newcommand{\U}{{\mathrm U}}

\newcommand{\SO}{{\mathrm{SO} }}

\newcommand{\G}{{\mathrm G}}

\newcommand{\GL}{{\mathrm{GL} }}

\newcommand{\Spin}{{\mathrm{Spin} }}

\newcommand{\rk}{{\mathrm{rk}}}

\newcommand{\gothg}{{\mathfrak g}}

\newcommand{\calA}{\mathcal{A}}

\newcommand{\calM}{\mathcal{M}}

\newcommand{\calN}{\mathcal{N}}

\newcommand{\End}{{\operatorname{End} \,}}

\renewcommand{\p}{{\partial}}
\newcommand{\eps}{{\varepsilon}}

\newtheorem{thm}{Theorem}[section]
\newtheorem{lemma}[thm]{Lemma}
\newtheorem*{lemma*}{Lemma}

\newtheorem{cor}[thm]{Corollary}

\newtheorem*{conj*}{Conjecture}

\newenvironment{claim}{\par\medskip\noindent\textit{Claim.}\space}{\par\medskip}
\newenvironment{claimproof}{\par\noindent\textit{Proof of claim.}\space}{\hfill$\diamond$\medskip\par}

   \newtheoremstyle{others}
     {3pt}
     {2pt}
     {}
     {}
     {\bf}
     {.}
     {.5em}
     {}

\theoremstyle{others}
\newtheorem{rmk}[thm]{Remark}
\newtheorem*{rmk*}{Remark}
\newtheorem{defn}[thm]{Definition}

\newtheorem*{question*}{Question}

\numberwithin{equation}{section}

\usepackage{mathtools} 
\usepackage{faktor}
\usepackage{dynkin-diagrams}
\usepackage{mathrsfs}
\usepackage{hyperref}
\usepackage{tikz-cd}
\usepackage{bbm}

\newcommand{\inj}{\hookrightarrow}

\usepackage{romanbar}

\newtheorem*{thm*}{Theorem}

\begin{document}

\title[Parabolic gap theorems]{Parabolic gap theorems for the Yang-Mills energy}
\author{Anuk Dayaprema}
\address{University of Wisconsin, Madison}
\email{dayaprema@wisc.edu}
\author{Alex Waldron}
\address{University of Wisconsin, Madison}
\email{waldron@math.wisc.edu}
\date{\today}

\begin{abstract}
    We prove parabolic versions of several known gap theorems in classical Yang-Mills theory. 
    On an $\SU(r)$-bundle of charge $\kappa$ over the 4-sphere, we show that the space of all connections with Yang-Mills energy less than $4 \pi^2 \left( |\kappa| + 2 \right)$ deformation-retracts under Yang-Mills flow onto the space of instantons, allowing us to simplify the proof of Taubes's path-connectedness theorem.
    On a compact quaternion-K\"ahler manifold with positive scalar curvature, we prove that the space of pseudo-holomorphic connections whose $\mathfrak{sp}(1)$ curvature component has small Morrey norm deformation-retracts under Yang-Mills flow onto the space of instantons. On a nontrivial bundle over a compact manifold of general dimension, we prove that the infimum of the scale-invariant Morrey norm of curvature is positive.
\end{abstract}

\maketitle

\tableofcontents

\thispagestyle{empty}

\section{Introduction}


Classical Yang-Mills theory is concerned with two closely-related 
semi-elliptic partial differential equations involving a \emph{connection} (known to physicists as a \emph{gauge field}) on a vector or principal bundle over an orientable Riemannian manifold. Solutions of the first are known as \emph{instantons}, while solutions of the second are known as \emph{Yang-Mills connections}. These equations are of first and second order, respectively, in the connection. 

In the 4-dimensional setting, an instanton $A$ satisfies the 
\emph{anti-self-duality equation}
$$F_A^+ = 0,$$
up to reversing orientation.
A Yang-Mills connection satisfies the \emph{Yang-Mills equation} 
$$D_A^*F_A = 0.$$
The latter is the Euler-Lagrange equation of the \emph{Yang-Mills energy functional}
$$\YM(A) = \frac12 \int_M |F_A|^2\, dV.$$
Since an instanton minimizes $\YM(\cdot)$ among all connections on the given bundle, it must be Yang-Mills.
However, 
the converse fails in general; for instance,
by the celebrated work of Sibner, Sibner, and Uhlenbeck
\cite{sibnersuhlenbeck}, we know that even the trivial $\SU(2)$-bundle over $S^4$ carries a vast supply of non-minimal Yang-Mills connections. 

Positive results in the converse direction are nonetheless common in the literature. Since the work of Bourguignon, Lawson, and Simons \cite{bourguignonlawson, bourguignonlawsonsimons}, these have fallen into two categories: \emph{stability theorems}, stating that a Yang-Mills connection with nonnegative second variation must be an instanton, and \emph{gap theorems}, stating that a Yang-Mills connection with appropriately small self-dual curvature
must be an instanton. 
We begin by reviewing a number of existing gap theorems.

The $L^2$ gap theorem on $S^4,$ due to Min-Oo \cite{minoogap}, runs as follows. Suppose that $A$ is
a Yang-Mills connection on a vector bundle $E \to S^4$ with
$$\YM(A) < 4 \pi^2 \left( |\kappa(E)| + \delta \right),$$
where $\delta$ is a positive universal constant and $\kappa(E)$ is the characteristic number defined in \S \ref{ss:characteristicnumber} below; then 
$A$ is an instanton. See Donaldson-Kronheimer \cite[\S 2.3.10]{donkron} for a proof, and Dodziuk and Min-Oo \cite{dodziukminoo} for an improved constant.
More recently, in the case that $E$ has structure group $\SU(r)$, Gursky, Kelleher, and Streets \cite{gurskykelleherstreets} have managed to improve the value to $\delta = 2,$ or twice the Yang-Mills energy of the unit-charge instanton. 
This is consistent with known examples of non-minimal Yang-Mills connections on $S^4$ \cite{bor, sadunsegert, sibnersuhlenbeck}.

More generally, the existence of an $L^2$ gap is known for Riemannian 4-manifolds where the scalar curvature dominates the self-dual Weyl tensor \cite{gerhardtgap, parker4Dgaugetheories},  
for compact K\"ahler surfaces with positive scalar curvature \cite{huangkahlersurfaces}, for bundles where the infimum of the Yang-Mills energy is zero \cite{feehanflatgap}, and for certain bundles over compact 4-manifolds with generic metrics 
\cite{feehanfourgap} (based on the Freed-Uhlenbeck Theorem \cite[Theorem 3.17]{freeduhlenbeck}).
Under mild geometric constraints, Vieira extended the Gursky-Kelleher-Streets result to noncompact manifolds with positive Yamabe constant \cite{vieirapositiveyamabegap}. Interestingly, the existence of an $L^2$ gap on a general compact Riemannian 4-manifold remains unknown. 

The notion of instanton can be extended to a higher-dimensional manifold endowed with an $(n-4)$-calibration 
$\Psi$ \cite{corrigandevchand, tiancalibrated}: a connection \(A\) is called a (\(\Psi\)-)\emph{instanton} if
\begin{align*}
    \ast (F_A \wedge \Psi) = -F_A.
\end{align*}
In particular, manifolds with special holonomy 
come with such calibrations, 
and the gap question is basic in higher-dimensional gauge theory. 
No firm gap results exist for $\G_2$- or $\Spin(7)$-manifolds, and the question appears to be a difficult one. However, on a quaternion-K{\"a}hler manifold with positive scalar curvature, Taniguchi established an \(L^\infty\) gap \cite{qksupgap} for pseudo-holomorphic (see \cite[\textsection 3.3]{oliveirawaldron}) Yang-Mills connections, and later an \(L^{\frac{n}{2}}\) gap \cite{qkL2gap}. The \(L^\infty\) result can be seen as a generalization of Bourguignon-Lawson-Simons's 
theorem on \(S^4\), while the \(L^{\frac{n}{2}}\) result can be seen as a generalization of Min-Oo's theorem. \par

Although higher-dimensional manifolds with generic holonomy lack a natural class of minimizers for \(\mathcal{YM}\) other than the flat connections, the gap question can still be posed at the zero energy.
Bourguignon-Lawson \cite{bourguignonlawson} proved an \(L^\infty\) gap theorem for the full curvature of a connection on a bundle over \(S^n, n \geq 4\). Gerhardt \cite{gerhardtgap} proved an \(L^{\frac{n}{2}}\) gap theorem for compact manifolds where the Riemann curvature satisfies a certain pointwise positivity condition. Zhou \cite{zhougap} has obtained \(L^p\), \(p \geq \frac{n}{2}\), gaps which depend on the constants in the global Sobolev inequality for functions on the manifold. The existence of a (non-explicit) $L^2$ gap near the zero energy is easily proved by combining standard estimates with the well-known \L ojasiewicz-Simon inequality; see Corollary \ref{cor:ellipticgap} 
below. 

Notice that these gap results all share an obvious drawback, namely, that the connection in question must solve the Yang-Mills equation {\it a priori}. 
Although Uhlenbeck's direct method succeeds at producing solutions
in dimension less than four, the Yang-Mills equation is itself extremely difficult to solve in dimension at least four. This is due both to the gauge freedom of the Yang-Mills functional and to the famous ``bubbling phenomenon'' which begins in dimension four and is less-well understood in higher dimensions.

The goal of this paper is to broaden the scope of these results by removing the 
assumption that the relevant connection be Yang-Mills. As an alternative,
we consider solutions of the Yang-Mills \emph{flow}:
\begin{equation}\label{ymf}
\frac{\p A}{\p t} = - D_A^*F_A.
\end{equation}
In the same circumstances 
as several of the elliptic gap theorems discussed above, we will 
establish that (\ref{ymf})
gives a deformation retraction from an appropriate 
neighborhood in the space of all connections 
onto the 
subspace of instantons. A result of this kind will be called a \emph{parabolic gap theorem}.

We will also demonstrate a few applications of our main results.

\vspace{5mm}

\subsection{Deformation to instantons on the 4-sphere}

Our first result is a parabolic gap theorem on $S^4$ based on Gursky-Kelleher-Streets's improvement of the Min-Oo gap theorem and the second-named author's thesis. 

\begin{thm}\label{thm:fourspheregap}
Let $A_0$ be a smooth connection on an $\SU(r)$-bundle $E \to S^4$ with energy
\begin{equation}\label{fourspheregap:energylessthan8pi2}
\YM(A_0) < 4 \pi^2 \left( |\kappa(E)|+ 2 \right).
\end{equation}
If $A(t)$ solves (\ref{ymf}) with $A(0) = A_0,$ then $A(t)$ converges smoothly to an instanton on $E$ as $t \to \infty.$ Moreover, the map $A_0 \mapsto \lim_{t \to \infty} A(t)$ gives a deformation retraction from the space of connections satisfying (\ref{fourspheregap:energylessthan8pi2}) onto the space of instantons. 
\end{thm}

Note that if we ask only for Uhlenbeck convergence along a subsequence, 
then the result follows directly from \cite{gurskykelleherstreets} and \cite{lte}; 
the main point is that the convergence takes place smoothly. 
We also note that the statement is true both before or after modding out by the gauge group.

The proof depends on an optimal ``split'' $\eps$-regularity theorem for the self-dual curvature along (\ref{ymf}), which is proved using a type-II blowup argument appealing to Uhlenbeck's compactness and removable singularity theorems. 
By the key estimate of \cite{instantons}, this guarantees not only long-time existence but smooth convergence when working 
on $S^4.$
Theorem \ref{thm:fourspheregap} can also be compared with \cite[Theorem 1.3]{gurskykelleherstreets}, which gives the same result on the trivial bundle. 

As an application of Theorem \ref{thm:fourspheregap}, we can give a streamlined proof of the following classic theorem of Taubes.
\begin{thm}[Taubes \cite{taubespathconnected}, 1984]\label{thm:pathconnected}
For any $\SU(r)$-bundle $E \to S^4,$ the moduli space of instantons on $E$ is path-connected.
\end{thm}
\noindent In our variant, described in \S \ref{ss:pathconnected}, all of the hard analysis in Taubes's proof is replaced by Theorem \ref{thm:fourspheregap}. 
We can also remove the requirement that $r = 2$ or $3,$ since the GKS gap theorem \cite{gurskykelleherstreets} renders the use of Taubes's index bounds \cite{taubesstability} unnecessary. For general $r,$ connectedness of the $\SU(r)$ moduli spaces on $S^4$ was previously observed by Donaldson \cite[p. 454]{donaldsoninstantonsandgit} using his algebraic description of framed instantons.

Last, it is interesting to point out that Theorem \ref{thm:pathconnected} provides a purely differential-geometric proof of the completeness of the ADHM construction \cite{adhm}
---see Corollary \ref{cor:adhm} below. We include this observation (without claiming any originality)
in order to highlight the significance of Taubes's theorem. 





\vspace{5mm}
 
\subsection{Deformation to instantons on quaternion-K\"ahler manifolds}

Let $M$ be a quater-nion-K{\"a}hler manifold of dimension $4m$. For \(m \geq 2\), this means a Riemannian manifold whose holonomy lies in
$$\Sp(m)\Sp(1) \subset \SO(4m).$$ When \(m = 1\), we further impose that \(M\) be Einstein 
and half-conformally flat. 

On a quaternion-K{\"a}hler manifold, we have the orthogonal decomposition \cite[\textsection 3.3]{oliveirawaldron}
\begin{align*}
    \Lambda^2 T^\ast M &= \Lambda_{\mathfrak{sp}(m)}^2 \oplus \Lambda_{\mathfrak{sp}(1)}^2  \oplus \Lambda_{(\mathfrak{sp}(m) \oplus \mathfrak{sp}(1))^\perp}^2 \\
    &=: \Lambda_-^2 \oplus \Lambda_+^2 \oplus \Lambda_\perp^2.
\end{align*} Hence, given a connection \(A\), the curvature decomposes as
\begin{align*}
    F_A = F_A^- + F_A^+ + F_A^\perp.
\end{align*} In this paper, an \emph{instanton} on a quaternion-K{\"a}hler manifold is a connection satisfying \(F_A = F_A^-\), i.e., whose curvature takes values in \(\Lambda_-^2(\mathfrak{g}_E)\). These are minimizers of the Yang-Mills functional and hence are Yang-Mills themselves. Connections whose curvatures take values either in \(\Lambda_+^2(\mathfrak{g}_E)\) or in \(\Lambda_{\perp}^2(\mathfrak{g}_E)\) are sometimes also referred to as instantons in the literature. \par 

More generally, a \emph{pseudo-holomorphic connection}, i.e., an \(\text{Sp}(m)\text{Sp}(1)\)-compatible connection per Oliveira-Waldron \cite[Def. 2.1]{oliveirawaldron}, is a connection whose curvature takes values in 
$$\Lambda_{-}^2(\mathfrak{g}_E) \oplus \Lambda_{+}^2(\mathfrak{g}_E).$$
We have the following result, which can be seen as a generalization of Theorem \ref{thm:fourspheregap}. Here the gap is measured in the $M^{2,4}$ Morrey norm (defined in (\ref{preliminaries:morreynormdefinition}) below), which is a natural generalization of the conformally invariant \(L^2\) norm in four dimensions.


\begin{thm}\label{thm:quaternionkahlergap}
Let $M$ be a compact quaternion-K\"ahler manifold of dimension \(4m\). There exists a constant $\delta_0 > 0,$ depending only on the geometry of $M,$ as follows.

Suppose that $A_0$ is a pseudo-holomorphic connection on $E \to M,$  with 
\begin{align}\label{quaternionkahler:energylessthandelta}
    \left\|F_{A_0}^+ \right\|_{M^{2, 4}} < \delta_0.
\end{align}
Then the solution \(A(t)\) of (\ref{ymf}) with $A(0) = A_0$ exists for all time. If $M$ has positive scalar curvature, then $A(t)$ converges smoothly as \(t \to \infty\) to an instanton $A_\infty,$ satisfying $F_{A_\infty} = F_{A_\infty}^{-}.$ Moreover, the space of 
instantons is a neighborhood deformation retract (see Definition \ref{defn:nbddefret} below) of the space of pseudo-holomorphic connections.
\end{thm}

\vspace{5mm}

\subsection{Deformation to flat connections on general Riemannian manifolds}
We have the following parabolic gap statement for the Morrey norm of the full curvature. 
\begin{thm}\label{thm:flatgap}
Let 
\(M\) be a compact 
Riemannian manifold of dimension $n \geq 4.$ There exists a constant $\delta_1 > 0,$ depending only on $\rk(E)$ and the geometry of $M,$ as follows. Suppose that \(A_0\) is a connection on \(E \to M\) with
\begin{equation}\label{flatgap:Morreysmall}
\| F_{A_0} \|_{M^{2, 4}} < \delta_1.
\end{equation}
Then the solution $A(t)$ of (\ref{ymf}) with $A(0) = A_0$ exists for all time and converges smoothly as \(t \to \infty\) to a flat connection. 
Moreover, the space of flat connections is a neighborhood deformation retract of the space of all connections.
\end{thm}

We include two proofs of this result, the first based on the parabolic Morrey estimates of the first-named author \cite{estimatespaper} and the second based on Hamilton's monotonicity formula for Yang-Mills flow \cite{hamiltonmonotonicity}. 
Theorem \ref{thm:flatgap} has the following consequence:

\begin{cor}
Let $E \to M$ be a vector bundle over a compact 
Riemannian manifold of dimension $n \geq 4$ and let $\mathcal{A}_E$ denote the space of all connections on $E.$ 
There exists a flat connection on \(E\) if and only if
\begin{align*}
    \inf\left\{\|F_A\|_{M^{2, 4}} \mid A \in \mathcal{A}_E \right\} = 0.
\end{align*}
\end{cor}

The corresponding result for maps between compact Riemannian manifolds is \cite[Corollary 1.4]{estimatespaper}, which 
is related to results of White \cite[Cor. 2 and 3 on pg. 128]{whiteinfima} on the infimum of the $L^{n}$-norm of the derivative on the space of maps. 

By contrast, work of Naito \cite{naito} on the existence of finite-time blowup for (\ref{ymf}) on spheres of dimension $n \geq 5$ implies that there does not in general exist a parabolic \(L^2\) gap. The method of \cite{naito} can also be found in prior work of Chen and Ding in the setting of harmonic map flow \cite[Theorem 1.1]{chending}. We recover the following version of Naito's result for arbitrary compact base manifolds, proved by Wang and Zhang \cite{wangzhangfinitetimeblowupymf}.

\begin{cor}[{\cite[Theorem 1.1]{wangzhangfinitetimeblowupymf}}]\label{cor:Naito}
Let $G$ be a compact Lie group. Supposing that $\pi_1(M)$ has no nontrivial representations in $G,$ there exists a constant \(\delta_2 > 0\), depending only on $G$ 
and the geometry of \(M\), as follows. Given any connection $A_0$ on a topologically nontrivial vector bundle \(E \to M\) with structure group $G,$ if 
\begin{equation}\label{smallym}
\YM(A_0) < \delta \leq \delta_2,
\end{equation}
then the maximal time of smooth existence of the solution of (\ref{ymf}) with $A(0) = A_0$ is less than $\delta_2^{-1}\delta^{\frac{2}{n-4}}$.
\end{cor}


\vspace{5mm}

\subsection{Acknowledgements} A.W. was partially supported by NSF DMS-2004661 during the preparation of this article. A.D. was supported by NSF DMS-2037851 during the preparation of this article.

\vspace{10mm}

\section{Preliminaries}\label{section:Preliminaries}

\subsection{Principal and vector bundles} Let $G$ be a compact Lie group and $\rho : G \to \GL_K(V)$ a faithful representation of dimension $r$ over $K = \R$ or $\C.$ We assume that $V$ is endowed with a (Hermitian) inner product preserved by $G,$ so that $\rho(G) < \O(r)$ if $K = \R$ and $\rho(G) < \U(r)$ if $K = \C.$ Let $\gothg$ denote the Lie algebra of $G$, which we identify with the space of left-invariant vector fields on \(G.\)

Let $M$ be a smooth manifold. Recall that a \emph{principal $G$-bundle over $M$} is a smooth fiber bundle $G \inj P \to M$ with fiber \(G\) and a right $G$-action \(P \times G \to P\) satisfying the following conditions. The action should be smooth, commute with the fiber bundle projection, and act freely and transitively on the fibers. Given a principal $G$-bundle $P \to M$ and a space $F$ on which $G$ acts on the left, we may form the \emph{associated fiber bundle}
$$P \times _G F = \{ P \times F \} / \left( x,y \right) \sim \left( x g, g^{-1} y \right).$$
We form the following associated fiber bundles:
\begin{enumerate}
\item The bundle of \emph{gauge transformations}:
$$\mathcal{G}_P = P \times_{\Ad} G.$$
By abuse of notation, we also refer to its sheaf of smooth sections by $\mathcal{G}_P;$ these correspond to local automorphisms of $P.$

\item The bundle of \emph{infinitesimal gauge transformations} (a.k.a. the \emph{adjoint bundle}):
$$\gothg_P = P \times_{\mathrm{Ad}} \gothg.$$
This is a vector bundle with fiber $\gothg.$

\item The vector bundle
$$E = P \times_\rho V,$$
which has rank $r$ over $K$ and inherits a bundle metric from the inner product on $V.$
\end{enumerate}

The maps $\rho : G \to \GL_K(V)$ and \(d \rho : \mathfrak{g} \to \mathfrak{gl}_K(V)\) yield embeddings of $\mathcal{G}_P$ and $\gothg_P,$ respectively, into $\End E = E \otimes_K E^*$. Hence, an (infinitesimal) gauge transformation of $P$ is equivalent to a section of $\End E$ that lies fiberwise in the image of $\rho(G)$ (resp. $\rho(\gothg)$). 
We endow both $\mathcal{G}_P$ and $\gothg_P$ with the induced metric from $\End E$; in particular, sections of the former are orthogonal (resp. unitary), while sections of the latter are skew-symmetric (resp. skew-Hermitian), for $K = \R$ or $\C,$ respectively.

The principal bundle $P$ can be reconstructed from the associated bundle $E$ by applying the inverse of the map $\rho$ to the transition functions. In fact, any real or complex vector bundle whose transition functions belong to the image of a representation of $G$ arises from the associated bundle construction for some principal \(G\)-bundle. We may therefore define a 
\emph{vector bundle with structure group $(G, \rho)$}, or \emph{$(G,\rho)$-bundle}, in the following two equivalent ways:

\begin{enumerate}

\item The associated bundle $E = P \times_\rho V$ for a principal $G$-bundle $P$ and $r$-dimensional real or complex representation, $V,$ or

\item A vector bundle of rank $r$ over $\R$ or $\C$ and a cover of \(M\) by trivializations for which the transition functions belong to $\rho(G).$

\end{enumerate}
The choice of representation $\rho$ is usually obvious; for example, a rank $r$ real vector bundle ``with structure group $\SO(r)$'' has all transition functions in the image of the standard rank $r$ real representation, while a rank $r$ complex vector bundle ``with structure group $\SU(r)$'' has all transition functions in the image of the standard rank $r$ complex representation.

In this paper, we prefer to work with vector rather than principal bundles, which entails no loss of generality since $G$ is compact. Henceforth, for a $(G,\rho)$-bundle $E,$ we will denote the images of $\mathcal{G}_P$ and $\gothg_P$ inside $\End E$ by $\mathcal{G}_E$ and $\gothg_E,$ respectively.


\vspace{5mm}

\subsection{Connections, curvature, and gauge transformations}

Recall that a connection on a principal \(G\)-bundle \(\pi : P \to M\) is a smooth distribution \(\mathcal{H} \subset TP\) which satisfies the following:
\begin{enumerate}
    \item \(TP \cong \mathcal{H} \oplus \ker{d\pi}\).
    \item For all \(g \in G\) and \(p \in P\), \((R_g)_\ast(\mathcal{H} \cap T_pP) = \mathcal{H} \cap T_{pg} P\).
\end{enumerate}
Given a tangent vector \(X\) at \(x \in M\) and a point \(p \in \pi^{-1}(x)\), there is a unique \(X^H \in \mathcal{H} \cap T_pP\), called the horizontal lift of \(X\), such that \(d\pi(X^H) = X\). \par 
A connection $A$ on \(P\) induces a covariant derivative \(\nabla_A\) on the vector bundle $E = P \times_\rho V \to M$ as follows. Let \(s\) be a smooth section of \(E\). Since \(G\) acts freely and transitively on the fibers of \(P\), \(s\) determines a unique, \(G\)-invariant section \(\tilde{s}\) of \(P \times V \to P\), which may equivalently be viewed as a \(G\)-equivariant function \(P \to V\). Given \(X \in T_xM\), \((d\tilde{s})(X_p^H)\) determines an element of \(\pi^{-1}(x) \times V\), and by \(G\)-equivariance, \((d\tilde{s})(X_p^H)\) lie in the same equivalence class of \(\pi^{-1}(x) \times_\rho V\) as \(p\) varies in \(\pi^{-1}(x)\). Hence we define \(\nabla_A\) by
\begin{align*}
    (\nabla_A s)(X) := [(d\tilde{s})(X^H)].
\end{align*} \par
Conversely, given a \((G, \rho)\)-bundle \(E\), suppose we have a covariant derivative $\nabla_A$ such that, in local frames for which the transition functions belong fiberwise to $\rho(G)$, the connection 1-forms take values in $d\rho(\gothg)$. We obtain a connection on \(P\) as follows. Let \(U\) be a trivialized neighborhood of \(x \in M\), and let \(X \in T_xM\). By abuse of notation, we denote the connection form with respect to the trivialization by \(A.\) Since \(\rho\) is faithful, we may define a horizontal vector at \((x, e) \in U \times G\) by \((X, (d\rho)^{-1}(A(X)) )\). We obtain a horizontal distribution over \(U \times G\) by right-translation, and since the transition functions of \(E\) belong to \(\rho(G)\) and \(\rho\) is faithful, this prescription glues together to give a connection on \(P\). Thus we refer to such covariant derivatives as \((G, \rho)\)-connections, and in the sequel we restrict our attention to \((G, \rho)\)-connections on \(E\) instead of principal connections on \(P\). \par 
The curvature of a connection \(A\) is a section \(F_A\) of \(\Lambda^2T^\ast M \otimes \mathfrak{g}_E\) defined locally by 
\begin{align*}
    F_A = dA + A \wedge A.
\end{align*} Given another \((G, \rho)\)-connection \(\tilde{A}\), whence \(a = \tilde{A} - A\) is a globally well-defined \(\mathfrak{g}_E\)-valued 1-form, we have
\begin{align*}
    F_{\tilde{A}} = F_A + D_Aa + a \wedge a,
\end{align*}
where $D_A$ denotes the covariant exterior derivative.
A gauge transformation \(u \in \Gamma(\mathcal{G}_E)\) acts on connections by the prescription \begin{align*}
    \nabla_{u(A)} s = u(\nabla_A(u^{-1}(s))).
\end{align*} Locally, the connection form transforms by
\begin{align}\label{preliminaries:connectiontransformationlaw}
    A \mapsto u A u^{-1} - du u^{-1},
\end{align} and the curvature transforms by
\begin{align}\label{preliminaries:curvatureisgaugeequivariant}
    F_{u(A)} = uF_Au^{-1}.
\end{align}

\vspace{5mm}

\subsection{The Yang-Mills functional}

Suppose now that $M$ is endowed with a Riemannian metric $g.$ We define the induced inner product on $k$-forms by requiring that
$$|e^1 \wedge \cdots \wedge e^k| = 1$$
for orthonormal covectors $e^1, \ldots, e^k.$ The Hodge star operator $* : \Omega^k \to \Omega^{n-k}$ is an isometry defined by
\[
\LA \alpha, \beta \RA \, dV = \alpha \wedge * \beta
\]
which extends linearly to forms valued in any vector bundle.
Combining this with the norm on $\End E,$ for $\omega, \eta$ (skew-symmetric/skew-Hermitian) sections of $\Omega^k(\gothg_E),$ we have
\begin{align}\label{preliminaries:forminnerproduct}
\LA \omega, \eta \RA \, dV = - \Tr \omega \wedge * \eta,
\end{align}
where $\Tr = \Tr_K.$ The Yang-Mills functional is defined by 
\begin{align}\label{preliminaries:ym}
   \mathcal{YM}(A) & = - \frac12 \int_M \Tr F_A \wedge * F_A \\
& = \frac12 \int_M |F_A|^2 \, dV . \nonumber
\end{align}

\vspace{5mm}

\subsection{The characteristic number}\label{ss:characteristicnumber}

We follow the convention of Donaldson-Kronheimer \cite[p. 42]{donkron} for the characteristic number $\kappa(E) \in \frac14 \Z$ of a real or complex vector bundle \(E\) over a closed, oriented 4-manifold:
\begin{align*}
    \kappa(E) = \begin{cases} \langle c_2(E) - \frac12 c_1(E)^2, [M]\rangle & K = \C \\
    -\frac14 \langle p_1(E), [M]\rangle & K = \R.
    \end{cases}
\end{align*}
The coefficient $-\frac14$ guarantees 
that for an $\text{SU}(2)$-bundle $E,$ 
we have $\kappa(E) = \kappa(\gothg_E)$ 
\cite[(2.1.37)]{donkron}. 
The number \(\kappa(E)\) has the property that it is an integer over the 4-sphere, according to Atiyah-Hitchin-Singer \cite[pg. 450]{ahs}.

According to Chern-Weil theory \cite[Appendix C]{milnorstasheff}, for any 
connection $A$ on $E,$ we have
\begin{align*}
    \kappa(E) = \frac{1}{8\pi^2 n_{K}} \int_M \Tr_K F_A \wedge F_A,
\end{align*}
where
\begin{align}\label{preliminaries:nK}
    n_{K} = \begin{cases} 1 & K = \C \\
    4 & K = \R.
    \end{cases}
\end{align} \par 
In four dimensions, the Hodge star operator \(\ast\) acts as an involution on 2-forms. Hence we have the pointwise orthogonal decomposition \(\Lambda^2T^\ast M = \Lambda_-^2 \oplus \Lambda_+^2\) into negative and positive eigenspaces of \(\ast\), and this decomposition extends to \(\mathfrak{g}_E\)-valued 2-forms. Since \(\ast F_A = F_A^+ - F_A^-\), we deduce from (\ref{preliminaries:forminnerproduct}) and (\ref{preliminaries:ym}) that
\begin{align*} 
    \YM(A) = 4 \pi^2 n_{K} \kappa(E) + \|F^+_A\|_{L^2}^2.
\end{align*} \par

Suppose now that \(M\) is a closed quaternion-K{\"a}hler manifold of dimension \(4m\). Let \(\Omega\) denote the Kraines form on \(M\) \cite{kraines}. Since \(d\Omega = 0\), it follows from Chern-Weil theory that the integral
\begin{align*}
    \int_M \text{Tr}_K(F_A \wedge F_A) \wedge \Omega^{m-1}
\end{align*} is a topological invariant of the bundle. A consequence is that for a pseudo-holomorphic connection \(A\), we have 
\begin{align}\label{preliminaries:qkchernweil}
    \mathcal{YM}(A) = \lambda(E) + c\|F^+\|_{L^2}^2,
\end{align} where \(\lambda(E)\) is a constant depending only on the bundle, and \(c\) is a constant depending only on \(m\) \cite[Lemma 2.3 (ii)]{oliveirawaldron}.
\begin{rmk}\label{rmk:gkscomparison}
Due to the different choices of norms, our convention for $\kappa(E)$ agrees with that of Gursky-Kelleher-Streets \cite{gurskykelleherstreets} in the complex case but differs by half in the real case. On the other hand, our convention for $\|F\|^2$ is half that of \cite{gurskykelleherstreets} in the complex case and agrees in the real case.
\end{rmk}

\vspace{5mm}

\subsection{Sobolev spaces of connections}
 Fix a smooth reference \((G, \rho)\)-connection \(\nabla_{\text{ref}}\) on $E.$ Let \(\nabla_{\text{ref}}\) also denote the induced connection on \(\mathfrak{g}_E\) as well as any tensor product connection obtained by coupling the Levi-Civita connection on a tensor bundle over \(M\) with \(\nabla_{\text{ref}}\) on \(E\) or \(\mathfrak{g}_E\). For \(p \in [1, \infty)\), we define the \(W^{k, p}\) Sobolev norm of a section of \(E, \mathfrak{g}_E,\) or any natural tensor product bundle, by
\begin{align}
    \|s\|_{W^{k, p}} := \left(\sum_{j = 0}^k \int_M |\nabla_{\text{ref}}^{(k)} s|^p \, dV \right)^{\frac{1}{p}}.
\end{align} When \(p = 2\), we abbreviate
\begin{align*}
    \|s\|_{H^k} &:= \|s\|_{W^{k, 2}} \\
    \|s\| &:= \|s\|_{L^2}.
\end{align*}

We next define the Morrey norm on a Riemannian manifold $M$ of dimension \(n\).
        \begin{defn}\label{preliminaries:morreynormdefinition}
		 Given \(q \in [1, \infty]\) and \(\lambda \in [0, n]\), the $M^{q, \lambda}$ \emph{Morrey norm} of a function \(f\) is given by
    \begin{align*}
        \|f\|_{q, \lambda} := 
        \begin{cases}
            \sup_{x \in M, 0 < R \leq 1} \left(R^{\lambda-n} \int_{B_R(x)} |f|^q \, dV\right)^{\frac{1}{q}} & q < \infty \\
            \|f\|_\infty & q = \infty.
        \end{cases}
    \end{align*} The Morrey norm of a section \(s\) of a bundle with fiber metric is just the Morrey norm of the function \(|s|\).
	\end{defn}  \par

To define the Sobolev topology on the space of connections $\calA_E,$ 
recall that the covariant derivative associated to $A$ can be written uniquely as
\[
\nabla_A = \nabla_{\text{ref}} + A
\]
for a \(\mathfrak{g}_E\)-valued 1-form $A$ on \(M\); abusing notation, we denote this 1-form by the same letter as the connection itself. The \(W^{k, p}\) norm of $A$ is simply defined as that of the corresponding 1-form. We may also define the space of Sobolev connections \(W^{k, p}(\calA_E)\) as the completion of the space of smooth connections $\calA_E$ with respect to the \(W^{k, p}\) norm. By compactness of \(M\), the ``Sobolev topology'' on this space is independent of the choice of a smooth reference connection. 
 \begin{rmk}\label{rmk:Hqcontrolsmorrey}
      For \(k \geq \frac{n}{2}-1\), Sobolev embedding implies \(H^{k-1} \inj L^p\) for some \(p \geq \frac{n}{2}\); we may also simultaneously replace both \(\geq\) by \(>\). On a compact manifold, H{\"o}lder's inequality implies that for such \(p\), the $M^{2, 4}$ Morrey norm is bounded above by a constant times the $L^p$ norm.
      Since the assignment \(A \mapsto F_A\) is a continuous map from \(H^k \to H^{k-1}\) for \(k \geq \frac{n}{2}-1\), it follows that the \(M^{2, 4}\) norm of the curvature is a continuous, gauge-invariant function on \(H^k(\mathcal{A}_E)\).
  \end{rmk}

The transformation law (\ref{preliminaries:connectiontransformationlaw}) implies that 
for $k > \frac{n}{2} - 1,$ the action of the gauge group $\mathcal{G}_E$ on $\calA_E$ extends continuously to the spaces of Sobolev connections \(H^k(\mathcal{A}_E)\) and gauge transformations \(H^{k+1}(\mathcal{G}_E)\). In fact, this is a smooth action by a Hilbert Lie group on a Hilbert manifold; see \cite[\textsection 6 and Appendix A]{freeduhlenbeck} or 
\cite[\textsection 4.2.1]{donkron}. \par 
Fix a point \(x_0 \in M\). Let $(\mathcal{G}_E)_{x_0}$ denote the fiber of the bundle of gauge transformations over $x_0.$ We define a framed connection to be an element of
\begin{align*}
    \tilde{\mathcal{A}}_E := \mathcal{A}_E \times (\mathcal{G}_E)_{x_0}.
\end{align*}The gauge group acts on \(\tilde{\mathcal{A}}_E\) by \(\sigma\cdot(A, u) = (\sigma(A), \sigma(x_0)u)\). Since \(H^{k+1}\) gauge transformations are continuous for \(k+1 > \frac{n}{2}\), we have a well-defined, smooth 
action of \(H^{k+1}(\mathcal{G}_E)\) on the \(H^k\) framed connections \(H^k(\tilde{\mathcal{A}}_E)\). Fixing $k > \frac{n}{2} - 1,$ we define the moduli space of framed connections by
$$\tilde{\mathcal{B}}_E:=H^k(\tilde{\mathcal{A}}_E)/H^{k+1}(\mathcal{G}_E).$$
This is a Banach manifold independent of \(k\), in the sense that for $\ell \geq k,$ the induced map \(H^\ell(\tilde{\mathcal{A}}_E)/H^{\ell+1}(\mathcal{G}_E) \to H^k(\tilde{\mathcal{A}}_E)/H^{k+1}(\mathcal{G}_E)\) 
is a weak homotopy equivalence \cite[\S 5.1.1]{donkron}. We also let $\calB_E = \tilde{\calB}_E / G$ denote the unframed moduli space of connections.

\vspace{5mm}

\subsection{Weitzenb{\"o}ck formulae} Recall that the adjoint \(D_A^\ast\) of the exterior covariant derivative, acting on \(\mathfrak{g}_E\)-valued \(k\)-forms, is given by
\begin{align*}
    D_A^\ast = (-1)^{n(k+1)+1} \ast D_A \ast.
\end{align*} We define the Hodge Laplacian \(\Delta_A\) by 
\begin{align*}
    \Delta_A = D_A^\ast D_A + D_A D_A^\ast.
\end{align*} The Hodge Laplacian and the connection Laplacian \(\nabla_A^\ast \nabla_A\) differ by algebraic terms involving the curvatures of \(A\) and \(g\). Let $R_{ijkl}, R_{ij},$ and $R$ respectively denote the Riemann, Ricci, and scalar curvature of $g.$ On \(\mathfrak{g}_E\)-valued 2-forms, we have from \cite[(4.6)]{oliveirawaldron} that
\begin{align*}
    (\Delta_A \omega)_{ij} = (\nabla_A^\ast \nabla_A \omega)_{ij} - [F_i^{\ k}, \omega_{kj}] + [F_j^{\ k}, \omega_{ki}] + R_{i \ j}^{\ k \ l}\omega_{lk} - R_{j \ i}^{\ k \ l}\omega_{lk} + R_i^{\ l}\omega_{lj} + R_j^{\ l}\omega_{il},
\end{align*}
 where we use the Einstein summation convention of summing over repeated indices. By the algebraic Bianchi identity,
 \begin{align*}
    R_{ikjl} = - R_{lijk} - R_{klji} = R_{jkil} + R_{ijkl},
\end{align*}
 so the Weitzenb{\"o}ck identity takes the form
\begin{align*}
    (\Delta_A \omega)_{ij} = (\nabla_A^\ast \nabla_A \omega)_{ij}
    - [F_i^{\ k}, \omega_{kj}] + [F_j^{\ k}, \omega_{ki}] - R_{ij}^{\ \ kl}\omega_{kl} + R_i^{\ l}\omega_{lj} + R_j^{\ l}\omega_{il}.
\end{align*} \par
We now specialize the general Weitzenb{\"o}ck formula to the quaternion-K{\"a}hler setting. On a quaternion-K{\"a}hler manifold of dimension \(4m\), Salamon \cite[Theorem 3.1]{salamonqkinventiones} showed that the Riemannian curvature tensor can be decomposed into a scalar multiple of the curvature tensor of \(\mathbb{HP}^m\) plus the curvature tensor of a hyperk{\"a}hler manifold. To be precise, let \({}^1 I, {}^2 I, {}^3 I\) be a local frame of 
\(\Lambda_+^2 \subset \Lambda^2T^\ast M\) satisfying the quaternion relations. Define the tensor
\begin{align*}
    R_{ijkl}^{\mathbb{HP}^m} := g_{il}g_{jk} -  g_{jl}g_{ik} + \sum_{\alpha = 1}^3 {}^\alpha I_{il}{}^\alpha I_{jk} - {}^\alpha I_{jl}{}^\alpha I_{ik} - 2 \ {}^\alpha I_{ij}{}^\alpha I_{kl}.
\end{align*} This represents the curvature tensor of \(\mathbb{HP}^m\) with scalar curvature \(16m(m+2)\). Thus by Salamon, we have
\begin{align*}
    R_{ijkl} = \frac{R}{16m(m+2)} R_{ijkl}^{\mathbb{HP}^m} + R_{ijkl}^{\text{hyp}},
\end{align*} where \(R_{ijkl}^{\text{hyp}}\) is a section of $\text{Sym}^2\left(\Lambda_-^2\right).$ 
Next we consider the action of \(R_{ijkl}^{\mathbb{HP}^m}\) on a form \(\omega \in \Omega_+^2.\) 
We compute
\begin{align*}
    R_{ij}^{\ \ kl} \omega_{kl} = \frac{R}{16m(m+2)}(-2\omega_{ij} + (-2 + 8m)\omega_{ij} + 4\omega_{ij}) = \frac{R}{2(m+2)} \omega_{ij}.
\end{align*} Therefore the Riemannian zeroth-order term in the Weitzenbock formula for \(\Omega_+^2(\mathfrak{g}_E)\) is
\begin{align*}
    \left(\frac{R}{4m} + \frac{R}{4m} -\frac{R}{2(m+2)}\right) \operatorname{id} = \frac{R}{m(m+2)}\operatorname{id}. 
\end{align*} In particular, suppose \(A\) is a pseudo-holomorphic connection. Since \([\mathfrak{sp}(m), \mathfrak{sp}(1)] = 0\) in \(\Lambda^2T^\ast M\), we have for \(\omega \in \Omega_+^2(\mathfrak{g}_E)\) that
\begin{align}\label{preliminaries:qkweitzenbock}
    (\Delta_A \omega)_{ij} = (\nabla_A^\ast \nabla_A \omega)_{ij} - [(F_A^+)_i^{\ k}, \omega_{kj}] + [(F_A^+)_j^{\ k}, \omega_{ki}] + \frac{R}{m(m+2)}\omega.
\end{align}

\vspace{5mm}

\subsection{Yang-Mills flow}
Given an initial connection $A_0,$ we wish to produce a solution $A(t)$ of (\ref{ymf}) on a maximal time interval. The gauge-invariance of the Yang-Mills functional \(\mathcal{YM}\) implies that the flow equation (\ref{ymf}), which is the downward gradient flow of \(\mathcal{YM}\), is not strictly parabolic. Nonetheless, we may follow the strategy of R{\aa}de \cite[Theorem 1]{rade}, based on DeTurck's first simplification\footnote{It is also possible to obtain well-posedness of the flow via the better-known Donaldson-DeTurck ``trick,'' see e.g. \cite{kozonomaedanaito4dglobal, ohtatarucaloric, struweym}, but obtaining the optimal result is no more straightforward than in R\aa de's setup.} \cite{deturckfirsttrick} of Hamilton's existence result for Ricci flow, to obtain a unique, short-time solution of (\ref{ymf}), as well as continuity of the solution with respect to the initial data in \(H^k(\mathcal{A}_E)\) for \(k > \lceil\frac{n}{2}-1\rceil\). The main well-posedness results are detailed in Theorem \ref{thm:wellposedness}. \par 
The key observation underpinning R{\aa}de's strategy is that if \(A(t)\) evolves by (\ref{ymf}), then the curvature evolves by the parabolic equation
\begin{align}\label{preliminaries:evolutionofcurvature}
    \frac{\partial F_A}{\partial t} + \Delta_A F_A = 0.
\end{align} Thus we obtain a solution of (\ref{ymf}) by solving for \((a(t), \Omega(t)) \in \Omega^1(\mathfrak{g}_E) \oplus \Omega^2(\mathfrak{g}_E)\) satisfying
\begin{align*}
    \frac{\partial a}{\partial t} &= -D_{\nabla_{\text{ref}} + a}^\ast \Omega \\
    \frac{\partial \Omega}{\partial t} &= -\Delta_{\nabla_{\text{ref}} + a} \Omega,
\end{align*} and then showing that if \(\Omega(0) = F_{\nabla_{\text{ref}} + a(0)}\), then \(\Omega(t) = F_{\nabla_{\text{ref}} + a(t)}\). This is explained in the short-time existence/uniqueness result, Theorem \ref{thm:shorttimeexistence}. \par
To prove the (neighborhood) deformation retraction statements of Theorems \ref{thm:fourspheregap}, \ref{thm:quaternionkahlergap}, and \ref{thm:flatgap}, we must \textit{a priori} work with Sobolev solutions of (\ref{ymf}), i.e., where \(A(t)\) is in \(H^k\). However, as we show in the appendix, any \(H^k\) solution is gauge-equivalent by a constant-in-time \(H^{k+1}\) gauge transformation to a smooth solution. Thus, in most instances we need only consider smooth solutions of (\ref{ymf}). \par 
Last, we note that the (neighborhood) deformation retractions of the previous three theorems are continuous on spaces of connections and not just on gauge-equivalence classes. To obtain convergence of solutions of (\ref{ymf}) and continuity of the infinite-time limit in \(H^k\) without modding out by gauge, we prove in the appendix a Sobolev distance estimate (Theorem \ref{thm:Hkestimates}) for solutions of (\ref{ymf}) with bounded curvature, as well as an \(H^{-1}\) {\L}ojasiewicz-Simon inequality in Lemma \ref{lemma:lojasiewicz}, cf. \cite[(9.1)]{rade}.

\vspace{10mm}

\section{Deformation to instantons on the 4-sphere}

\subsection{Optimal split $\eps$-regularity} We first give an optimal version of the split $\eps$-regularity estimate relied upon in \cite{instantons}.

\begin{defn}\label{defn:delta} Let $\calB_{S^4,G}$ denote the space of all gauge-equivalence classes of smooth connections carried by principal $G$-bundles $P \to S^4,$ including all topologies. 
Given a representation $\rho: G \to \GL_K(r),$ where $K = \R \text{ or } \C,$ define
\begin{equation*}
\delta_{G, \rho} = \inf \left\{ \frac{1}{4 \pi^2 n_{K}} \left\| F^+_{B_\rho} \right\|^2 \mid \LB B \RB \in \calB_{S^4,G}, D^*F_B = 0, F^+_B \not\equiv 0 \right\}.
\end{equation*}
Here, $B_\rho$ is the connection induced on the vector bundle $E = P \times_\rho V$ by the representation $\rho.$
\end{defn}
\begin{rmk}\label{rmk:deltaGrho} We make several comments on the foregoing definition.
\begin{enumerate}[itemsep=2mm]

\item Evidently, $\delta_{G, \rho}$ is the largest number such that the implication
\begin{equation*}
\begin{split}
D_A^* F_A = 0, \quad \|F^+_A\|^2 < 4 \pi^2 n_K \delta_{G, \rho} 
\quad \Rightarrow \quad F^+_A = 0
\end{split}
\end{equation*}
holds generally for connections $A$ on $(G, \rho)$-bundles $E \to S^4.$ In particular, if $\kappa(E) \geq 0,$ any Yang-Mills connection with
$ \YM(A) < 4 \pi^2 n_{K} \left( \kappa(E) + \delta_{G, \rho} \right) $ is anti-self-dual.

\item We have $\delta_{G, \rho} \geq \delta_{\SO(r), \rho_{std}}$ for $K = \R$ and $\delta_{G, \rho} \geq \delta_{\SU(r), \rho_{std}}$ for $K = \C.$ 
By the Min-Oo gap theorem, 
these are bounded below by a positive universal constant.

\item $\delta_{G, \rho}$ determines $\delta_{G, \rho'}$ for all other representations $\rho'$ by an algebraic constant; the difference comes only from the norm on $\rho(\gothg) \subset V \otimes V^*.$

\item
For $G = \SU(r)$ and the standard representation, according to \cite[Corollary 1.2]{gurskykelleherstreets} (see Remark \ref{rmk:gkscomparison}), we have
$$\delta_{\SU(r), \rho_{std}} = 2.$$
This corresponds to the general case $\gamma_1 = \frac{4}{\sqrt{3}}$ in the notation of \cite{gurskykelleherstreets}.

\item For $G = \SO(r)$ and the standard representation, we also have
$\delta_{\SO(r), \rho_{std}} = 1$ for $r \geq 4$ and $\delta_{\SO(3), \rho_{std}} = 2.$

\end{enumerate}
\end{rmk}

\begin{thm}\label{thm:strongepsreg}
Given \(\gamma > 0\), there exist \(k \in \N\) and a constant $C_{(\ref{thm:strongepsreg})} > 0$, which depends on \(\gamma \) and the structure group \((G, \rho)\), as follows. 

Let $B = B_1(0)$ be a geodesic ball for a smooth metric $g$ which is sufficiently close in $C^k$ to the Euclidean metric in normal coordinates. 
Suppose that \(A(t)\) is a smooth solution of (\ref{ymf}) on a $(G, \rho)$-bundle $E \to B$ for $t \in \left( -1, 0 \right)$, with 
    \begin{align}\label{strongepsreg:energybound}
        \sup_{-1 < t < 0} \YM(A(t)) \leq \gamma^{-1}
    \end{align}
    and
    \begin{align}\label{strongepsreg:GKSbound}
        \sup_{-1 < t < 0} \|F_{A(t)}^+\|_{L^2(B)}^2 \leq \delta \leq 4 \pi^2 n_{K} \delta_{G, \rho} - \gamma,
    \end{align}
    where self-duality and norms are defined by the metric $g.$
    Then 
    \begin{equation}\label{epsilonregularitybound}
        \sup_{B_{1/2} \times \LB -1/4, 0 \right)} | F_{A(t)}^+(x,t) | \leq C_{(\ref{thm:strongepsreg}) } \sqrt{\delta}.
    \end{equation}
\end{thm}

The proof relies on the following two lemmas. In the sequel, we denote $F^+(t) := F_{A(t)}^+.$

\begin{lemma}\label{lemma:energynotsmall}
    There exists a constant $0 < c < \frac{1}{10}$, depending only on the structure group \((G, \rho)\), as follows. Given a solution $A(t)$ satisfying the assumptions of Theorem \ref{thm:strongepsreg}, if $0 < r \leq 1$ is such that
    \begin{equation}\label{energynotsmall:assn}
       \int_{B_{r}} \left|F^+\left( -(cr)^2 \right) \right|^2 \leq c^2,
    \end{equation}
    then
    \begin{equation}\label{energynotsmall:est}
        \sup_{ B_{r/2} \times \LB - \frac12 (cr)^2, 0 \right) } |F^+| \leq \frac{1}{cr^2}.
    \end{equation}
\end{lemma}
\begin{proof} 
By parabolic rescaling, we may assume without loss of generality that $r = 1.$ 
The assumption (\ref{energynotsmall:assn}) reads
\begin{equation}
    \int_{B} |F^+(-c^2)|^2 \leq c^2.
\end{equation} In the sequel, \(c\) may decrease in each appearance but will depend only on the structure group \((G, \rho)\). We also let $C > 0$ denote a constant depending only on $(G, \rho)$ which may increase in each appearance.
Recall that the curvature evolves by
\begin{align*}
    \partial_t F + D D^\ast F = 0.
\end{align*} Letting \(\varphi \in C_c^\infty(B)\), we take the inner product with \(\varphi^2 F^+\), integrate in space, and then integrate by parts. This yields
\begin{align*}
    \frac{1}{2}\frac{d}{dt} \int_B \varphi^2 |F^+|^2 &= \int_B \langle \varphi^2 F^+, \partial_t (F^+)\rangle = \int_B \langle \varphi^2 F^+, (\partial_t F)^+\rangle = \int_B \langle \varphi^2 F^+, \partial_t F\rangle \\
    &= -\int_B \langle \varphi^2 F^+, D D^\ast F\rangle \\
    &= -\int_B \varphi^2\langle D^\ast F^+, D^\ast F\rangle + 2\int_B \langle \nabla \varphi \intprod F^+,  \varphi D^\ast F\rangle.
\end{align*} Recall that \(D^\ast F = 2D^\ast F^+\). Hence we obtain
\begin{align*}
    \frac{d}{dt}\int_B \varphi^2 |F^+|^2 + \int_B \varphi^2 |D^\ast F|^2 &= 4\int_B  \langle  \varphi D^\ast F, \nabla \varphi \intprod F^+\rangle.
\end{align*}
The Cauchy-Schwarz and Young's inequalities imply
\begin{align*}
    \left|4\int_B \varphi \langle D^\ast F^+, (\nabla \varphi) \intprod F^+\rangle\right| \leq \frac{1}{2}\int_B \varphi^2 |D^\ast F|^2 + 8\int_B |\nabla \varphi|^2 |F^+|^2.
\end{align*} Integrating in time and rearranging yields 
\begin{align}\label{localenergyinequality}
    \int_B \varphi^2 |F^+(b)|^2 + \frac{1}{2}\int_a^b\int_B \varphi^2 |D^\ast F|^2 \leq \int_B \varphi^2 |F^+(a)|^2 + 8\int_a^b\int_B |\nabla \varphi|^2|F^+|^2.
    \end{align}
    Let \(\varphi\) be a cutoff for \(B_{3/4}.\) Since the metric is assumed to be close to Euclidean, we have $\|\nabla \varphi\|_\infty \leq C.$ Letting \(a = -c^2\) and \(b = 0\) in (\ref{localenergyinequality}), we obtain
    \begin{align}\label{localenergybound}
    \sup_{-c^2 \leq t \leq 0} \int_{B_{3/4}} |F^+(t)|^2 \leq c^2 + C \delta_{G, \rho} c^2 \leq C c^2.
    \end{align}
    For $c$ sufficiently small, it follows from standard $\eps$-regularity (Prop. 2.3 of \cite{instantons}) that 
    $$\sup_{B_{1/2} \times \LB -\frac12 c^2, 0 \right) }|F^+| \leq \frac{C c^2}{c^2} = C \leq \frac{1}{c}.$$
    Undoing the parabolic rescaling yields (\ref{energynotsmall:est}).
\end{proof}

We will use the following well-known point selection lemma. Denote 
\begin{align*}
    P_r(x, t) = B_r(x) \times [t - r^2, t],
\end{align*} where \(B_r(x)\) is a geodesic ball of radius \(r\) centered at \(x\). For the following lemma, we do not make any assumptions about the metric on \(B_r(x)\).

\begin{lemma}\label{lemma:comparablesup}
    Let \(\lambda > 0\) and let \(f\) be a non-negative continuous function on \(P_r(x_0, t_0)\) such that \(f(x_0, t_0) = \lambda^{-2} > 16r^{-2}\). Then there exists a spacetime point \((\tilde{x}, \tilde{t})\) such that, denoting \(\mu^{-2} := f(\tilde{x}, \tilde{t})\), we have \(\mu \leq \lambda\), \((\tilde{x}, \tilde{t}) \in P_{2\lambda}(x_0, t_0)\), and     \begin{align} \label{comparablesupbound}
\sup_{P_\mu(\tilde{x}, \tilde{t})} f \leq 4\mu^{-2}.
    \end{align}
\end{lemma}
\begin{proof}
    Suppose for 
    contradiction that such a spacetime point does not exist. Then in particular the desired conclusion does not hold for \((x_0, t_0)\), so we have
    \begin{align*}
        \sup_{P_\lambda(x_0, t_0)} f > 4\lambda^{-2}.
    \end{align*} Assume by induction that we have found \(n+1\) spacetime points \((x_i, t_i)\), \(i = 0, \ldots, n\) satisfying
    the following conditions:
    \begin{align*}
       (\ast) \begin{cases}
            (x_{i+1}, t_{i+1}) \in \overline{P_{\mu_i}(x_i, t_i)} \ \ \ \text{ where } \mu_i^{-2} := f(x_i, t_i), \\
            f(x_{i + 1}, t_{i+1}) = \sup_{P_{\mu_i}(x_i, t_i)} f > 4\mu_i^{-2}.
        \end{cases}
    \end{align*} Note that \(\mu_{i+1} < 2^{-1}\mu_i\), and hence \(\mu_n < 2^{-n}\lambda\). We let \(d\) denote the parabolic distance \(d((p, t), (q, s)) = \max\{d_g(p, q), \sqrt{|t-s|}\}\), for which parabolic cylinders are metric balls. Observe that 
    \begin{align*}
        d((x_0, t_0), (x_n, t_n)) \leq \sum_{i = 0}^n d((x_{i+1}, t_{i+1}), (x_i, t_i)) \leq \sum_{i = 0}^n \mu_i < \lambda \sum_{i = 0}^n 2^{-i} < 2\lambda.
    \end{align*} Thus for \(0 \leq i \leq n-1\), \((x_i, t_i)\) is in \(P_{2\lambda}(x_0, t_0)\). Additionally, \(P_{\mu_i}(x_i, t_i)\) is compactly contained in \(P_r(x_0, t_0)\) since \(\lambda < r/4\). Our assumption that there is no spacetime point satisfying the conclusion of the lemma implies the existence of a spacetime point \((x_{n+1}, t_{n+1})\) satisfying \((\ast)\). Therefore, by induction we obtain an infinite sequence of spacetime points \(\{(x_i, t_i)\}\) satisfying \((\ast)\). Since this sequence is contained in the compact set \(\overline{P_{2\lambda}(x_0, t_0)}\) and since \(f\) is continuous, and yet \(\mu_i^{-2}\) tends to infinity, we reach a contradiction, as desired.
\end{proof}

\begin{proof}[Proof of Theorem \ref{thm:strongepsreg}]
We first derive the desired estimate without the $\delta$ on the right-hand side.

Suppose for contradiction that there exists a sequence of solutions $A_i$ on the unit parabolic cylinder (with metrics tending to Euclidean in $C^\infty$) such that
$$\sup_{B_{1/2} \times \LB -\frac14, 0 \RB} |F^+_{A_i}(x,t) | \geq c^{-4i-1},$$
where $c$ is the same constant from Lemma \ref{lemma:energynotsmall}. We abbreviate $F_i := F_{A_i}.$ After recentering and rescaling by a factor no less than $\frac12,$ we may assume
$$|F^+_i(0,0)| \geq c^{-4i-1}.$$
Let
$$t_j = -c^{4j}$$
and
$$I_j = \LB t_{j-1}, t_{j} \RB.$$
In view of (\ref{strongepsreg:GKSbound}) and the convergence of the metrics, (\ref{localenergyinequality}) implies that we can assume
$$\int_{P_1(0, 0)} |D_i^*F_i|^2 \, dV dt \leq C \delta.$$
Since the time-intervals $I_j$ are essentially disjoint, for each $i$ there exists $j_i \in \{\lceil \frac{i}{2} \rceil, \ldots, i\},$ such that
\begin{align}\label{tensiongoingtozero}
    \int_{B_1(0) \times I_{j_i}} |D_i^*F_{i}|^2 \, dV dt \leq \frac{2C \delta}{i}.
\end{align} Denote
\begin{align*}
    \tau_i := t_{j_i}.
\end{align*} By the contrapositive of Lemma \ref{lemma:energynotsmall}, we must 
have $$\int_{B_{c^{-1}\sqrt{-\tau_i}}(0)} |F^+_i(\tau_i)|^2 > c^2.$$
In particular, there exist 
points \(\tilde{x}_{i}\) in \(B_{c^{-1}\sqrt{-\tau_i}}(0)\) such that
\begin{equation}
\lambda_i^{-2} := |F^+_i(\tilde{x}_{i}, \tau_i)| > \frac{c^{3}}{C(-\tau_i)}.
\end{equation}
Denote $r_i^2 := (c^{-4}-1)(-\tau_i).$
Since for \(c\) sufficiently small (but still depending only on the structure group \((G, \rho)\))
\begin{align*}
    r_i^2\lambda_{i}^{-2} > r_i^2\frac{c^{3}}{C(-\tau_i)} = -\tau_i\left(\frac{1}{c^4}-1\right)\frac{c^{3}}{C(-\tau_i)} > 16,
\end{align*}
it follows that \(|F_i^+|\) satisfies the hypotheses of Lemma \ref{lemma:comparablesup} on $P_{r_i}(\tilde{x}_i, \tau_i).$ 
Thus there exists a spacetime point $(x_i, s_i) \in P_{2\lambda_i}(\tilde{x}_i, \tau_i)$
such that, denoting \(\mu_i^{-2} := |F_i^+(x_i, s_i)|\),
\begin{align}\label{F+comparablesup}
    &\sup_{P_{\mu_i}(x_i, s_i)} |F_i^+| \leq 4\mu_i^{-2}.
\end{align} Moreover, we have \(\mu_i \leq \lambda_{i} \leq \frac{r_i}{4}\), so
\begin{align*}
    P_{\mu_i}(x_i, s_i) \subset P_{\mu_i + 2\lambda_{i}}(\tilde{x}_i, \tau_i) \subset P_{r_i}(\tilde{x}_i, \tau_i). 
\end{align*} In particular, we must have
\begin{align*}
    [s_i - \mu_i^2, s_i] \subset [\tau_i - r_i^2, \tau_i] = I_{j_i}.
\end{align*}
By construction, $j_i \geq \lceil \frac{i}{2}\rceil.$ Thus
\begin{align*}
    \sqrt{-\tau_i} \leq \sqrt{c^{4\frac{i}{2}}} = c^i
\end{align*} and consequently
\begin{align*}
    \lambda_i \leq \frac{r_i}{4} = \frac{1}{4}\sqrt{(c^{-4}-1)(-\tau_i) } \leq c^{i-2}.
\end{align*} Therefore
\begin{align*}
    d_g(x_i, 0) \leq d_g(x_i, \tilde{x}_i) + d_g(\tilde{x}_i, 0) \leq 2\lambda_i + c^{-1}\sqrt{-\tau_i} 
    \leq c^{i-3} \to 0
\end{align*}
as $i \to \infty.$
Hence, for \(i\) large enough, we deduce that
\begin{align}\label{containedwhereD*Fsmall}
    P_{\mu_i}(x_i, s_i) \subset B_{1/2}(x_i) \times [s_i-\mu_i^2, s_i] \subset B_1(0) \times I_{j_i}.
\end{align}\par
 We now parabolically rescale \(B_{1/2}(x_i) \times [s_i-\mu_i^2, s_i]\) about \((x_i, s_i)\) by the factor \(\mu_i\). This yields a sequence of solutions of (\ref{ymf}), denoted by \(\tilde{A}_i\), on \(B_{(2\mu_i)^{-1}}(0) \times [-1, 0]\), where the metric on \(B_{(2\mu_i)^{-1}}(0)\) is approaching the Euclidean metric in \(C_\text{loc}^\infty\). In view of the bounds (\ref{strongepsreg:energybound}) and (\ref{tensiongoingtozero}) and the containment (\ref{containedwhereD*Fsmall}), we may apply standard Uhlenbeck compactness theory (see e.g. \cite[Theorem 1.3 and \textsection 5]{waldronuhlenbeck}) as follows. After passing to a subsequence of the \(\tilde{A}_i\), there exists a finite collection of points \(y_k\) in \(\R^4\), a finite-energy smooth connection \(A_\infty\) on a \((G, \rho)\)-bundle over \(\R^4 \setminus \{y_k\}\),  and bundle maps \(u_i\) defined on an exhaustion of \(\R^4 \setminus \{y_k\}\), such that \(u_i(\tilde{A}_i)\) converges to a Yang-Mills connection \(A_\infty\) in \(C_{\text{loc}}^\infty(\R^4\setminus \{y_k\} \times (-1, 0))\). 
 
Writing \(\tilde{F}_i^+ = F_{u_i(\tilde{A}_i)}^+\),
    note that \begin{align}\label{capturedenergy}
    |\tilde{F}_i^+(0, 0)| = 1,
\end{align}
so in particular by Lemma \ref{lemma:energynotsmall}, we have
\begin{equation*}
\int_{B_{1/2}} \left| \tilde{F}_i^+ \left( -\frac12 \right) \right|^2 \geq c^2.
\end{equation*}
Meanwhile, (\ref{F+comparablesup}) implies
\begin{align}\label{blowupbounded}
    \sup_{P_1(0, 0)} |\tilde{F}_i^+| \leq 4.
\end{align}
Therefore our convergence is strong in $L^2$ on $B_1(0),$ and we still have
\begin{equation}\label{strongepsreg:F+Ainftynonzero}
\int_{B_{1/2}} |F_{A_\infty}^+ |^2 \geq c^2
\end{equation}
in the limit. \par
 On the other hand, by Uhlenbeck's removable singularity theorem, we may extend the bundle in question over the point at infinity to obtain a smooth Yang-Mills connection \((\tilde{E}, A_\infty)\) on the 4-sphere. It follows from (\ref{strongepsreg:GKSbound}), conformal invariance of the energy under the rescaling, Fatou's lemma, and (\ref{strongepsreg:F+Ainftynonzero}) that
 $$c^2 \leq \|F_{A_\infty}^+\|^2 \leq 4\pi^2 n_{K} \delta_{G, \rho} - \gamma.$$
 But then the implication of Remark \ref{rmk:deltaGrho}(1) gives $F^+_{A_\infty} \equiv 0,$ which is a contradiction. 
 
 We have established the desired bound (\ref{epsilonregularitybound}) without \(\delta\) on the right-hand side.
    To obtain the bound with \(\delta\), we proceed as follows. Here \(C_{(\ref{thm:strongepsreg})}\) may increase from line to line but will keep the same dependencies. The bound without \(\delta\) implies
    \begin{align*}
        \sup_{P_{3/4}(0, 0)} |F^+| \leq C_{(\ref{thm:strongepsreg})}.
    \end{align*} Hence the self-dual curvature satisfies the differential inequality
    \begin{align*}
        (\partial_t - \Delta) |F^+| \leq C_{(\ref{thm:strongepsreg})}|F^+|.
    \end{align*} Parabolic Moser iteration (see \cite[Lemma 19.1]{ligeomanalysis}) then yields the bound
    \begin{align*}
        \sup_{P_{1/2}(0, 0)} |F^+| \leq C_{(\ref{thm:strongepsreg})}\sqrt{\delta},
    \end{align*} as desired.
\end{proof}

\vspace{5mm}

\subsection{Proof of Theorem \ref{thm:fourspheregap}} 
We will use the following result, which is a version of the central estimate of \cite[Theorem 2.5]{instantons}.

\begin{thm}[Cor. 7.3 of \cite{oliveirawaldron}]\label{thm:F+boundimpliesFbound}
    Let \(M\) be a closed Riemannian \(4\)-manifold, and suppose \(A(t)\) is a solution of (\ref{ymf}) on \([0, T)\). Let \(0 \leq t_1 \leq t_2 < T\), and set
    \begin{align*}
       E_0 &:= \|F(0)\|^2 \\
       K &:= \int_{t_1}^{t_2} \|F^+(t)\|_\infty \, dt.
    \end{align*} There exist positive constants \(C\), which is universal, and  \(R_0\), depending on \(E_0\) and the geometry of \(M\), such that
    \begin{align}\label{F+boundimpliesFbound:bound}
        \|F(t_2)\|_\infty \leq \max\left\{R_0^{-2}, e^{C\max\left\{1, K \right\}}\|F(t_1)\|_\infty \right\}.
    \end{align}
\end{thm}

\begin{thm}[Cf. Theorem \ref{thm:fourspheregap}]\label{thm:fourspheregeneralgap}
Let \(k \geq 2\), and let $E \to S^4$ be a $(G, \rho)$-bundle. Suppose \(A_0\) is an \(H^k\) connection on \(E\) satisfying
\begin{equation}\label{fourspheregeneralgap:energylessthandelta}
\YM(A_0) < 4 \pi^2 n_{K}\left( |\kappa(E)|+ \delta_{G, \rho} \right).
\end{equation}
Here $\kappa(E)$ is defined in \S \ref{ss:characteristicnumber} and $\delta_{G, \rho}$ is given in Definition \ref{defn:delta}.
Then the solution \(A(t)\) of (\ref{ymf}) with \(A(0) = A_0\) exists for all time and converges in \(H^k\) exponentially as \(t \to \infty\) to an instanton on $E$. Moreover, the map $A_0 \mapsto \lim_{t \to \infty} A(t)$ for \(A_0\) satisfying (\ref{fourspheregeneralgap:energylessthandelta}) is homotopic to the identity relative to the 
space of instantons.
\end{thm}
\begin{proof}
Define the set
\begin{align*}
    \calN := \{A_0 \in H^k(\mathcal{A}_E) \mid A_0 \text{ satisfies } (\ref{fourspheregeneralgap:energylessthandelta})\},
\end{align*}which is open in the \(H^k\) topology. Given \(A_0\ \in \calN\), item (\ref{wellposedness:gaugeequivalenttosmooth}) of Theorem \ref{thm:wellposedness} implies that there exists an \(H^{k+1}\) gauge transformation \(u\) such that \(u(A(t))\) is a smooth solution of (\ref{ymf}) for \(t > 0\). Thus if \(u(A(t))\) exists for all time and converges smoothly to an instanton, \(A(t)\) will exist for all time and converge in \(H^k\) to an instanton. Thus to prove the first part of the theorem, we assume that \(A(t)\) is smooth. \par 
Because $\mathcal{YM}(A(t))$ is nonincreasing, the assumption (\ref{fourspheregeneralgap:energylessthandelta}) is preserved. 
There exists a universal \(R_0 > 0\) such that for any point \(x \in S^4\), rescaling \(B_{R_0}(x)\) by the factor \(R_0\) yields a ball satisfying the metric hypothesis of Theorem \ref{thm:strongepsreg}, whence for \(t > 0\)
\begin{align}\label{fourspheregeneralgap:selfdualsupisbounded}
    |F^+(x, t)| \leq \frac{4\pi C_{(\ref{thm:strongepsreg})}\sqrt{\delta_{G, \rho}}}{\min\{1, t\}}.
\end{align} It then follows from (\ref{F+boundimpliesFbound:bound}) and item (\ref{wellposedness:blowupcharacterization}) of Theorem \ref{thm:wellposedness} that \(A(t)\) exists for all time. Moreover, the full curvature blows up at most exponentially at infinite time. We will now argue that the solution in fact converges smoothly. \par 

Applying Uhlenbeck's compactness and removable-singularity theorems as in the proof of Theorem \ref{thm:strongepsreg}, we may extract an Uhlenbeck limit \(A_\infty'\) from the sequence of connections $A(n),$ $n \in \N,$ which is a Yang-Mills connection on a possibly different \((G, \rho)\)-bundle over $S^4.$ By Fatou's lemma and the global energy inequality, we have  
$$\left\|F^+_{A_\infty'}\right\|^2 \leq \|F^+(0)\|^2 < 4\pi^2n_{K}\delta_{G, \, \rho}.$$
It now follows from the implication of item (1) of Remark \ref{rmk:deltaGrho} that $A_\infty'$ must be anti-self-dual:
\begin{align}\label{fourspheregeneralgap:UhlenbecklimitisASD}
    F^+_{A_\infty'} \equiv 0.
\end{align}

However, the Weitzenb{\"o}ck formula implies that anti-self-dual connections on $S^4$ have vanishing self-dual cohomology; it therefore follows directly from \cite[Corollary 3.8]{instantons} that $A(t)$ converges exponentially in the smooth topology. For purposes of exposition, we provide the following simpler argument in the present case.

     In the sequel, \(C\) denotes a positive universal constant which may increase from line to line.
     The bound (\ref{fourspheregeneralgap:selfdualsupisbounded}) implies that the Uhlenbeck convergence \(\left| F^+(n_j) \right| \to \left| F_{A_\infty'}^+ \right| \) is strong in \(L^2\). By (\ref{fourspheregeneralgap:UhlenbecklimitisASD}), \(\|F^+(n_j)\| \to 0\), so \(\|F^+(t)\| \to 0\) since the self-dual energy is decreasing. Now, the round metric on \(S^4\) has vanishing Weyl tensor and positive scalar curvature. Hence the Weitzenb{\"o}ck formula for self-dual forms implies
        \begin{align*}
         \frac{1}{2}\frac{d}{dt} \|F^+\|^2 &= - (\nabla^\ast \nabla  F^+, F^+) + (\llbracket F^+, F^+\rrbracket, F^+) - 4 \|F^+\|^2 \\
        &=  -\| \nabla  F^+\|^2 + \left(\llbracket F^+, F^+\rrbracket, F^+ \right) - 4\|F^+\|^2.
    \end{align*} By Cauchy-Schwarz and \cite[Lemma 2.30]{bourguignonlawson},
    \begin{align*}
        \left( \llbracket F^+, F^+\rrbracket, F^+ \right) \leq C \|F^+\|_4^2\|F^+\|.
    \end{align*}By Kato's inequality and the Sobolev inequality on functions,
    \begin{align*}
        \|F^+\|_4^2 \leq C(\|\nabla F^+\|^2 + \|F^+\|^2),
    \end{align*} where \(C\) is independent of the reference connection (see e.g., \cite[(1.16)]{instantons}). Since \(\|F^+(t)\| \to 0\) as \(t \to \infty\), there is some  time \(\tau_0 > 1\) such that for \(t \geq \tau_0\) we have
    \begin{align}\label{fourspheregeneralgap:selfdualenergysmall}
        \|F^+(t)\| < \min\left\{\frac{1}{4C},  1\right\},
    \end{align}whence
    \begin{align*}
        \frac{d}{dt} \|F^+\|^2 \leq -6\|F^+\|^2.
    \end{align*} By Gronwall's inequality, we have for \(t \geq \tau_0\)
    \begin{align}\label{fourspheregeneralgap:selfdualexponentialdecay}
        \|F^+(t)\| \leq \|F^+(\tau_0)\|e^{-3(t - \tau_0)} \leq e^{-3(t-\tau_0)}.
    \end{align} Then  (\ref{fourspheregeneralgap:selfdualexponentialdecay}) and $\varepsilon$-regularity for \(F^+\) yield
    \begin{align}\label{fourspheregeneralgap:F+exponentialdecay}
        \|F^+(t)\|_\infty \leq Ce^{-3(t-\tau_0)}, \ \ \ t \geq \tau_0 + 1.
    \end{align} Thus \(\|F^+(t)\|_\infty\) is in \(L^1([\tau_0 + 1, \infty))\), so it follows from (\ref{F+boundimpliesFbound:bound}) that
    \begin{align}\label{fourspheregeneralgap:fullcurvaturesupbound}
        \sup_{t \geq \tau_0 + 2} \|F(t)\|_\infty < \infty.
    \end{align}Note from this bound that the Uhlenbeck convergence \(A(n_j) \to A_\infty'\) is in fact smooth convergence modulo gauge, and \(A_\infty'\) is a connection on the original bundle.
    As in \cite[Theorem 3.7]{instantons}, we have for \(t > t_1 \geq \tau_0\)
    \begin{align}\label{fourspheregeneralgap:rhosmall}
        \int_{t_1}^t \|D^\ast F(s)\| \, ds \leq Ce^{-3(t_1 - \tau_0)}.
    \end{align} It then follows from (\ref{fourspheregeneralgap:fullcurvaturesupbound}) and Theorem \ref{thm:Hkestimates} that as \(t \to \infty\), \(A(t)\) converges exponentially in \(H^k\) for all \(k\) to some smooth instanton \(A_\infty\). It also follows that \(A_\infty\) is smoothly gauge equivalent to \(A_\infty'\) \cite[\textsection 2.3.7]{donkron}. \par 
     Finally, we explain the deformation-retraction property of the flow. Given a connection $A$ and $t \geq 0,$ let $\varphi_A(t)$ denote the solution of (\ref{ymf}) starting at $A$ evaluated at time $t.$ By the first part of the theorem, we may define the following map
    \begin{align*}
        \Psi &: \calN \times [0, 1] \to \calN \\
        \Psi(A, s) &:= 
        \begin{cases}
            \varphi_A\left(\frac{s}{1-s}\right) & 0 \leq s < 1 \\
            \lim_{t \to \infty} \varphi_A(t) & s = 1.
        \end{cases}
    \end{align*} Continuity of \(\Psi\) at \((A, s)\) for \(s \in [0, 1)\) follows from item (\ref{wellposedness:wellposedness}) of Theorem \ref{thm:wellposedness}. Continuity at \((A, 1)\) follows from Theorem \ref{thm:continuityatinfinitetime}, since instantons are in particular local minimizers of the Yang-Mills energy. Therefore, \(\Psi\) is continuous, and since the instantons are fixed under (\ref{ymf}), \(\Psi\) is a deformation retraction. Finally, the Sobolev multiplication theorem implies that the space of \(H^{k+1}\) gauge transformations acts smoothly on the space of \(H^k\) connections. Since this action preserves \(\calN\) and commutes with the mapping \(A \mapsto \varphi_A(\cdot)\), \(\Psi\) descends to the quotient of \(\calN\) by the action of \(H^{k+1}\) gauge transformations, as desired.
\end{proof}

\begin{rmk}
   Recall that for \(\CP^2\) with the Fubini-Study metric, there are no ASD instantons on the \(SU(2)\) bundle with \(c_2 = 1\) \cite[4.1.3, Example (iii)]{donkron}. Hence, there is no parabolic gap. This is in contrast to the elliptic \(L^2\) gap of Huang \cite{huangkahlersurfaces} for positive scalar curvature K{\"a}hler surfaces.
\end{rmk}

\vspace{5mm}

\subsection{Path-connectedness}\label{ss:pathconnected} We now describe our shortcut in Taubes's proof of his path-connectedness theorem. 

Fix $r \geq 2$ and let $\tilde{\mathcal{B}}_k$ denote the moduli space of framed connections on the charge-$k$ $\SU(r)$-bundle on $S^4,$ and let $\calM_k \subset \tilde{\calB}_k$ denote the moduli space of framed instantons. Taubes's proof has two separate ingredients, Theorems 1.3 and 1.4 in \cite{taubespathconnected}. The first contains Taubes's Lusternik-Schnirelman theory for the Yang-Mills functional,\footnote{It is interesting to note that the Lusternik-Schnirelman part of Taubes's proof also relies on a flow (different from Yang-Mills flow) to give a deformation retraction from a smaller neighborhood onto the space of instantons. See \cite[Proposition 3.1]{taubespathconnected}.} while the second contains his key trick for proving path-connectedness. 
We can replace 
the Lusternik-Schnirelman arguments entirely with Theorem \ref{thm:fourspheregap} and also simplify one step in the proof of the second result. 

\begin{thm}[Taubes \cite{taubespathconnected}, Theorem 1.4]\label{thm:taubespathconnected}
Suppose that for every $1 \leq k' < k,$ $\calM_{k'}$ is path-connected. Then the subset
\begin{equation}\label{leq2}
\calN_k = \{ \LB A \RB \in \tilde{\calB}_k \mid \YM(A) < 4 \pi^2 \left( k + 2 \right) \}
\end{equation}
is also path-connected.
\end{thm}
\begin{proof} We adopt Taubes's notation and recall the proof in \cite[\S 6]{taubespathconnected}, while obviating the use of \cite[Prop. 6.4]{taubespathconnected}. 


Let $m_0, m_1 \in \calM_k.$ By \cite[Lemma 6.3]{taubespathconnected}, we may assume without loss of generality that both $m_0$ and $m_1$ have nonvanishing anti-self-dual curvature 
at the south pole $s \in S^4.$ In this case,
according to \cite[Prop. 6.2]{taubespathconnected}, it is possible to choose $b_0, b_1 \in \mathcal{M}_1$ such that $\YM((m_i - b_i)_\rho) < 4 \pi^2 \left( k + 1 \right)$
for $i = 0,1$; in particular, we have
$$(m_i - b_i)_\rho \in \calN_{k-1}.$$
Here, $(m_i - b_i)_\rho,$ $i = 0,1,$ are the connections of charge $k -1$ obtained by Taubes's ``subtraction'' procedure, and $\rho > 0$ is a small scale parameter. By hypothesis, $\calM_{k-1}$ is path-connected, so by Theorem \ref{thm:fourspheregap}, $\calN_{k-1}$  is as well. Hence, there exists a path $\gamma(t) \in \calN_{k-1},$ $t \in \LB 0,1 \RB,$ with $\gamma(i) = (m_i - b_i)_\rho$ for $i = 0,1.$

Next, since $\calM_1$ is path-connected (by hypothesis), we can choose $\phi(t) \in \calM_1$ with $\phi(i) = b_i$ for $i = 0,1.$ According to \cite[Lemma 6.5]{taubespathconnected}, 
for $r$ sufficiently small, the glued path
$$\psi(t):= \left( \gamma(t) - \alpha \phi(t) \right)_r$$
has energy less than $(k + 1) + 1 = k + 2,$ so lies entirely within $\calN_k.$ Here, $\alpha$ is the inversion map across the equator of $S^4$. Its endpoints are
$$\psi(i) = \left( (m_i - b_i)_\rho - \alpha t_\rho b_i \right)_r,$$
for $i = 0,1,$ where $t_\rho$ denotes pullback under dilation by the factor $\rho.$

Finally, by \cite[Lemma 6.6]{taubespathconnected}, it is possible to connect each endpoint $\psi(i)$ to $m_i$ within $\calN_k$ by ``cancelling'' $b_i$ against $t_\rho \alpha b_i,$ while again staying entirely within $\calN_k.$ The result is a path connecting $m_0$ and $m_1$ within $\calN_k,$ as desired.
\end{proof}

\begin{proof}[Proof of Theorem \ref{thm:pathconnected}] First note that $\calM_1$ 
is indeed path-connected---see the proof of \cite[Proposition 6.2]{taubespathconnected} on p. 381. Since $\calM_k \simeq \calN_k$ by our Theorem \ref{thm:fourspheregap}, path-connectedness of $\calM_k$ now follows from Theorem \ref{thm:taubespathconnected} by induction.
\end{proof}

\begin{cor}[Atiyah-Drinfeld-Hitchin-Manin \cite{adhm}]\label{cor:adhm} For $r = 2,$ the space $\calM_k$ consists only of instantons arising from ADHM data. 
\end{cor}
\begin{proof} Let $\calM'_k \subset \calM_k$ denote the subset of instantons that arise from the ADHM construction (see \cite[Ch. II]{atiyahgeometry} for an elementary description). It is established in Donaldson-Kronheimer \cite[\S 3.4.2-4]{donkron} that $\calM'_k$ is a closed submanifold of $\tilde{\calB}_k$ of dimension $8k.$ Meanwhile, according to the Atiyah-Hitchin-Singer Theorem \cite{ahs}, the space $\calM_k$ is also a closed submanifold of $\tilde{\calB}_k$ of the same dimension. Therefore $\calM'_k$ is both closed and open relative to $\calM_k;$ since $\calM_k$ is connected, the two must be equal. 
\end{proof}

\vspace{10mm}

\section{Deformation to instantons on quaternion-K\"ahler manifolds}

We first relate the \(M^{2, 4}\) norm to the monotone 
quantity \(\Phi\) from \cite[Def. 5.5]{oliveirawaldron}, which we now recall.  Let \(\varphi \in C_c^\infty([0, 1))\) be a standard cutoff for \([0, \frac{1}{2}]\), let \(x, x_1 \in M\), and let
$$\rho_1 = \min\{\text{inj}(M), 1\}.$$
Set \(\varphi_{x_1, \rho_1}(y) := \varphi\left(\frac{d(x_1, y)}{\rho_1}\right)\), where \(d\) is the Riemannian distance. For a function \(f\) on \(M\), we set
    \begin{align*}
        \Phi_{x_1, \rho_1}(f, R, x) := \frac{R^{4-n}}{(4\pi)^{\frac{n}{2}}} \int_M |f(y)|^2 e^{-\frac{d(x, y)^2}{4R^2}} \varphi_{x_1, \rho_1}(y) \, dV(y).
    \end{align*}
\begin{lemma}\label{lemma:morreycontrolsmonotonicity}
    There exists \(C_{(\ref{morreycontrolsmonotonicity:bound})} \geq 1\), depending only on the geometry of \(M\), such that for \(R \in (0, 1], q \in [2, \infty],\) and \(\lambda \in [0, n]\), we have
    \begin{align}\label{morreycontrolsmonotonicity:bound}
        \Phi_{x, \rho_1}(f, R, x) \leq C_{(\ref{morreycontrolsmonotonicity:bound})} R^{4-\frac{2\lambda}{q}} \|f\|_{q, \lambda}^2.
    \end{align}
\end{lemma}
\begin{proof}
Since \(M\) is compact, for \(R \in (0, 1]\) the function \((4\pi R^2)^{-\frac{n}{2}}e^{-\frac{d(x, y)^2}{4R^2}}\) approximates the heat kernel \(H(x, y, R^2)\) of \(M\). In particular, the proof of \cite[Proposition 3.2]{estimatespaper} in the case \(s_1 = \frac{q}{2}, s_2 = \infty,\) and \(t = R^2\) implies that 
\begin{align*}
    \Phi_{x, \rho_1}(f, R, x) \leq R^4C_{(\ref{morreycontrolsmonotonicity:bound})}(R^2)^{-\frac{\lambda}{q}}\||f|^2\|_{\frac{q}{2}, \lambda} \leq C_{(\ref{morreycontrolsmonotonicity:bound})}R^{4-\frac{2\lambda}{q}} \|f\|_{q, \lambda}^2.
\end{align*}
\end{proof}

We recall the definition of a neighborhood deformation retract.
\begin{defn}\label{defn:nbddefret}
    Given a topological space \(X\), a closed subset \(B \subset X\) is called a neighborhood deformation retract of \(X\) if there is a continuous function \(f : X \to [0, 1]\), with \(B = f^{-1}(0)\), and a homotopy \(G : X \times [0, 1] \to X\) satisfying \(G_0 = \text{id}\), \(G_s|_B = \text{id}\) for \(s \in [0, 1]\), and \(G_1(x) \in B\) if \(f(x) < 1\).
\end{defn}
 We will prove the following detailed version of Theorem \ref{thm:quaternionkahlergap}.
\begin{thm}[Cf. Theorem \ref{thm:quaternionkahlergap}]\label{thm:detailedquaternionkahlergap}
  Let $M$ be a compact quaternion-K{\"a}hler manifold of dimension \(n = 4m\) and let $k \geq 2m.$
  There exists \(\delta_0 \in (0, 1)\) and \(C_{(\ref{detailedquaternionkahlergap:estimatepapercriticalbounds})} > 1\),
  depending on the geometry of \(M\), as follows.

  Let $A_0$ be an \(H^k\) pseudo-holomorphic connection on a \((G, \rho)\)-bundle \(E \to M\), for which
\begin{align}\label{detailedquaternionkahlergap:F+initiallysmall}
      \|F_{A_0}^+\|_{2, 4} < \delta \leq \delta_0.
  \end{align} 
  Then the solution of (\ref{ymf}) starting from \(A_0\) exists for all time, and we have the estimate 
  \begin{align}\label{detailedquaternionkahlergap:estimatepapercriticalbounds}
      \|F^+(t)\|_{2, 4} + \min\{1, t\}\|F^+(t)\|_\infty \leq C_{(\ref{detailedquaternionkahlergap:estimatepapercriticalbounds})}\delta.
  \end{align} 
  Suppose further that \(M\) has positive scalar curvature. Then the solution converges exponentially to an instanton as $t \to \infty,$ and the space of \(H^k\) 
  instantons is a neighborhood deformation retract of the space of \(H^k\) pseudo-holomorphic connections.
  
  Finally, let $E_0 > 0,$ $q \in (2, \infty],$ and \(\lambda \in [0, 4)\). Let $\varepsilon_0, R_0$ be as in \cite[Theorem 6.2]{oliveirawaldron}. Assuming that $\delta > 0$ is sufficiently small, depending also on $E_0,q,$ and $\lambda,$ the following estimate is true. 
  In addition to (\ref{detailedquaternionkahlergap:F+initiallysmall}), suppose that
  \begin{align}\label{detailedquaternionkahlergap:hypotheses1}
      \|F_{A_0}\|_{2,4} + \| F^+_{A_0} \|_{q,\lambda} < E_0.
  \end{align}
 Given $x \in M,$ let $R \in \left( 0, R_0 \RB$ be such that
  \begin{align}\label{detailedquaternionkahlergap:hypotheses2}
      \Phi_{x, \rho_1} \left( |F_{A_0}|, R, x \right) \leq \frac{\eps_0}{2}.
  \end{align}
  We then have
  \begin{align}\label{detailedquaternionkahlergap:regularityscale}
      |F(x,t)| \leq \frac{C_n}{R^2}
  \end{align}
  for $(x,t) \in B_{R/2}(x) \times \LB R^2, \infty \RB.$
\end{thm}

  \begin{proof}
  In the sequel, \(C > 0\) is a constant depending on the geometry of \(M\) which may increase in each appearance. We first prove the bound (\ref{detailedquaternionkahlergap:estimatepapercriticalbounds}). Since the hypothesis (\ref{detailedquaternionkahlergap:F+initiallysmall}) and the conclusion (\ref{detailedquaternionkahlergap:estimatepapercriticalbounds}) are gauge-invariant, we WLOG assume that \(A_0\) is smooth, in view of item \ref{wellposedness:gaugeequivalenttosmooth} of Theorem \ref{thm:wellposedness}. By \cite[Prop. 4.1]{oliveirawaldron}, \(A(t)\) is pseudo-holomorphic as long as it exists. Let \([0, T)\) be the interval of existence of \(A(t)\). By the Weitzenb{\"o}ck formula (\ref{preliminaries:qkweitzenbock}) for \(\Omega_+^2(\mathfrak{g}_E)\) and Kato's inequality,
  \begin{align*}
      \left(\frac{\partial}{\partial t} - \Delta \right) |F^+| \leq B_1|F^+|^2 + B_2|F^+|,
  \end{align*} where \(B_1 > 0\) is universal and \(B_2 > 0\) depends on the geometry of \(M\). Thus by \cite[Theorem 1.2]{estimatespaper}, (\ref{detailedquaternionkahlergap:estimatepapercriticalbounds}) holds provided
  \begin{align*}
      \|F_{A_0}^+\|_{2, 4} + \sup_{0 \leq t < T} \|F^+(t)\| < \delta_0
  \end{align*} for \(\delta_0\) small enough. Now, on a compact manifold of dimension \(n \geq 4\), the \(L^2\) norm is dominated by the \(M^{2, 4}\) norm up to a factor of \(C\). Then since \(\|F^+(t)\|\) is nonincreasing in view of (\ref{preliminaries:qkchernweil}), we have that
    \begin{align*}
        \sup_{0 \leq t < T} \|F^+(t)\| \leq C\|F_{A_0}^+\|_{2, 4}.
    \end{align*} Thus (\ref{detailedquaternionkahlergap:estimatepapercriticalbounds}) indeed holds. Since \(\|F^+(t)\|_\infty \in L_{\text{loc}}^1((0, T))\), we deduce from \cite[Cor. 7.3]{oliveirawaldron} and item (\ref{wellposedness:blowupcharacterization}) of Theorem \ref{thm:wellposedness} that \(A(t)\) exists for all time, with \(\|F(t)\|_\infty\) blowing up at most exponentially as \(t \to \infty\). \par 
    We next specialize to the case that \(M\) has positive scalar curvature. 
    In the sequel, \(c > 0\) is a constant depending on \(M\) which may decrease in each appearence but importantly will remain positive. The evolution equation for \(F\) and the Weitzenb{\"o}ck formula (\ref{preliminaries:qkweitzenbock}) imply
    \begin{align*}
        \frac{1}{2}\frac{d}{dt} \int |F^+|^2 &= -\int \langle \nabla^\ast \nabla  F^+, F^+\rangle + \int \langle \llbracket F^+, F^+\rrbracket, F^+\rangle - 2c \int |F^+|^2 \\
        &=  -\int| \nabla  F^+|^2 + \int \langle \llbracket F^+, F^+\rrbracket, F^+\rangle - 2c \int |F^+|^2 \\
        &\leq -\int| \nabla  F^+|^2 + C\int |F^+|^3 - 2c\int |F^+|^2.
    \end{align*} By H{\"o}lder's inequality,
    \begin{align*}
        \int |F^+|^3 \leq \left(\int |F^+|^{\frac{2n}{n-2}}\right)^{\frac{n-2}{n}}\left(\int |F^+|^{\frac{n}{2}}\right)^{\frac{2}{n}}.
    \end{align*}By Kato's inequality and the Sobolev inequality on functions,
    \begin{align*}
       \|F^+\|_{\frac{2n}{n-2}}^2 \leq C(\|\nabla F^+\|^2 + \|F^+\|^2),
    \end{align*} where again \(C\) is independent of the reference connection. Hence if \(\delta_0\) is sufficiently small, (\ref{detailedquaternionkahlergap:estimatepapercriticalbounds}) implies that for \(t \geq 1\),
    \begin{align*}
        \frac{d}{dt} \int |F^+|^2 \leq -2c \int |F^+|^2.
    \end{align*} Thus by Gronwall's inequality, we have for \(t \geq 1\)
    \begin{align}\label{detailedquaternionkahlergap:F+energyexponentialdecay}
        \|F^+(t)\| \leq \|F^+(0)\|e^{-ct} \leq C\delta e^{-ct}.
    \end{align}  Then  (\ref{detailedquaternionkahlergap:estimatepapercriticalbounds}) and parabolic Moser iteration yield 
    \begin{align}\label{detailedquaternionkahlergap:F+supexponentialdecay}
        \|F^+(t)\|_\infty < C\delta e^{-ct}, \ \ \ t \geq 1.
    \end{align} Thus \(\|F^+(t)\|_\infty\) is in \(L^1([1, \infty))\), so it follows from \cite[Cor. 7.3]{oliveirawaldron} that
    \begin{align}\label{detailedquaternionkahlergap:Fsupbound}
       \sup_{t \geq 1} \|F(t)\|_\infty < \infty. 
    \end{align} Furthermore, it follows from (\ref{preliminaries:qkchernweil}) and (\ref{detailedquaternionkahlergap:F+supexponentialdecay}) that we have, analogously to (\ref{fourspheregeneralgap:rhosmall}),
    \begin{align*}
        \int_{t'}^{t} \|D^\ast F(s)\| \, ds \leq C e^{-ct'}, \ \ \ 2 \leq t' < t \leq \infty.
    \end{align*} Therefore, we have as in the proof of Theorem \ref{thm:fourspheregeneralgap} that \(A(t)\) converges in \(C^\infty\) to an instanton exponentially as \(t \to \infty\). \par 
    We next show that the subset of instantons is a neighborhood deformation retract of the space of pseudo-holomorphic connections. 
    Let
    \begin{align*}
        X_1 &:= \{A \in H^k(\mathcal{A}_E) \mid A \text{ is pseudo-holomorphic}\} \\
        X_0 &:= \{A \in X_1 \mid F_A^+ = 0\} \\
        U_1 &:= \{A \in X_1 \mid \|F_A^+\|_{2, 4} < C_{(\ref{detailedquaternionkahlergap:estimatepapercriticalbounds})}^{-1} \delta_0\} \\
        U_2 &:= \{A \in X_1 \mid \|F_A^+\|_{2, 4} < \delta_0\}.
    \end{align*} Recall that $C_{(\ref{detailedquaternionkahlergap:estimatepapercriticalbounds})} > 1.$ Let \(h : [0, \infty) \to [0, 1]\) be a smooth, nondecreasing function such that \(h(0) = 0\) and \(h([C_{(\ref{detailedquaternionkahlergap:estimatepapercriticalbounds})}^{-1}\delta_0, \infty)) = 1\). Let \(\tilde{h} : [0, \infty) \to [0, 1]\) be a smooth, nonincreasing function such that \(\tilde{h}([0, C_{(\ref{detailedquaternionkahlergap:estimatepapercriticalbounds})}^{-1}\delta_0]) = 1\) and \(\tilde{h}([\delta_0, \infty)) = 0\). Let $\varphi_A(t)$ be as in the proof of Theorem \ref{thm:fourspheregeneralgap}. Define
    \begin{align*}
        f &: X_1 \to [0, 1] \\
        f(A) &:= h(\|F_A^+\|_{2, 4}) \\
        H &: X_1 \times [0, 1] \to [0, 1] \\
        H(A, s) &:= \min\{\tilde{h}(\|F_A^+\|_{2, 4}), s\} \\
        \Psi(A, s) &:= 
        \begin{cases}
            \varphi_A\left(\frac{H(A, s)}{1-H(A, s)}\right) & 0 \leq s < 1, \\
            \lim_{t \to \infty} \varphi_A(t) & H(A, s) = 1.
        \end{cases}
    \end{align*} Note that \(\Psi(U_1 \times [0, 1]) \subset U_2\) by (\ref{detailedquaternionkahlergap:estimatepapercriticalbounds}). By item (\ref{wellposedness:wellposedness}) of Theorem \ref{thm:wellposedness} and Prop. \ref{thm:continuityatinfinitetime}, \(\Psi\) is continuous. Thus \(X_0\) is a neighborhood deformation retract of \(X_1\) with homotopy \(\Psi\) and associated function \(f\). \par 
    We finally establish the refined estimate (\ref{detailedquaternionkahlergap:regularityscale}) on the full curvature of \(A(t)\) using \cite[Theorem 6.2]{oliveirawaldron}. Recall the constants \(\varepsilon_0, R_0\) from the statement of \cite[Theorem 6.2]{oliveirawaldron}. By (\ref{detailedquaternionkahlergap:hypotheses1}) and Lemma \ref{lemma:morreycontrolsmonotonicity}, 
    the required energy and entropy bounds on $A_0$ \cite[(6.7-8)]{oliveirawaldron} are satisfied. Moreover, by (\ref{detailedquaternionkahlergap:hypotheses2}) we have
    \begin{align*}
        \Phi_{x, \rho_1}(|F_{A_0}|, R, x) \leq \frac{\varepsilon_0}{2}.
    \end{align*} Thus, to obtain hypothesis \cite[(6.9)]{oliveirawaldron}, we just need to show
    \begin{align*}
        \kappa \int_0^\infty \|F^+(t)\|_\infty \, dt < \frac{\varepsilon_0}{2},
    \end{align*} where \(\kappa\) depends only on \(n\). By assumption
    \begin{align*}
        \|F_{A_0}^+\|_{q, \lambda} \leq E_0,
    \end{align*} and we have
    \begin{align*}
        \alpha := \frac{\lambda}{2q} < 1.
    \end{align*} Thus by \cite[(1.7)]{estimatespaper}, we have for some \(C' > 0\), depending on \(E_0, \alpha,\) and the geometry of \(M\), that
    \begin{align*}
        \|F^+(t)\|_{\infty} \leq \min\{C't^{-\alpha}, C\delta t^{-1}\} + C\delta. 
    \end{align*}  Since \(\alpha < 1\), we may choose \(T' \in (0, 1)\) such that 
    \begin{align*}
        \int_0^{T'} \|F^+(t)\|_\infty \, dt < \frac{\varepsilon_0}{4\kappa}.
    \end{align*} On the other hand, if \(\delta\) is small enough depending on \(E_0, \alpha,\) and the geometry of \(M\),  (\ref{detailedquaternionkahlergap:F+supexponentialdecay}) implies
    \begin{align*}
        \int_{T'}^\infty \|F^+(t)\|_\infty < \frac{\varepsilon_0}{4\kappa}.
    \end{align*} Hence
    \begin{align*}
        \int_0^\infty \|F^+(t)\|_\infty \, dt < \frac{\varepsilon_0}{2\kappa}. 
    \end{align*} The desired bound (\ref{detailedquaternionkahlergap:regularityscale}) now follows from \cite[Theorem 6.2]{oliveirawaldron}.
\end{proof}

\vspace{10mm}

\section{Deformation to flat connections}

\begin{proof}[First proof of Theorem \ref{thm:flatgap}] 
As in the proof of Theorem \ref{thm:quaternionkahlergap}, we WLOG assume \(A_0\) is smooth. By \cite[Theorem 1.2]{estimatespaper}, we have for \(\delta_1 > 0\) small enough, depending only on the geometry of \(M\), that if
\begin{align*}
    \|F_{A_0}\|_{2, 4} < \delta \leq \delta_1,
\end{align*} then
\begin{align}\label{flatgap:estimatepapercriticalbounds}
    \|F(t)\|_{2, 4} + \min\{1, t\}\|F(t)\|_\infty \leq C_{(\ref{flatgap:estimatepapercriticalbounds})} \delta, \quad t > 0.
\end{align} Here, \(C_{(\ref{flatgap:estimatepapercriticalbounds})}\) only depends on the geometry of \(M\). The long-time existence statement of the theorem now follows in view of item (\ref{wellposedness:blowupcharacterization}) of Theorem \ref{thm:wellposedness}. \par
We next prove convergence to a flat connection. 
It suffices to prove the statement for vector bundles of a given rank $r.$ Suppose for contradiction that there does not exist \(\delta_1 > 0\), now depending on $r$ in addition to \(M\), such that for any smooth connection \(A_0\) on a bundle $E \to M$ of rank $r$ with 
\begin{align*}
    \|F_{A_0}\|_{2, 4} < \delta_1,
\end{align*} the solution \(A(t)\) of (\ref{ymf}) starting at \(A_0\) satisfies:
\begin{center}
    \(A(1)\) is contained in a gauge-invariant neighborhood \(U \supset \mathcal{F}\) as in Lemma \ref{lemma:uniformlojasiewicz}. \ \ \((\ast)\)
\end{center}
Then there exists a sequence of smooth connections \(A_i\) with 
\begin{align}\label{flatgap:morreygoingtozero}
    \|F_{A_i}\|_{2, 4} \searrow 0, 
\end{align} but for which the solution \(\tilde{A}_i(t)\) of (\ref{ymf}) starting at \(A_i\) does not satisfy \((\ast)\). In view of (\ref{flatgap:estimatepapercriticalbounds}), (\ref{flatgap:morreygoingtozero}) implies 
\begin{align*}
    \lim_{i \to \infty} \sup_{\frac{1}{2} \leq t < \infty} \|F_i(t)\|_\infty = 0.
\end{align*} Thus by standard strong Uhlenbeck compactness arguments for Yang-Mills flow, e.g. \cite[Theorem 1.3]{waldronuhlenbeck}, we may pass to a subsequence, also labeled by \(i\), for which there exist smooth gauge transformations \(u_i\) such that \(u_i(\tilde{A}_i(1))\) smoothly converges to a flat connection \(A_\infty\) on $E \to M.$ In particular, a neighborhood \(U\) of $A_\infty$ as stipulated in Lemma \ref{lemma:uniformlojasiewicz} exists, and furthermore, gauge-invariance of \(U\) implies \(\tilde{A}_i(1) \in U\) for \(i\) large enough. Thus \((\ast)\) holds, which yields the desired contradiction. \par

We may therefore let $\delta_1$ be such that $(\ast)$ holds for $A(1),$ so that the desired convergence to a flat connection follows from Lemma \ref{lemma:uniformlojasiewicz}. The argument that $\mathcal{F}$ is a neighborhood deformation retract of \(H^k(\mathcal{A}_E)\) is essentially identical to the argument given in the proof of Theorem \ref{thm:detailedquaternionkahlergap}, so we omit the details.
\end{proof}

As mentioned in the introduction, we can also give a proof of Theorem \ref{thm:flatgap} using monotonicity and $\varepsilon$-regularity in place of \cite[Theorem 1.2]{estimatespaper}.

\begin{proof}[Second proof of Theorem \ref{thm:flatgap}]
     As before, we WLOG assume \(A_0\) is smooth. We first establish long-time existence of \(A(t)\). Let $\rho_1 := \min\{\text{inj}(M), 1\}.$ Suppose \(A(t)\) exists on \([0, T]\) with \(T \in (0, \rho_1^2]\). Let \(\varepsilon_0, R_0\) be as in the statement of $\varepsilon$-regularity \cite[Theorem 6.2]{oliveirawaldron} (with \(\gamma = 1\)) with \(E_0 = E = 1\). By Lemma \ref{lemma:morreycontrolsmonotonicity} and $\varepsilon$-regularity \cite[Theorem 6.2]{oliveirawaldron} (with \(\gamma = 1\)), we have that if \(\delta_1\) is sufficiently small, depending on the geometry of \(M\), then \(\|F(T)\|_\infty < \infty\). Consequently, we may assume \(T \geq R_0^2\), and we have 
    \begin{align*}
        \|F(R_0^2)\|_\infty \leq \frac{C_n}{R_0^2}.
    \end{align*}
    Thus the fact that the energy is non-increasing implies that for \(\delta_1\) sufficiently small,
    \begin{align*}
        \sup_{x \in M, t \geq R_0^2} \Phi_{x, \rho_1}(|F(t)|, R_0, x) < \varepsilon_0.
    \end{align*} Hence, it follows from \cite[Theorem 6.2]{oliveirawaldron} again that \(A(t)\) is uniformly bounded at any time \(t \geq R_0^2\). Therefore, \(A(t)\) exists for all time. \par 
    The convergence to a flat connection follows as in the first proof of Theorem \ref{thm:flatgap}, which requires \(\delta_1\) to be small depending on $\rk(E)$ in addition to \(M\). \par 
    Finally, to show that the subset of flat connections is a neighborhood deformation retract, we just need the existence of some \(C_0\), depending only on the geometry of \(M\), such that
    \begin{align*}
        \|F(t)\|_{M^{2, 4}} \leq C_0 \delta,  \ \ \ t \geq 0.
    \end{align*} This bound follows from monotonicity, e.g. \cite[Theorem 5.7]{oliveirawaldron} (with \(\gamma = 1\)), for short time, and from $\varepsilon$-regularity for later times.
\end{proof}

\begin{cor}\label{cor:lowYMandtimeoneimpliesconvergence}
    Given $E \to M,$ with \(M\) of dimension \(n \geq 4\), there exists $\delta_2 > 0$, depending only on $\rk(E)$ and the geometry of \(M\), as follows. Let $A_0$ be an \(H^k\) connection with $\YM(A_0)  < \delta \leq \delta_2,$ and suppose that the solution $A(t)$ of (\ref{ymf}) with $A(0) = A_0$ exists for at least time $T > \delta_2^{-1}\delta^{\frac{2}{n-4}}$. Then $T = \infty$ and \(A(t)\) 
    converges to a flat connection.
\end{cor}
\begin{proof}
    Suppose \(A(t)\) exists on \(\left[0, \delta_2^{-1}\delta^{\frac{2}{n-4}}\right]\). It follows from monotonicity \cite[Theorem 5.7]{oliveirawaldron} (with \(\gamma = 1\)) and the definition of \(\Phi\) that if \(\delta_2\) is sufficiently small, then for some \(\tau\), we have
    \begin{align*}
        \left\|F\left(\tau\right)\right\|_{M^{2,4}} \leq C\sup_{x \in M, 0 < R \leq \rho_1} \Phi_{x, \rho_1}\left(\left|F(\tau)\right|, R, x\right) < \delta_1.
    \end{align*} Thus the corollary now follows from Theorem \ref{thm:flatgap}.
\end{proof}

\begin{proof}[Proof of Corollary \ref{cor:Naito}]
This follows from the contrapositive of Corollary \ref{cor:lowYMandtimeoneimpliesconvergence}.
\end{proof}

\begin{cor}\label{cor:ellipticgap}
    If $A$ is a 
    Yang-Mills connection on $E \to M$ with $\YM(A) < \delta_2,$ then $A$ is flat.
\end{cor}

\begin{proof}[First proof of Corollary \ref{cor:ellipticgap}]
If $A = A_0$ is initially Yang-Mills, then $A(t) \equiv A_0,$ so this follows directly from Corollary \ref{cor:lowYMandtimeoneimpliesconvergence}.
\end{proof}

\begin{proof}[Second proof of Corollary \ref{cor:ellipticgap}]\label{ellipticgapsecondproof}
It is easy to give a purely elliptic proof of Corollary \ref{cor:ellipticgap} along the lines of the second proof of Theorem \ref{thm:flatgap}. Supposing for 
contradiction that there is no gap, there exists a sequence \(A_i\) of 
Yang-Mills connections with $F_i \not \equiv 0$ and \(\|F_i\| \searrow 0\), which we may assume are smooth after gauge transforming. As in \cite[Fact 2.2 and Lemma 3.1]{nakajimahigherdimensions}, it follows that \(\|F_i\|_\infty \searrow 0\), together with bounds on all covariant derivatives. By strong Uhlenbeck compactness, we may pass to a subsequence, also labeled by \(i\), such that there exist smooth gauge transformations \(u_i\) for which \(u_i(A_i)\) smoothly converge to a flat connection \(A_\infty\). But then the Lojasiewicz inequality, (\ref{lojasiewicz:lojasiewiczinequality}) or \cite[Lemma 12]{byanguniqueness}, implies that \(F_{u_i(A_i)} \equiv 0 \equiv F_{A_i} \) for \(i\) large enough, contradicting our assumption.
\end{proof}

\begin{rmk}\label{rmk:ellipticgap}
Note that Corollary \ref{cor:Naito} 
gives another instance in which the parabolic $L^2$ gap fails, in this case at finite time, even while the elliptic $L^2$ gap holds.
We also note that Corollary \ref{cor:ellipticgap} evidently implies the \(L^{\frac{n}{2}}\) gap result published by Feehan \cite{feehanflatgap}.
\end{rmk}

\vspace{10mm}

\appendix

\section{Short-time existence and well-posedness of Yang-Mills flow}

  The goal of this section is to prove an existence, uniqueness, and well-posedness result for Yang-Mills flow, Theorem \ref{thm:wellposedness}. First we give our definition of a solution to (\ref{ymf}) in the Sobolev setting.

\begin{defn}[Cf. {\cite[Def. 2.1]{struweym}}]
Let $G$ be a compact Lie group, let $
\rho : G \to \GL_K(V)$ be a faithful orthogonal/unitary representation over a $K$-vector space $V$, let \(E \to M\) be a \((G, \rho)\)-bundle over a closed Riemannian \(n\)-manifold \(M\), let \(k\) be an integer such that \(k > \frac{n}{2}-1\), and let \(\tau_0 \in (0, \infty)\). We say that \(A(t)\) is a weak solution of (\ref{ymf}) on \([0, \tau_0]\) with initial data \(A_0\) if, writing $\nabla_{A(t)} = \nabla_{\text{ref}} + a(t)$ and $\nabla_{A_0} = \nabla_{\text{ref}} + a_0,$ the following hold:
\begin{enumerate}
    \item We have \(\left(a(t), F_{A(t)}\right) \in U(\tau_0),\) where \(U(\tau_0)\) is a certain Hilbert space defined in (\ref{shorttimeexistence:definitionofU}) below.
\item \(a(0) = a_0\).
\item  For all \(\varphi \in C_c^\infty(M \times (0, \tau_0), \Omega^1(\mathfrak{g}_E))\) we have
\begin{align*}
    \int_0^{\tau_0} \int_M \left\langle a(t), -\frac{\partial \varphi}{\partial t}\right\rangle +\langle F_{A(t)}, D_{A(t)} \varphi\rangle = 0.
\end{align*}
\end{enumerate}
\end{defn}
 
 \begin{thm}\label{thm:wellposedness}
     Given $K > 0,$ there exists $\tau> 0,$ depending on $K, k, \nabla_{\text{ref}},$ and the geometry of $M,$ such that for each $A_0 \in H^k\left( \mathcal{A}_E \right)$ with $\|A_0 \|_{H^k} < K,$ there exists $T \in [\tau,  \infty]$ and a unique weak solution $A(t)$ of (\ref{ymf}) with $A(0) = A_0,$ satisfying $a(t) \in C^0_{loc}(\left[0, T \right), H^k)$ and $\left( a(t), F_{A(t)} \right) \in U^+(T),$ where \(U^+(T)\) is defined in (\ref{shorttimeexistence:localhilbertspace}) below. We have the additional regularity
     \begin{equation}
        F_{A(t)} \in L_{\text{loc}}^2([0, T), H^k), \quad D_{A(t)}^\ast F_{A(t)} \in L_{\text{loc}}^2([0, T), H^{k-1}).
        \end{equation}
        The following properties also hold:
     \begin{enumerate}
         \item \label{wellposedness:wellposedness} The flow is well-posed: given $\eps > 0$ and $0 < T_0 < T,$ there exists $\delta > 0$ such that for $A_0' \in B_{\delta}(A_0) \subset H^k \left( \mathcal{A}_E \right),$ the corresponding solution $A'(t)$ exists on $\LB 0, T' \right),$ with $T' > T_0,$ and $\| A'(t) - A(t) \|_{H^k} < \eps$ for all $t \in \LB 0 , T_0 \RB.$
         \item \label{wellposedness:gaugeequivalenttosmooth} There exists a (constant-in-time) gauge transformation \(u \in H^{k+1}(\mathcal{G}_E)\) such that \(u(A(t))\) is smooth and solves (\ref{ymf}) classically for \(t > 0\).
         \item \label{wellposedness:blowupcharacterization} Either $T = \infty$ or \(\limsup_{t \nearrow T} \|F_{A(t)}\|_\infty = \infty\).
     \end{enumerate}
 \end{thm}
The first part of Theorem \ref{thm:wellposedness} is the content of Theorem \ref{thm:shorttimeexistence} in \textsection \ref{radegendim}. Item (\ref{wellposedness:wellposedness}) readily follows from Theorem \ref{thm:shorttimeexistence} and is proven in \textsection \ref{item1proof}. We prove item (\ref{wellposedness:gaugeequivalenttosmooth}) in \textsection \ref{gaugeeqtosmth} using the estimates of Theorem \ref{thm:Hkestimates} from \textsection \ref{sobolevdistest}. The proof of item (\ref{wellposedness:blowupcharacterization}) in \textsection \ref{ltecriterion} is then short given the rest of Theorem \ref{thm:wellposedness}.

\vspace{5mm}

\subsection{R{\aa}de in general dimension} \label{radegendim}
In dimensions two and three, R{\aa}de provided a robust short-time existence theory for Yang-Mills flow with \(H^1\) initial data \cite{radethesis, rade}. Assuming stronger regularity on the initial data, R{\aa}de's method extends straightforwardly to higher dimensions. 
\begin{thm}\label{thm:shorttimeexistence}
     Let \(n \geq 2\), and let \(q \in \R\) satisfy \(q \geq 1\) and \(q > \frac{n}{2} - 1\). Let \(E\) be a \((G, \rho)\)-bundle over a closed Riemannian \(n\)-manifold \(M\). Given a connection \(A_0 \in H^q(\mathcal{A}_E)\), there exists \(\tau_0 > 0\), depending on \(q, \|A_0\|_{H^q},\) \(\nabla_{\text{ref}}\), and the geometry of \(M\), such that the following holds. There is a unique weak solution \(A(t) \) of (\ref{ymf}) with \(A(0) = A_0,\) $a(t) \in C^0([0, \tau_0], H^q),$ and $\left(a(t), F_{A(t)} \right)$ is in the space $U(\tau_0)$ defined in (\ref{shorttimeexistence:definitionofU}) below.
     The curvature satisfies \(F_{A(t)} \in L^2([0, \tau_0], H^q)\).
     The solution \(A(t)\) is smooth on \(M \times [0, \tau_0]\) if \(A_0\) is smooth.
 \end{thm}

 \begin{proof}
     Denote the curvature of the fixed smooth reference $(G, \rho)$-connection $\nabla_{\text{ref}}$ by $F_{\text{ref}},$ and the induced covariant exterior derivative by $D_{\text{ref}}.$ 
     If we set \(\Omega(t) = F_{A(t)}\), then (\ref{ymf}) and 
     the corresponding evolution equation for the curvature can be expressed as the following system (where \(\#\) denotes various multilinear (not necessarily associative) operations whose precise forms will not be germane):
 \begin{empheq}[left=\empheqlbrace]{equation}
  \begin{split}\label{shorttimeexistence:yangmillssystem}
     \hspace{0.6em} \frac{d}{dt} a &+ D_{\text{ref}}^\ast \Omega = a \# \Omega \\
     \frac{d}{dt} \Omega &+ \nabla_{\text{ref}}^\ast \nabla_{\text{ref}} \Omega = F_{\text{ref}} \# \Omega + Rm \# \Omega + \nabla_{\text{ref}} a \# \Omega + a \# \nabla_{\text{ref}} \Omega  + a \# a \# \Omega    \\
     a(0) &= a_0 \\
     \Omega(0) &= \Omega_0.
  \end{split}
\end{empheq} As in \cite[(18)]{radethesis}, we shall show existence/uniqueness of solutions to the above system for a family \(a = a(t)\) of \(\mathfrak{g}_E\)-valued 1-forms and a family \(\Omega = \Omega(t)\) of \(\mathfrak{g}_E\)-valued 2-forms, where \(\Omega\) need not be the curvature of \(\nabla_{\text{ref}} + a\); at the 
end, we will show that 
\(\Omega\) is indeed the curvature of \(\nabla_{\text{ref}} + a\) if \(\Omega_0\) is the curvature of \(\nabla_{\text{ref}} + a_0\), 
just like \cite[pg. 22]{radethesis}. \par
 We next describe the function spaces we shall work in. Let \(\tau_0 > 0\). For real numbers \(r\) and \(s\), we will use the shorthand \(H^r\) to denote the Hilbert Sobolev space \(H^r(T^\ast M \otimes \mathfrak{g}_E)\) or \(H^r(\Lambda^2 T^\ast M \otimes \mathfrak{g}_E)\) and \(H^{r, s}\) to denote \(H^r([0, \tau_0], H^s)\) \cite[\textsection D.1, pg. 38]{radethesis}. We will use \(H_P^{r, s}\) to denote the parabolic subspace of \(H^{r, s}\) \cite[pg. 42]{radethesis}. Set  
 \begin{align*}
     \mu &:= q - \frac{n}{2}+1 > 0 \\
     \varepsilon &:= \frac{\min\left\{1, \mu\right\}}{10} > 0.
 \end{align*} Set
\begin{align}\label{shorttimeexistence:definitionofU}
X &= \left\{(b_0, \Psi_0) \mid b_0 \in H^q, \Psi_0 \in H^{q-1}\right\} \nonumber \\
U(\tau_0) &= \left\{(b, \Psi) \mid b \in H^{\frac{1}{2} + \varepsilon, q - 2\varepsilon} \cap H^{\frac{1}{2}, q}, \Psi \in  H^{\frac{1}{2} + \varepsilon, q-1-2\varepsilon} \cap H^{-\frac{1}{2}, q +1}\right\} \\
U_P(\tau_0) &= \left\{(b, \Psi) \mid b \in H_P^{\frac{1}{2} + \varepsilon, q - 2\varepsilon} \cap H_P^{\frac{1}{2}, q}, \Psi \in  H_P^{\frac{1}{2} + \varepsilon, q-1-2\varepsilon} \cap H_P^{-\frac{1}{2}, q +1}\right\} \nonumber \\
W_P(\tau_0) &= \left\{(b, \Psi) \mid b \in H_P^{-\frac{1}{2} + \varepsilon, q - 2\varepsilon} \cap H_P^{-\frac{1}{2}, q}, \Psi \in H_P^{-\frac{1}{2} + \varepsilon, q - 1 - 2\varepsilon} \cap H_P^{-\frac{1}{2}, q- 1}\right\}. \nonumber
\end{align} Note that in the function space $U(\tau_0),$ $\Psi$ has the space-time regularity of a solution of the heat equation with $H^{q-1}$ initial data, and $b$ has one more spatial derivative than $\Psi$. This is sensible since the curvature of a solution of (\ref{ymf}) obeys a heat-type equation. \par
For \(T \in (0, \infty]\), we also set 
\begin{align}\label{shorttimeexistence:localhilbertspace}
    U^+(T) = \left\{(a, \Omega) \mid  (a, \Omega) \in U(\tau_0) \ \forall \ \tau_0 < T\right\}.
\end{align} We say \((a, \Omega)\) is a solution of (\ref{shorttimeexistence:yangmillssystem}) if \((a, \Omega) \in U(\tau_0)\), \((a(0), \Omega(0)) = (a_0, \Omega_0)\), and \((a, \Omega)\) satisfies for all 
\begin{align*}
    (\varphi, \psi) \in C_c^\infty\left(M \times (0, \tau_0), \Omega^1(\mathfrak{g}_E) \oplus \Omega^2(\mathfrak{g}_E)\right)
\end{align*} that
\begin{align*}
    &\int_0^{\tau_0} \hspace{-0.6em} \int_M \left\langle a, -\frac{\partial \varphi}{\partial t}\right\rangle +  \langle \Omega, D_{\text{ref}} \varphi\rangle =  \int_0^{\tau_0} \hspace{-0.6em} \int_M \langle a \# \Omega, \varphi\rangle \\
    &\int_0^{\tau_0} \hspace{-0.6em} \int_M \left\langle \Omega, -\frac{\partial \psi}{\partial t}\right\rangle + \langle \Omega, \nabla_{\text{ref}}^\ast \nabla_{\text{ref}} \psi \rangle =  \int_0^{\tau_0} \hspace{-0.6em} \int_M \langle \Upsilon, \psi\rangle, 
\end{align*} where \(\Upsilon\) is the RHS of the second equation in (\ref{shorttimeexistence:yangmillssystem}). \par
We now begin the proof of existence of solutions to (\ref{shorttimeexistence:yangmillssystem}). We will solve for \((a, \Omega)\) of the form \((a_1, \Omega_1) + (a_2, \Omega_2)\), where \((a_1, \Omega_1)\) solves a homogeneous initial-value problem and \((a_2, \Omega_2)\) solves an inhomogeneous system. In the sequel, we take \(\tau_0 \leq 1\). First we solve the initial-value problem
 \begin{empheq}[left=\empheqlbrace]{equation}\label{shorttimeexistence:ivp}
  \begin{split}
\frac{d}{dt} a_1 &+ D_{\text{ref}}^\ast \Omega_1 = 0 \\
\frac{d}{dt} \Omega_1  &+ \nabla_{\text{ref}}^\ast \nabla_{\text{ref}} \Omega_1 = 0 \\
a_1(0) &= a_0 \\
\Omega_1(0) &= \Omega_0,
 \end{split}
\end{empheq} with initial data \((a_0, \Omega_0) \in X\) and \((a_1, \Omega_1) \in U(\tau_0)\). In the sequel, \(C\) will denote a positive constant, depending on the geometry of \(M\), \(\nabla_{\text{ref}}\), and \(q\), and \(C\) may increase from line to line. Since \(\nabla_{\text{ref}}\) is smooth, \cite[Prop. D.9]{radethesis} implies that the initial-value problem
\begin{align*}
&\frac{d}{dt} \Omega_1  + \nabla_{\text{ref}}^\ast \nabla_{\text{ref}} \Omega_1 = 0 \\
&\Omega_1(0) = \Omega_0
\end{align*} has a unique weak solution \(\Omega_1(t) \in H^{\frac{1}{2}-r, q-1 + 2r}([0, \tau_0]), r \in \R\). In particular, for each \(\tau_0 \in (0, 1]\), the map \(\Omega_0 \to \Omega_1(t)\) yields operators
\begin{align*}
H^{q-1} &\to H^{\frac{1}{2}-r, q-1 + 2r}([0, \tau_0]) \\
H^{q-1} &\to C^0([0, \tau_0], H^{q-1}),
\end{align*} with norms bounded respectively by \(C\max\{1, \tau_0^r\}\) and \(C\). Thus, \(\Omega_1 \in  H^{\frac{1}{2} + \varepsilon, q-1-2\varepsilon} \cap H^{-\frac{1}{2}, q +1}\). Once we obtain \(\Omega_1\), we solve for \(a_1\) by
\begin{align*}
    a_1 = a_0 - \int_0^t D_0^\ast \Omega_1(s) \, ds \in H^{\frac{1}{2} + \varepsilon, q - 2\varepsilon} \cap H^{\frac{1}{2}, q}.
\end{align*} Overall, we have obtained an operator 
\begin{align*}
    M : X &\to U(\tau_0) \\
    (a_0, \Omega_0) &\to (a_1, \Omega_1),
\end{align*} and by 
 \cite[Prop. D.7 and D.9]{radethesis} we have
\begin{align*}
    \|M\| \leq C \tau_0^{-\varepsilon}.
\end{align*} \par
Now that we have \((a_1, \Omega_1)\), we solve the inhomogeneous system 
\begin{empheq}[left=\empheqlbrace]{equation}\label{inhomogeneoussystem}
  \begin{split}
     \hspace{0.6em}
\partial_t a_2 + D_{\text{ref}}^\ast \Omega_2 &= (a_1 + a_2) \# (\Omega_1 + \Omega_2)  \\
(\partial_t + \nabla_{\text{ref}}^\ast \nabla_{\text{ref}})\Omega_2 &= F_{\text{ref}} \# (\Omega_1 + \Omega_2) + Rm \# (\Omega_1 + \Omega_2) \\ 
&+ \nabla_{\text{ref}} (a_1 + a_2) \# (\Omega_1 + \Omega_2) +  (a_1 + a_2) \# \nabla_{\text{ref}}(\Omega_1 + \Omega_2) \\
&+ (a_1 + a_2) \# (a_1 + a_2) \# (\Omega_1 + \Omega_2) \\
a_2(0) &= 0 \\
\Omega_2(0) &= 0
\end{split}
\end{empheq} via \cite[Lemma B.5]{radethesis}. To do this, we will show that the operators \(L, Q_1, Q_2, Q_3\) from \cite[pg. 15]{radethesis}, given by
\begin{align*}
    L(b, \Psi) &:= \left(\frac{db}{dt} + D_{\text{ref}}^\ast \Psi, \frac{d\Psi}{dt} + \nabla_{\text{ref}}^\ast \nabla_{\text{ref}} \Psi\right) \\
    Q_1(b, \Psi) &:= (0, F_{\text{ref}} \# \Psi + Rm \# \Psi) \\
    Q_2(b, \Psi) &:= (b \# \Psi, b \# \nabla_{\text{ref}} \Psi + \nabla_{\text{ref}} b \# \Psi) \\
    Q_3(b, \Psi) &:= (0, b \# b \# \Psi),
\end{align*} extend to maps
\begin{align*}
    L &: U_P(\tau_0) \to W_P(\tau_0) \\
    Q_1 &: U(\tau_0) \to W_P(\tau_0) \\
    Q_2 &: \text{Sym}^2U(\tau_0) \to W_P(\tau_0) \\
    Q_3 &: \text{Sym}^3U(\tau_0) \to W_P(\tau_0),
\end{align*} with \(L\) invertible, and satisfy the operator norm bounds:
\begin{empheq}[left=\empheqlbrace]{equation}\label{shorttimeexistence:estimatesforexistence}
\begin{split}
    \|L^{-1}\| &\leq C \\
    \|Q_1\| &\leq C \tau_0^{1- \varepsilon} \\
    \|Q_2\| &\leq C \tau_0^{3\varepsilon} \\
    \|Q_3\| &\leq C \tau_0^{4\varepsilon}.
\end{split}
\end{empheq} \par
By \cite[Prop. D.7 and D.8]{radethesis}, \(\frac{d}{dt}\) extends to a bounded, invertible operator
\begin{align*}
    H_P^{\frac{1}{2} + \varepsilon, q - 2\varepsilon} \cap H_P^{\frac{1}{2}, q} \to H_P^{-\frac{1}{2} + \varepsilon, q - 2\varepsilon} \cap H_P^{-\frac{1}{2}, q}
\end{align*}and \(\frac{d}{dt} + \nabla_{\text{ref}}^\ast \nabla_{\text{ref}}\) extends to a bounded, invertible operator
\begin{align*}
    H_P^{\frac{1}{2} + \varepsilon, q-1-2\varepsilon} \cap H_P^{-\frac{1}{2}, q +1} \to H_P^{-\frac{1}{2} + \varepsilon, q - 1 - 2\varepsilon} \cap H_P^{-\frac{1}{2}, q- 1},
\end{align*} and the norms of these operators and their inverses are bounded by \(C\). Moreover, \(D_{\text{ref}}^\ast\) maps
\begin{align*}
     H_P^{-\frac{1}{2}, q +1} \to H_P^{-\frac{1}{2}, q}.
\end{align*} Next, by interpolation
\begin{align*}
    H_P^{\frac{1}{2} + \varepsilon, q-1-2\varepsilon} \cap H_P^{-\frac{1}{2}, q +1} \hookrightarrow H^{-\frac{1}{2} + \varepsilon, q + 1 -2\varepsilon},
\end{align*} and we further have that \(D_{\text{ref}}^\ast\) maps
\begin{align*}
    H^{-\frac{1}{2} + \varepsilon, q + 1 -2\varepsilon} \to H^{-\frac{1}{2} + \varepsilon, q -2\varepsilon}.
\end{align*} Hence \(D_{\text{ref}}^\ast\) is a bounded operator
\begin{align*}
     H_P^{\frac{1}{2} + \varepsilon, q-1-2\varepsilon} \cap H_P^{-\frac{1}{2}, q +1} \to H_P^{-\frac{1}{2} + \varepsilon, q - 2\varepsilon} \cap H_P^{-\frac{1}{2}, q},
\end{align*}with norm bounded by \(C\). 
Since 
\begin{align*}
    L^{-1} = 
    \begin{pmatrix}
        (\frac{d}{dt})^{-1} & -(\frac{d}{dt})^{-1}D_{\text{ref}}^\ast (\frac{d}{dt} + \nabla_{\text{ref}}^\ast \nabla_{\text{ref}})^{-1} \\
        0 & (\frac{d}{dt} + \nabla_{\text{ref}}^\ast \nabla_{\text{ref}})^{-1}
    \end{pmatrix},
\end{align*} we deduce that \(\|L^{-1}\| \leq C\). \par 
Next we bound the norm of \(Q_1\). Since \(F_{\text{ref}}\) and \(Rm\) are smooth in space and constant in time, they are in \(H^{\frac{1}{2} + \varepsilon, q+\frac{n}{2}}\) with \(\|F_{\text{ref}}\|_{H^{\frac{1}{2} + \varepsilon, q+ \frac{n}{2}}} + \|Rm\|_{H^{\frac{1}{2} + \varepsilon, q+\frac{n}{2}}} \leq  C \tau_0^{-\varepsilon}\) \cite[pg. 43]{radethesis}. Also, by interpolation we have \(\Omega \in H^{\frac{1}{2}, q-1}\) \cite[(73)]{radethesis}.
By hypothesis,
\begin{align*}
    &2q+ \frac{n}{2}-1 > 0, \\
    &q - 1 \leq \min\left\{q-1, q + \frac{n}{2},  2q -1\right\}.
\end{align*} It follows from the Sobolev multiplication theorem \cite[Prop. D.7]{radethesis} that \(Q_1\) maps \(H^{\frac{1}{2} + \varepsilon, q + \frac{n}{2}} \times H^{\frac{1}{2}, q - 1} \to H^{-\frac{1}{2} + \varepsilon, q - 1}\). By \cite[(70)]{radethesis},
\begin{align*}
    Q_1(a, \Omega) \in H^{-\frac{1}{2} + \varepsilon, q - 1}
    \begin{cases}
        \inj H^{-\frac{1}{2} + \varepsilon, q-1 - 2\varepsilon} =  H_P^{-\frac{1}{2} + \varepsilon, q-1-2\varepsilon}; \\
        = H_P^{-\frac{1}{2}+\varepsilon, q-1} \inj H_P^{-\frac{1}{2}, q-1}.
    \end{cases}
\end{align*} Moreover, by the last statement of \cite[Prop. D.7]{radethesis}, we have the estimate
\begin{align*}
    \|Q_1(a, \Omega)\|_{W_P(\tau_0)} &\leq C \tau_0(\|F_{\text{ref}}\|_{H^{\frac{1}{2} + \varepsilon, q+\frac{n}{2}}} + \|Rm\|_{H^{\frac{1}{2} + \varepsilon, q+\frac{n}{2}}})\|(0, \Omega)\|_{U(\tau_0)} \\
    &\leq C \tau_0^{1 - \varepsilon}\|(a, \Omega)\|_{U(\tau_0)}.
\end{align*} \par
Next we consider \(Q_2\). By interpolation, \(\Omega \in H^{-\frac{1}{2}+4\varepsilon, q+1-8\varepsilon}\). Since \(8\varepsilon < \min\{1, \mu\}\), we have
the Sobolev multiplication map \(H^{\frac{1}{2}, q} \times H^{-\frac{1}{2} + 4\varepsilon, q+1-8\varepsilon} \to H^{-\frac{1}{2} + \varepsilon, q}\). Hence
\begin{align*}
    &a \# \Omega \in H^{-\frac{1}{2} + \varepsilon, q}
    \begin{cases}
        \inj H^{-\frac{1}{2} + \varepsilon, q - 2\varepsilon} = H_P^{-\frac{1}{2} + \varepsilon, q-2\varepsilon} ; \\
        = H_P^{-\frac{1}{2} + \varepsilon, q}\inj H_P^{-\frac{1}{2}, q}. 
    \end{cases}
\end{align*} Moreover, we have the estimate 
\begin{align*}
    \|(a \# \Omega, 0)\|_{W_P(\tau_0)} \leq C \tau_0^{3\varepsilon}\|(a, \Omega)\|_{U(\tau_0)}^2.
\end{align*} Furthermore, we have \(\nabla_{\text{ref}} a \in H^{\frac{1}{2}, q-1}\) and \(\nabla_{\text{ref}} \Omega \in H^{-\frac{1}{2} + 4\varepsilon, q-8\varepsilon}\). By our assumptions on \(q\) and \(\varepsilon\),
we have the Sobolev multiplication maps \(H^{\frac{1}{2}, q-1} \times H^{-\frac{1}{2} + 4\varepsilon, q-8\varepsilon} \to H^{-\frac{1}{2} +\varepsilon, q-1}\) and \(H^{\frac{1}{2}, q} \times H^{-\frac{1}{2} + 4\varepsilon, q+1-8\varepsilon} \to H^{-\frac{1}{2} + \varepsilon, q-1}\). Hence
\begin{align*}
    \nabla_{\text{ref}} a \# \Omega, \ a \# \nabla_{\text{ref}} \Omega\in H^{-\frac{1}{2} + \varepsilon, q-1}
    \begin{cases}
        \inj H^{-\frac{1}{2} + \varepsilon, q-1-2\varepsilon} = H_P^{-\frac{1}{2} + \varepsilon, q-1-2\varepsilon}; \\
        = H_P^{-\frac{1}{2} + \varepsilon, q-1} \inj H_P^{-\frac{1}{2}, q-1}.
    \end{cases}
\end{align*} Moreover, we have the bound
\begin{align*}
    \|(0, \nabla_{\text{ref}} a \# \Omega + a \# \nabla_{\text{ref}} \Omega)\|_{W_P(\tau_0)} \leq C\tau_0^{3\varepsilon} \|(a, \Omega)\|_{U(\tau_0)}^2.
\end{align*} Overall, we have
\begin{align*}
    \|Q_2\| \leq C\tau_0^{3\varepsilon}.
\end{align*} \par 
Finally, we consider \(Q_3\). We have \(a \in H^{\frac{1}{2} + \varepsilon, q-2\varepsilon}\), and by interpolation \(\Omega \in H^{-\frac{1}{2} + 3\varepsilon, q + 1 -6\varepsilon}\). By our assumptions on \(q\) and \(\varepsilon\),
 we have the Sobolev multiplication maps \(H^{\frac{1}{2} + \varepsilon, q-2\varepsilon} \times H^{-\frac{1}{2} + 3\varepsilon, q+1-6\varepsilon} \to H^{-\frac{1}{2} + \varepsilon, q-2\varepsilon}\) and \(H^{\frac{1}{2} + \varepsilon, q - 2\varepsilon} \times H^{-\frac{1}{2} + \varepsilon, q-2\varepsilon} \to H^{-\frac{1}{2} + \varepsilon, q - 1}\). Hence
\begin{align*}
    &a \# a \# \Omega \in H^{-\frac{1}{2} + \varepsilon, q-1}
    \begin{cases}
        \inj H^{-\frac{1}{2} + \varepsilon, q-1-2\varepsilon} = H_P^{-\frac{1}{2} + \varepsilon, q - 1 -2\varepsilon} ; \\
        = H_P^{-\frac{1}{2} + \varepsilon, q-1} \inj H_P^{-\frac{1}{2}, q-1}.
    \end{cases}
\end{align*}Moreover, we have the bound 
\begin{align*}
    \|Q_3\| \leq C \tau_0^{4\varepsilon}.
\end{align*}\par 
It now follows from (\ref{shorttimeexistence:estimatesforexistence}) that the operators \(A_0, \dots, A_3\) from \cite[pg. 19]{radethesis}, which are given by applying the $Q_i$ to $(a_1, \Omega_1)$ and $(a_2, \Omega_2)$ such that $A_j$ is $j$-homogeneous in $(a_2, \Omega_2),$ satisfy
\begin{align*}
    \|A_0\| &\leq C(\tau_0^{1-\varepsilon}\tau_0^{-\varepsilon} + \tau_0^{3\varepsilon}\tau_0^{-2\varepsilon} + \tau_0^{4\varepsilon}\tau_0^{-3\varepsilon})(1 + \|(a_0, \Omega_0)\|_X)^3 \\
    &\leq C\tau_0^\varepsilon(1 + \|(a_0, \Omega_0)\|_X)^3 \\
    \|A_1\|  &\leq C(\tau_0^{1-\varepsilon} + \tau_0^{3\varepsilon}\tau_0^{-\varepsilon} + \tau_0^{4\varepsilon}\tau_0^{-2\varepsilon})(1 + \|(a_0, \Omega_0)\|_X)^2 \\
    &\leq C\tau_0^{2\varepsilon}(1 + \|(a_0, \Omega_0)\|_X)^2 \\
    \|A_2\| &\leq C(\tau_0^{3\varepsilon} + \tau_0^{4\varepsilon}\tau_0^{-\varepsilon})(1 + \|(a_0, \Omega_0)\|_X) \\
    &\leq C \tau_0^{3\varepsilon}(1 + \|(a_0, \Omega_0)\|_X) \\
    \|A_3\| &\leq C\tau_0^{4\varepsilon}.
\end{align*} Hence the left-hand side of \cite[(24), pg. 18]{radethesis} is bounded above by
\begin{align*}
     C\tau_0^{\varepsilon}(m^{-1}(1 + \|(a_0, \Omega_0)\|_X)^3 + (1+ \|(a_0, \Omega_0)\|_X)^2 + m(1 + \|(a_0, \Omega_0)\|_X) + m^2).
\end{align*} Therefore, if we take \(m\) equal to say \(\tau_0^{-\frac{\varepsilon}{4}}\) and
\begin{align*}
    \tau_0 < \frac{(C(1 + \|(a_0, \Omega_0)\|_X)^3)^{-\frac{2}{\varepsilon}}}{2\|L^{-1}\|},
\end{align*} then by \cite[Lemma B.5]{radethesis}, the inhomogeneous system (\ref{inhomogeneoussystem}) has a solution \((a_2, \Omega_2) \in U_P(\tau_0)\). Moreover, we have the bound \(\|(a_2, \Omega_2)\|_{U_P(\tau_0)} \leq m\), and this solution is unique in \(B_m(U_P(\tau_0))\). Since \(m \to \infty\) as \(\tau_0 \searrow 0\), the same argument as on \cite[pg. 21]{radethesis} yields uniqueness of solutions to the system (\ref{shorttimeexistence:yangmillssystem}). \par We next follow the argument on \cite[pg. 20-21]{radethesis} to show that
\begin{align*}
    a &\in C^0([0, \tau_0], H^q) \\
    \Omega &\in C^0([0, \tau_0], H^{q-1}) \cap H^{0, q}.
\end{align*} By \cite[Prop. D.9]{radethesis}, \(\Omega_1 \in C^0([0, \tau_0], H^{q-1}) \cap H^{0, q}\). Consequently, it follows from integrating (\ref{shorttimeexistence:ivp}) in time that
\begin{align*}
    (1 + \nabla_{\text{ref}}^\ast \nabla_{\text{ref}})\int_0^t \Omega_1(s) \, ds = \int_0^t \Omega_1(s) \, ds + \Omega_0 - \Omega_1 \in C^0([0, \tau_0], H^{q-1}).
\end{align*} Thus by \cite[Prop. D.1]{radethesis}, \(\int_0^t \Omega_1(s) \, ds \in C^0([0, \tau_0], H^{q+1})\), whence \(a_1 \in C^0([0, \tau_0], H^q)\). Next we consider the regularity of \((a_2, \Omega_2)\). Since \((a, \Omega) \in U(\tau_0)\), the Sobolev multiplication maps obtained above when estimating the norms of \(Q_1, Q_2, Q_3\) imply that
\begin{align*}
    &a \# \Omega \in H^{-\frac{1}{2} + \varepsilon, q} \\
    & (F_{\text{ref}} + Rm)\# \Omega, (\nabla_{\text{ref}} a) \# \Omega, a \# \nabla_{\text{ref}} \Omega, a \# a \# \Omega \in H^{-\frac{1}{2} + \varepsilon, q-1}. 
\end{align*} Hence
\begin{align*}
    &\frac{d}{dt} a_2 + D_{\text{ref}}^\ast \Omega_2 \in H^{-\frac{1}{2} + \varepsilon, q} \\
    &\frac{d}{dt} \Omega_2 + \nabla_{\text{ref}}^\ast \nabla_{\text{ref}} \Omega_2 \in H^{-\frac{1}{2} + \varepsilon, q-1}.
\end{align*} Thus by \cite[Prop. D.8]{radethesis}, \(\Omega_2 \in  H^{\frac{1}{2} + \varepsilon, q-1} \cap  H^{-\frac{1}{2} + \varepsilon, q+1}\). By Sobolev embedding \cite[(74)]{radethesis}, \(\Omega_2 \in C^0([0, \tau_0], H^{q-1}) \supset  H^{\frac{1}{2} + \varepsilon, q-1}\). By interpolation \(\Omega_2 \in H^{\varepsilon, q} \subset H^{0, q}\). Furthermore, \(D_{\text{ref}}^\ast \Omega_2 \in H^{-\frac{1}{2} + \varepsilon, q}\), so that by integrating in time and Sobolev embedding, \(a_2 \in C^0([0, \tau_0], H^q)\). Thus we deduce that \((a, \Omega) = (a_1, \Omega_1) + (a_2, \Omega_2)\) has the claimed regularity properties. By bootstrapping, \((a, \Omega)\) is smooth on \(M \times [0, \tau_0]\) if \((a_0, \Omega_0)\) is smooth. \par
Finally, it follows from the implicit function theorem \cite[Proof of Lemma B.5, pg. 19]{radethesis} that given \(K > 0\), there exists a uniform \(\tau_0 > 0\) such that for all \((a_0, \Omega_0) \in B_K(0) \subset X\), the maps
\begin{align}
    &Y_1 : B_K(0) \to U(\tau_0) \\
    &Y_2 : B_K(0) \to C^0([0, \tau_0], X) \label{shorttimeexistence:smoothdependence}
\end{align} sending initial data to the corresponding solution of (\ref{shorttimeexistence:yangmillssystem}) are smooth. \par 
We still need to show that if \(\Omega_0 = F_{A_0}\), then \(\Omega(t) = F_{A(t)}\). We just need to verify that if \(q > \frac{n}{2}-1\), then the curvature of an \(H^q\) connection is in \(H^{q-1}\), so since we also have \(q \geq 1\), the argument in section B.3 of \cite{radethesis} holds. We expand
\begin{align*}
    F_{A_0} = F_{\text{ref}} + D_{\text{ref}}a_0 + a_0 \wedge a_0,
\end{align*} so since \(q > \frac{n}{2}-1\), the \(a_0 \wedge a_0\) term is in \(H^{q-1}\), and thus the claim follows.
 \end{proof}

\vspace{5mm}

\subsection{Proof of item (\ref{wellposedness:wellposedness}) of Theorem \ref{thm:wellposedness}} \label{item1proof}
 Since \(T_0 < \min\{T, \infty\}\), we may set
   \begin{align*}
       K := \sup_{0 \leq t \leq T_0} \|A(t)\|_{H^q}+1 < \infty.
   \end{align*} By Theorem \ref{thm:shorttimeexistence}, there exists \(\tau > 0\) such that for any connection \(B_0 \in H^q\) with \(\|B_0\|_{H^q} \leq K\), the solution of (\ref{ymf}) starting at \(B_0\) exists on \([0, \tau]\). By shrinking \(\tau\), we may WLOG assume \(\frac{T_0}{\tau}\) is an integer \(N \geq 1\). \par 
   Let \(\varepsilon > 0\). By the short-time well-posedness (\ref{shorttimeexistence:smoothdependence}),  we have the following. For each integer \(0 \leq i \leq N-1\), there exists an \(H^q\) neighborhood \(U_i\) of \(A(i\tau)\) such that for any \(B_0 \in U_i\), the solution \(B(t)\) of (\ref{ymf}) with \(B(0) = B_0\) exists on \([0, \tau]\) and satisfies 
   \begin{align*}
       \sup_{i\tau \leq t \leq (i+1)\tau} \|B(t-i\tau) - A(t)\|_{H^q} < \varepsilon.
   \end{align*} If \(N = 1\), then this is precisely item (\ref{wellposedness:wellposedness}) of Theorem \ref{thm:wellposedness} and we are done. If \(N \geq 2\), we have by (\ref{shorttimeexistence:smoothdependence}) again that we may choose a neighborhood \(V_{N-2} \subset U_{N-2}\) of \(A((N-2)\tau)\) such that for 
   \(B_0 \in V_{N-2}\), \(B(\tau) \subset V_{N-1} := U_{N-1}\). By induction and (\ref{shorttimeexistence:smoothdependence}), we obtain, for \(0 \leq i \leq N-2\), neighborhoods \(V_i \subset U_i\) of \(A(i\tau)\) such that for \(B_0 \in V_i\), \(B(\tau) \in V_{i+1}\). \par
   Let \(\delta > 0\) be small enough so that \(B_\delta(A_0) \subset V_0\). We claim that this choice of \(\delta\) yields item (\ref{wellposedness:wellposedness}). To see this, let \(A_0' \in B_{\delta}(A_0)\). Then the solution of (\ref{ymf}) \(A'(t)\) starting at \(A_0'\) exists on \([0, \tau]\), \(\sup_{0 \leq t \leq \tau} \|A'(t) - A(t)\|_{H^q} < \varepsilon\), and \(A'(\tau) \subset V_1\). Then by definition of \(V_i, U_i\), we have by induction that for \(0 \leq i \leq N\), \(A'(t)\) exists on \([0, i\tau]\), \(A'(i\tau) \in V_i\), and \(\sup_{0 \leq t \leq i\tau} ||A'(t) - A(t)\|_{H^q} < \varepsilon\). Taking \(i = N\), and recalling that \(N\tau = T_0\), we deduce item (\ref{wellposedness:wellposedness}) of Theorem \ref{thm:wellposedness} in general, as desired.


\vspace{5mm}

\subsection{Sobolev-distance estimate under (\ref{ymf})} \label{sobolevdistest}

To prove items (\ref{wellposedness:gaugeequivalenttosmooth}) and (\ref{wellposedness:blowupcharacterization}) of Theorem \ref{thm:wellposedness}, we will use the following adaptation of \cite[Prop. 3.3]{instantons} to higher dimensions. We note that this result does not require compactness of $M$. 
\begin{thm}\label{thm:Hkestimates}
    Let \(k\) be an integer such that
    \begin{align*}
        k \geq 
        \begin{cases}
            0 & n \leq 4 \\
            \frac{n}{2} - 1 & n \geq 5.
        \end{cases}
    \end{align*}
   
    Suppose that \(A(t)\) is an \(H^q\) solution of (\ref{ymf}), for \(q > k\), on $\LB 0, T \right)$ such that for times \(t_1 < t_2 \in [0, T)\) and constants positive constants \(\tau, K\) we have
    \begin{align}\label{Hkestimates:curvaturebound}
        \sup_{t \in [t_1 - \tau, t_2]} \|F(t)\|_\infty \leq K.
    \end{align} Set
    \begin{align*}
        \rho := \int_{t_1 - \tau}^{t_2} \|D^\ast F(t)\| \, dt .
    \end{align*} 
    Given $\Omega \Subset M,$ there exists a constant \(C_{(\ref{Hkestimates:Hkbounds})} \geq 1\), depending on \(k, \tau, K\), and the geometry of $M$ near $\Omega,$ such that
    \begin{align}\label{Hkestimates:Hkbounds}
        \|A(t_2) - A(t_1)\|_{H^{k+1}(\Omega)} \leq C_{(\ref{Hkestimates:Hkbounds})}\left(1+\rho^{(k+2)!} \right) \left(1 + \|A(t_i)\|_{H^k}^{(k+1)!} \right) \rho, \quad i = 1, 2.
    \end{align}
\end{thm}
\begin{proof}
By item (\ref{wellposedness:wellposedness}) of Theorem \ref{thm:wellposedness}, it suffices to prove (\ref{Hkestimates:Hkbounds}) when \(A(t)\) is smooth. The argument of \cite[Prop. 3.3]{instantons} works for \(n \leq 4\). In general dimension, we do the following. Let \(r \geq 0\) be an integer, and set
\begin{align*}
    f_r(t) := \max_{0 \leq i \leq r} \|\nabla^{(i)} D^\ast F(t)\|_{L^\infty(\Omega_1)}.
\end{align*} By the curvature bound (\ref{Hkestimates:curvaturebound}), one can easily pass from an $L^1$ bound on $\| D^*F(t) \|$ to supremum bounds on all higher derivatives in the manner of \cite[Prop. 3.2]{instantons} or \cite[Lemma 2.2]{waldronuhlenbeck}, to obtain
\begin{align}\label{Hkestimates:rhobound}
    \int_{s_1}^{s_2} f_r(t) \, dt \leq C_{r, \tau, K, M} \rho, \quad t_1 \leq s_1 < s_2 \leq t_2.
\end{align}
We now use (\ref{Hkestimates:rhobound}) to estimate $A(t)$ in the Sobolev distances. Since \(\Omega\) is compactly contained in \(M\), there exists 
an open set \(\Omega_1\) such that \(\Omega\) has \(C^1\) boundary and \(\Omega \Subset \Omega_1 \Subset M\). In the sequel, all Lebesgue/Sobolev spaces will be on \(\Omega_1\) unless otherwise noted. Let
     \begin{align*}
         q_r &:=
         \begin{cases}
             \frac{n}{r+1} & 0 \leq r \leq \frac{n}{2}-1 \\
             2 & r > \frac{n}{2}-1.
         \end{cases} \\
        \tilde{r} &:= \max\left\{r, \frac{n}{2}\right\}.
     \end{align*}
For \(0 \leq r \leq \frac{n}{2}-1\), define the norm
    \begin{align*}
        \|B\|_{X_r} := \max_{0 \leq i \leq r} \|\nabla_{\text{ref}}^{(i)} B\|_{q_i}.
    \end{align*}
    We will first establish the estimate
    \begin{align}\label{Hkestimates:lowinductionestimate}
        \|\nabla_{\text{ref}}^{(r)}(A(s_2) - A(s_1))\|_{q_r} \leq C_{r, \Omega_1} \left( 1 + \sup_{s_1 \leq s \leq s_2} \|A(s)\|_{X_{r-1}}^r \right) \int_{s_1}^{s_2} f_r(s) \, ds, \ \ t_1 \leq s_1 < s_2 \leq t_2,
    \end{align}
    for \(0 \leq r \leq \frac{n}{2}-1\), and the estimate
    \begin{align}\label{Hkestimates:highinductionestimate}
        \|\nabla_{\text{ref}}^{(r)}(A(s_2) - A(s_1))\|_2 \leq C_{r, \Omega_1} \left(1 + \sup_{s_1 \leq s \leq s_2} \|A(s)\|_{H^{r-1}}^r \right) \int_{s_1}^{s_2} f_r(s) \, ds, \ \  t_1 \leq s_1 < s_2 \leq t_2,
    \end{align} for \(r > \frac{n}{2}-1\).
    
   We compute
    \begin{align*}
        \nabla_{\text{ref}}^{(r)}(A(s_2) - A(s_1)) = \int_{s_1}^{s_2} \frac{d}{dt} \nabla_{\text{ref}}^{(r)} A(t) \, dt &= -\int_{s_1}^{s_2} \nabla_{\text{ref}}^{(r)} D^\ast F(t)  \,dt.
    \end{align*}  
    For all \(r \geq 0\), we have
    \begin{align}\label{Hkestimates:expression}
        \nabla_{\text{ref}}^{(r)} D^\ast F = \sum \nabla_{\text{ref}}^{(\ell_1)} A \# \cdots \# \nabla_{\text{ref}}^{(\ell_j)} A \# \nabla_A^{(m)} D^\ast F,
    \end{align} where \(j\) and the indices \(m, \ell_1, \dots, \ell_j\) are nonnegative and satisfy 
    \begin{align}\label{Hkestimates:indexsum}
        \sum_{i=1}^j \ell_i + j + m\leq r.
    \end{align}
The \(k = 0\) case of (\ref{Hkestimates:expression}) is clear, and we have
    \begin{align*}
        \nabla_{\text{ref}}\sum \nabla_{\text{ref}}^{(\ell_1)} A \# \cdots \# \nabla_{\text{ref}}^{(\ell_j)} A \# \nabla_A^{(m)} D^\ast F = &\sum \nabla_{\text{ref}} \left( \nabla_{\text{ref}}^{(\ell_1)} A \# \cdots \# \nabla_{\text{ref}}^{(\ell_j)} A \right) \# \nabla_A^{(m)} D^\ast F \\
        &+ \sum \nabla_{\text{ref}}^{(\ell_1)} A \# \cdots \# \nabla_{\text{ref}}^{(\ell_j)} A \# \nabla_{\text{ref}} \nabla_A^{(m)} D^\ast F.
    \end{align*} Using the Leibniz rule and the identity \(\nabla_{\text{ref}} = \nabla_A + A \#\) on the second term, we can obtain (\ref{Hkestimates:expression}) by induction.
    
Next, we will use the Sobolev and H\"older's inequalities to prove (\ref{Hkestimates:lowinductionestimate}) and (\ref{Hkestimates:highinductionestimate}).
      Now, by (\ref{Hkestimates:expression}) and H{\"o}lder's inequality, we have
     \begin{align*}
         \|\nabla_{\text{ref}}^{(r)} D^\ast F(s) \|_{q_r} &\leq \sum \|\nabla_{\text{ref}}^{(\ell_1)} A \# \cdots \# \nabla_{\text{ref}}^{(\ell_j)} A \# \nabla_A^{(m)} D^\ast F\|_{q_r} \\
        &\leq C_{\Omega_1}\left(1 + \sum_{j \geq 1} \|\nabla_{\text{ref}}^{(\ell_1)} A \# \cdots \# \nabla_{\text{ref}}^{(\ell_j)} A\|_{q_r} \right)f(s),
     \end{align*}
     where the indices are still subject to (\ref{Hkestimates:indexsum}).
     We will now estimate the individual terms $\|\nabla_{\text{ref}}^{(\ell_1)} A \# \cdots \# \nabla_{\text{ref}}^{(\ell_j)} A\|_{q_r}$ appearing in the sum for $j \geq 1.$
     
     \begin{claim}
     For each \(r \geq 0\), we have
    \begin{align}\label{Hkestimates:sumclaim}
        \sum_{\ell_i > r-1- \frac{n}{2}} \frac{n-2(\tilde{r}-1 - \ell_i)}{2n}
        \begin{cases}
            \leq \frac{1}{q_r}, & 0 \leq r \leq \frac{n}{2}, \text{ or } r > \frac{n}{2} \text{ and } j = 1, \\
            < \frac{1}{q_r}, & r > \frac{n}{2} \text{ and } j \geq 2.
        \end{cases}
    \end{align}
    \end{claim}
    \begin{claimproof}
    \emph{Case 1.} \(r \leq \frac{n}{2}\). By definition of \(\tilde{r}\), We have
    \begin{align*}
        \sum_{\ell_i > r - 1 - \frac{n}{2}} \frac{n-2(\tilde{r}-1 - \ell_i)}{2n} & = \sum_{i = 1}^j \frac{\ell_i+1}{n}.
    \end{align*}
    In view of (\ref{Hkestimates:indexsum}), we have 
        \begin{align*}
        \sum_{i = 1}^j \frac{\ell_i+1}{n} \leq \frac{r}{n} 
        \begin{cases}
            \leq \frac{r+1}{n} = \frac{1}{q_r} & r \leq \frac{n}{2}- 1, \\
            \leq \frac{1}{2} = q_r & \frac{n}{2}-1 < r \leq \frac{n}{2}.
        \end{cases}
    \end{align*}
        \emph{Case 2.} \(r > \frac{n}{2} \text{ and } j = 1\). Then \(\ell_j = r-1\), so
    \begin{align*}
        \sum_{\ell_i > r - 1 - \frac{n}{2}} \frac{n-2(\tilde{r}-1-\ell_i)}{2n} = \frac{n-2(r-1-(r-1))}{2n} = \frac{1}{2} = \frac{1}{q_r}.
    \end{align*}
    \emph{Case 3.} Finally suppose that \(r > \frac{n}{2} \text{ and } j \geq 2\). Let \(J := |\{\ell_i \mid \ell_i >  r - 1 - \frac{n}{2}\}|\). It follows from (\ref{Hkestimates:indexsum}) that either \(J \geq 2\), or \(J = 1\) and \(\ell_1 \leq r-2\). In the first case
    \begin{align*}
         \sum_{\ell_i  > r - 1 - \frac{n}{2}} \frac{n-2(\tilde{r}-1-\ell_i)}{2n} &= J\frac{n-2r}{2n} + \sum_{\ell_i  > r - 1 - \frac{n}{2}} \frac{\ell_i+1}{n}  \\
         &\leq J\frac{n-2r}{2n} + \frac{r}{n} \\
         &= \frac{1}{2} + (J-1)\frac{n-2r}{2n} \\
         &< \frac{1}{2} = \frac{1}{q_r},
    \end{align*} and in the second case
    \begin{align*}
        \frac{n-2(r-1-\ell_1)}{2n} \leq \frac{n-2(r-1-(r-2))}{2n} < \frac{1}{2} = \frac{1}{q_r}.
    \end{align*}
    \end{claimproof}
    We can now prove (\ref{Hkestimates:lowinductionestimate}) and (\ref{Hkestimates:highinductionestimate}). From H{\"o}lder's inequality, if \(1 \leq p_1, \cdots, p_s, q \leq \infty\) are such that \(\sum_{i = 1}^s \frac{1}{p_i} \leq \frac{1}{q}\), then there exists a constant \(C\), depending on \(\text{Vol}(\Omega_1), p_1, \cdots, p_s, q\), such that for functions \(f_1, \cdots, f_s\),
    \begin{align}\label{Hkestimates:Holdersinequality}
        \left\|\prod_{i = 1}^s f_i\right\|_q \leq C\prod_{i = 1}^s\|f_i\|_{p_i}.
    \end{align} First note that if \(\ell_i > r - 1 - \frac{n}{2}\), then
    \begin{align}\label{Hkestimates:validLebesgueexponent}
        0 < \frac{n-2(\tilde{r}-1 - \ell_i)}{2n} \leq 1.
    \end{align} The strict lower bound of \(0\) follows for \(r > \frac{n}{2}\), whence \(\tilde{r} = r\), by the assumption on \(\ell_i\) since \(\frac{n-2(\tilde{r}-1 - \ell_i)}{2n} = \frac{n-2(r-1 - \ell_i)}{2n}\), and follows for \(r \leq \frac{n}{2}\), whence \(\tilde{r} = \frac{n}{2}\), since \(1 + \ell_i > 0\). The upper bound of \(1\) follows from \(\tilde{r} - 1 - \ell_i \geq 0\) by (\ref{Hkestimates:indexsum}). \par
    Now, if \(r \leq \frac{n}{2}\), then \(\ell_i > r - 1 - \frac{n}{2}\) for all \(i\), and moreover \(\ell_i \leq \frac{n}{2}-1\) by (\ref{Hkestimates:indexsum}). Therefore, (\ref{Hkestimates:sumclaim}), (\ref{Hkestimates:Holdersinequality}), and (\ref{Hkestimates:validLebesgueexponent}) yield
    \begin{align*}
        \|\nabla_{\text{ref}}^{(\ell_1)} A \# \cdots \# \nabla_{\text{ref}}^{(\ell_j)} A\|_{q_r} &\leq  C_{r, \Omega_1} \prod_{i=1}^j \|\nabla_{\text{ref}}^{(\ell_i)} A\|_{\frac{2n}{n-2(\tilde{r}-1-\ell_i)}} = C_{r,  \Omega_1} \prod_{i=1}^j \|\nabla_{\text{ref}}^{(\ell_i)} A\|_{\frac{n}{\ell_i+1}}.
    \end{align*} The bound (\ref{Hkestimates:lowinductionestimate}) now follows the definition of the \(X_r\) norm and the fact that \(j \leq r\) by (\ref{Hkestimates:indexsum}). If \(r > \frac{n}{2}\) and \(j = 1\), then (\ref{Hkestimates:indexsum}) implies that the term is \(\nabla_{\text{ref}}^{(r-1)} A\), and clearly
    \begin{align}\label{Hkestimates:onlyoneterm}
        \|\nabla_{\text{ref}}^{(r-1)} A\|_2 \leq \|A\|_{H^{r-1}}.
    \end{align} On the other hand, suppose \(r > \frac{n}{2}\) and \(j \geq 2\). Let
    \begin{align*}
        G_r := \max\left\{\ell_1, \cdots, \ell_j \text{ satisfying } (\ref{Hkestimates:indexsum}) \mid \sum_{\ell_i > r - 1 -\frac{n}{2}} \frac{n-2(r-1-\ell_i)}{2n}\right\}.
    \end{align*} Note that \(\frac{1}{2} - G_r > 0\) since \(G_r\) is the maximum of a finite set of numbers which are strictly less than \(\frac{1}{2}\) by (\ref{Hkestimates:sumclaim}). Moreover, \(G_r \geq 0\) by the left inequality in (\ref{Hkestimates:validLebesgueexponent}). Let
    \begin{align*}
        \frac{1}{p_r} := \frac{1}{r}\left(\frac{1}{2} - G_r\right). 
    \end{align*} Note that \(p_r\) depends only on \(n\) and \(r\), and moreover \(1 \leq p_r < \infty\) since \(0 \leq G_r < \frac{1}{2}\). Thus by (\ref{Hkestimates:indexsum}),
    \begin{align*}
        \sum_{\ell_i \leq r - 1 - \frac{n}{2}} \frac{1}{p_r} + \sum_{\ell_i > r-1-\frac{n}{2}} \frac{n-2(r-1-\ell_i)}{2n} &\leq \frac{r}{p_r} + \sum_{\ell_i > r-1-\frac{n}{2}} \frac{n-2(r-1-\ell_i)}{2n}  \\
        &\leq \frac{1}{2}.
    \end{align*}Consequently, 
    \begin{align*}
        \|\nabla_{\text{ref}}^{(\ell_1)} A \# \cdots \# \nabla_{\text{ref}}^{(\ell_j)} A\|_{q_r} &\leq  C_{r, \Omega_1} \prod_{\ell_i \leq r-1-\frac{n}{2}} \|\nabla_{\text{ref}}^{(\ell_i)} A\|_{p_r} \prod_{\ell_i >r-1-\frac{n}{2}} \|\nabla_{\text{ref}}^{(\ell_i)} A\|_{\frac{2n}{n-2(r-1-\ell_i)}}.
    \end{align*} 
     Since the reference connection is metric-compatible and since \(\Omega_1\) has \(C^1\) boundary, it follows from Kato's inequality and the Sobolev inequality on functions that we have the following bounds for \(B \in 
     \Gamma((\otimes^l T^\ast M) \otimes \mathfrak{g}_E), l \geq 0\):
    \begin{align}\label{Hkestimates:sobolevinequalities}
        \|\nabla_{\text{ref}}^{(\ell)} B\|_{\frac{2n}{n-2(r-1-\ell)}} &\leq C_{\Omega_1} \|B\|_{H^{r-1}}, \quad \ell > r - 1 - \frac{n}{2} \\
        \|\nabla_{\text{ref}}^{\ell} B\|_{p} &\leq C_{p, \Omega_1} \|B\|_{H^{\ell + \lceil\frac{n}{2}\rceil}}, \quad 1 \leq p < \infty. \nonumber
    \end{align} By integrality of \(\ell_i\) and \(r\), \(\ell_i \leq r-1 - \frac{n}{2}\) implies \(\ell_i \leq r-1 - \lceil\frac{n}{2}\rceil\). Hence
    \begin{align*}
         \prod_{\ell_i \leq r-1-\frac{n}{2}} \|\nabla_{\text{ref}}^{(\ell_i)} A\|_{p_r} \prod_{\ell_i >r-1-\frac{n}{2}} \|\nabla_{\text{ref}}^{(\ell_i)} A\|_{\frac{2n}{n-2(r-1-\ell_i)}} \leq C_{r,  \Omega_1}\left(1 + \|A\|_{H^{r-1}}^r\right).
    \end{align*} Taking into account (\ref{Hkestimates:onlyoneterm}), the bound (\ref{Hkestimates:highinductionestimate}) now follows. \par
    We finally prove the desired estimate (\ref{Hkestimates:Hkbounds}), first with \(t_i = t_1\). Any dependence on \(\Omega_1\) is a dependence on \(\Omega\). Then for \(t \in [t_1, t_2]\)
    \begin{align*}
        \sup_{t_1 \leq t \leq t_2} \|A(t)\|_{q_0} &\leq \|A(t_1)\|_{q_0} + \sup_{t_1 \leq t \leq t_2} \|A(t) - A(t_1)\|_{q_0} \\
        &\leq \|A(t_1)\|_{X_0} + \sup_{t_1 \leq t \leq t_2} \|A(t) - A(t_1)\|_{q_0} \\ 
        &\leq \|A(t_1)\|_{X_0} + C_{\tau, K, \Omega} \rho. 
    \end{align*} Suppose that for some \(r\), with \(0 \leq r \leq \frac{n}{2}-2\), we have the estimate
    \begin{align*}
        \sup_{t_1 \leq t \leq t_2} \|A(t)\|_{X_{r}} \leq C_{r, \tau, K, \Omega}\left(1+\rho^{(r+1)!} \right) \left(1+\|A(t_1)\|_{X_{r}}^{r!} \right).
    \end{align*} Then by (\ref{Hkestimates:lowinductionestimate})
    \begin{align*}
        \sup_{t_1 \leq t \leq t_2} \|\nabla_{\text{ref}}^{(r+1)} A(t)\|_{q_{r+1}} &\leq \|\nabla_{\text{ref}}^{(r+1)} A(t_1)\|_{q_{r+1}} + \sup_{t_1 \leq t \leq t_2} \|\nabla_{\text{ref}}^{(r+1)}(A(t) - A(t_1))\|_{q_{r+1}} \\
    &\leq \|A(t_1)\|_{X_{r+1}} + \sup_{t_1 \leq t \leq t_2} \|\nabla_{\text{ref}}^{(r+1)}(A(t) - A(t_1))\|_{q_{r+1}} \\
    &\leq \|A(t_1)\|_{X_{r+1}} + C_{r+1, \tau, K, \Omega} (1 + \sup_{t_1 \leq t \leq t_2} \|A(t)\|_{X_{r}}^{r+1})\rho  \\
    &\leq \|A(t_1)\|_{X_{r+1}} + C_{r+1, \tau, K, \Omega} (1 + ((1+\rho^{(r+1)!})(1+\|A(t_1)\|_{X_{r}}^{r!}))^{r+1})\rho  \\
    &\leq C_{r+1, \tau, K, \Omega}\left(1+\rho^{(r+2)!} \right)\left(1+\|A(t_1)\|_{X_{r+1}}^{(r+1)!} \right).
    \end{align*} Thus by induction
    \begin{align*}
        \sup_{t_1 \leq t \leq t_2} \|A(t)\|_{X_{\lfloor\frac{n}{2}\rfloor-1}} \leq C_{\tau, K, \Omega} \left(1 + \rho^{\lfloor\frac{n}{2}\rfloor!} \right) \left(1 + \|A(t_1)\|_{X_{\lfloor\frac{n}{2}\rfloor-1}}^{(\lfloor\frac{n}{2}\rfloor-1)!} \right).
    \end{align*} Observe that \(q_r \geq 2\) for all \(r\). Hence
    \begin{align*}
        \|B\|_{H^{\lfloor\frac{n}{2}\rfloor-1}} \leq C_\Omega \|B\|_{X_{\lfloor\frac{n}{2}\rfloor-1}}.
    \end{align*} On the other hand, the Sobolev inequalities (\ref{Hkestimates:sobolevinequalities}) imply
    \begin{align*}
        \|B\|_{X_{\lfloor\frac{n}{2}\rfloor-1}} \leq C_\Omega \|B\|_{H^{\lceil\frac{n}{2}\rceil-1}}.
    \end{align*} Hence
    \begin{align*}
        \sup_{t_1 \leq t \leq t_2} \|A(t)\|_{H^{\lfloor\frac{n}{2}\rfloor-1}} \leq C_{\tau, K, \Omega} \left(1+\rho^{\lfloor\frac{n}{2}
    \rfloor!} \right) \left(1 + \|A(t_1)\|_{H^{\lceil\frac{n}{2}\rceil-1}}^{(\lfloor\frac{n}{2}\rfloor-1)!} \right).
    \end{align*} The bound (\ref{Hkestimates:Hkbounds}) now follows by a similar induction as just carried out, using (\ref{Hkestimates:highinductionestimate}) in place of (\ref{Hkestimates:lowinductionestimate}), combined with the obvious bounds \(\|B\|_{W^{\ell, s}(\Omega)} \leq \|B\|_{W^{\ell, s}(\Omega_1)} \leq \|B\|_{W^{\ell, s}(M)}\). The case \(i = 2\) follows by an essentially identical argument, starting from
    \begin{align*}
        \sup_{t_1 \leq t \leq t_2} \|A(t)\|_{q_0} \leq \|A(t_2)\|_{q_0} + \sup_{t_1 \leq t \leq t_2} \|A(t_2) - A(t)\|_{q_0}.
    \end{align*}
\end{proof}

\vspace{5mm}

\subsection{Proof of item (\ref{wellposedness:gaugeequivalenttosmooth}) of Theorem \ref{thm:wellposedness}} \label{gaugeeqtosmth}
    We will show the existence of a gauge transformation as claimed in item (\ref{wellposedness:gaugeequivalenttosmooth}) via a compactness argument. \par
   Since \(C^\infty(\mathcal{A}_E)\) is dense in \(H^q(\mathcal{A}_E)\), there exists a sequence of smooth connections \(A_i\) converging to \(A_0\) in \(H^q\). By item (\ref{wellposedness:wellposedness}) of Theorem \ref{thm:wellposedness}, we may assume that the solution \(\tilde{A}_i(t)\) of (\ref{ymf}) with \(\tilde{A}_i(0) = A_i\) exists on \([0, \tau]\), and moreover that the \(\tilde{A}_i(t)\) converge in \(C^0([0, \tau], H^q)\) to \(A(t)\). In particular,
    \begin{align*}
        \sup_{i, t \in [0, \tau]} \|\tilde{A}_i(t)\|_{H^q} < \infty,
    \end{align*} and hence
    \begin{align*}
        \sup_{i, t \in [0, \tau]} \|F_i(t)\|_{H^{q-1}} < \infty.
    \end{align*}Since \(q > \frac{n}{2}-1\), \(H^{q-1} \inj L^p\) for some \(p \in (\frac{n}{2}, \infty]\). Given a uniform \(L^p\) bound on the curvature for such \(p\), $\varepsilon$-regularity, e.g. \cite[Lemma 2.2]{waldronuhlenbeck}, yields that for each \(\delta > 0\)
    \begin{align}\label{gaugeequivalenttosmoothprop:curvaturebound}
        \sup_{i, t \in [\delta, \tau]} \|F_i(t)\|_\infty < \infty.
    \end{align} Moreover, since \(\tau \leq 1\), the global energy identity and the convergence of \(A_i \to A_0\) yields
    \begin{align}\label{gaugeequivalenttosmoothprop:tensionbound}
        \sup_i \int_0^{\tau} \|D_i^\ast F_i(t) \| \, dt < \infty.
    \end{align}Then by standard strong Uhlenbeck compactness arguments for Yang-Mills flow, e.g. \cite[Theorem 1.3]{waldronuhlenbeck}, we can pass to a subsequence, also labeled by \(i\), such that there exist smooth gauge transformations \(u_i\) for which \(u_i(\tilde{A}_i(\tau))\) converge in \(C^\infty\) to some smooth connection. \par 
    Let \(t_m := \frac{\tau}{m}\). Note that the bounds (\ref{gaugeequivalenttosmoothprop:curvaturebound}) and (\ref{gaugeequivalenttosmoothprop:tensionbound}) are gauge-invariant, and so they also hold for the \(u_i(\tilde{A}_i(t))\). Since \(u_i(\tilde{A}_i(\tau))\) is smoothly convergent, it follows from Theorem \ref{thm:Hkestimates}, with \(t_2 = \tau\) and \(t_i = t_2\) in the statement, that for each \(m\) and for each \(k \geq \lceil\frac{n}{2}\rceil\)
    \begin{align}\label{gaugeequivalenttosmoothprop:Hkbounds}
        \sup_{i, t_m \leq t \leq \tau} \|u_i(\tilde{A}_i(t))\|_{H^k} &\leq \sup_{i, t_m \leq t \leq \tau} \left(\|u_i(\tilde{A}_i(\tau))\|_{H^k} + \|u_i(\tilde{A}_i(t)) - u_i(\tilde{A}_i(\tau))\|_{H^k} \right) < \infty.
    \end{align} Now, we may exchange time derivatives of \(u_i(\tilde{A}_i)\) with spatial derivatives using the flow equation. Hence, (\ref{gaugeequivalenttosmoothprop:Hkbounds}) and Sobolev embedding yields that for each \(m\) and each \(k\)
    \begin{align*}
        \sup_i \|u_i(\tilde{A}_i)\|_{C^k(M \times [t_m, \tau])} < \infty.
    \end{align*}
    Thus for each \(m\), Arzela-Ascoli implies that we may pass to a subsequence, also labeled by \(i\), such that \(u_i(\tilde{A}_i)\) converges in \(C^\infty(M \times [t_m, \tau])\) to a family of connections \(B_m(t)\), which is also a solution of (\ref{ymf}). Note that \(B_m(t) = B_n(t)\) on their common interval of definition by uniqueness of solutions to (\ref{ymf}). Passing to a subsequence, \(u_i(\tilde{A}_i)\) converge in \(C_{\text{loc}}^\infty(M \times (0, \tau])\) to a smooth solution \(B(t)\) of Yang-Mills flow. \par
    Finally, since the \(\tilde{A}_i(\tau)\) themselves converge in \(H^q\), we have as in \cite[\textsection 2.3.17]{donkron} that the \(u_i\) converge to some gauge transformation \(u_\infty\) in \(H^{q+1}\). Then since
    \begin{align*}
        u_i(\tilde{A}_i(t)) - u_\infty(A(t)) = u_i(\tilde{A}_i(t) - A(t)) + (u_i - u_\infty)(A(t)),
    \end{align*} we deduce that \(B(t) = u_\infty(A(t))\) for \(t \in (0, \tau]\). By the uniqueness of solutions to (\ref{ymf}), the same holds for all \(t\) for which \(A(t)\) is defined, and \(B(t)\) remains smooth. Thus \(u_\infty\) satisfies the conclusions of item (\ref{wellposedness:gaugeequivalenttosmooth}) of Theorem \ref{thm:wellposedness}, as desired. \qed

\vspace{5mm}

\subsection{Proof of item (\ref{wellposedness:blowupcharacterization}) of Theorem \ref{thm:wellposedness}} \label{ltecriterion}
    Suppose for 
    contradiction that
    \begin{align*}
        \limsup_{t \nearrow T} \|F(t)\|_\infty < \infty.
    \end{align*} By item (\ref{wellposedness:gaugeequivalenttosmooth}) of Theorem \ref{thm:wellposedness}, we may gauge transform \(A(t)\) so that \(u(A(t))\) is smooth. The pointwise norm of the curvature is gauge-invariant, so we still have \(\limsup_{t \nearrow T}\ \|F_{u(A(t))}\|_\infty < \infty\). It then follows from Theorem \ref{thm:Hkestimates} and finiteness of \(T\) that for each \(k \geq 0\), the \(H^k\) norm of \(u(A(t))\) remains bounded as \(t \nearrow T\). Thus the \(H^q\) norm of \(A(t)\) remains bounded as \(t \nearrow T\), so by Theorem \ref{thm:shorttimeexistence}, \(A(t)\) exists on \([0, T + \varepsilon]\) for some \(\varepsilon > 0\). This contradicts the maximality of \(T\), as desired. \qed


\vspace{10mm}

\section{Continuity at infinite time}

Concerning infinite-time behavior, we will prove the following standard result, cf. \cite[Prop. 7.4]{rade}. See also Kozono-Maeda-Naito \cite{kozonomaedaasymptoticstability, kozonomaedanaitostablemanifold} for early work on asymptotic stability of Yang-Mills flow.
\begin{thm}\label{thm:continuityatinfinitetime}
    Let \(q\) be such that \(q \geq 1\) and \(q > \lceil\frac{n}{2}-1\rceil\). 
    Suppose that \(A_0 \in H^q(\mathcal{A}_E)\) is such that the solution $A(t)$ of (\ref{ymf}) starting at \(A_0\) exists for all time and converges in \(H^q\) to a Yang-Mills connection \(A_\infty =: A(\infty)\) as \(t \to \infty\). There exists \(\delta_\infty > 0 \) (depending on \(A_\infty, q,\) and \(M\)) such that given any $\eps > 0,$ there exists \(\delta > 0\) (depending on \(A_0, q,\) and \(M\)) as follows.
    
    Let \(A'_0 \in B_{\delta}(A_0)\) and consider the solution $A'(t)$ of (\ref{ymf}) with $A'(0) = A_0'.$ Then exactly one of the following alternatives holds:
    \begin{enumerate}
        \item There exists \(0 < \tau < \infty\) such that $A'(t)$ is defined on $\LB 0 , \tau \RB$ and \(\mathcal{YM}(A'(\tau)) \leq \mathcal{YM}(A_\infty) - \delta_\infty\), or
        \item The solution \(A'(t)\) exists for all time and converges in \(H^q\) to a Yang-Mills connection $A'(\infty)$ as \(t \to \infty\), with \begin{align}\label{continuityatinfinitetime:continuityestimate}
            \| A'(t) - A(t) \|_{H^q} < \eps
        \end{align}
        for all $0 \leq t \leq \infty.$ 
    \end{enumerate}
\end{thm}

In \textsection \ref{lojasiewiczineq}, we first describe the setup of an $H^{-1}$ \L ojasiewicz inequality in general dimension, which we then prove in Lemma \ref{lemma:lojasiewicz}. We also specialize to the case of flat connections in Lemma \ref{lemma:uniformlojasiewicz}. Theorem \ref{thm:continuityatinfinitetime} is proven in \textsection \ref{continuityatinftime} using Lemma \ref{lemma:lojasiewicz}.

\vspace{5mm}

\subsection{\L ojasiewicz inequalities} \label{lojasiewiczineq} We will use a variant of the \L ojasiewicz-Simon inequality proved by R\aa de in dimension less than four and by B. Yang \cite[Lemma 12]{byanguniqueness} in general dimension. First we describe our setup. In the setup, \(C\) denotes a constant, depending only on \(M\), which may increase from line to line. \par
Given an \(H^q\) connection \(A\), with \(q > \frac{n}{2}-1\) and \(q \geq 1\), we define the following quantities, cf. \cite[(2.10) in Def. 2.3]{taubespathconnected}. For \(b \in H^1\left(T^\ast M \otimes \mathfrak{g}_E\right)\), set
\begin{align*}
    \|b\|_{H_A^1} &:= \sqrt{\left\|\nabla_A b\right\|^2 + \|b\|^2} \\
    \mathcal{V}(A) &:= \sup_{\substack{b \in H^1\left(T^\ast M \otimes \mathfrak{g}_E\right), \\ b \neq 0}} \frac{\left|\int \langle F_A, D_A b\rangle\right|}{\|b\|_{H_A^1}}.
\end{align*} Note that \(\|b\|_{H_A^1}\) and the expression in the above supremum are finite since, by Sobolev multiplication, \(D_A\) and \(\nabla_A\) map \(H^1\) to \(L^2\). Moreover, the supremum itself is finite by Cauchy-Schwarz and the pointwise domination \(|D_A b| \lesssim |\nabla_A b|\). \par
Moreover, we claim that \(\mathcal{V}\) is continuous with respect to the \(H^q\) topology.
\begin{claimproof}
    Let \(A_1, A_2 \in H^q(\mathcal{A}_E)\). We first establish equivalence of \(\|\cdot\|_{H_{A_1}^1}\) and \(\| \cdot \|_{H_{A_2}^1}\) if \(A_1, A_2\) are close in \(H^q\). Note that
\begin{align*}
   | \|\nabla_{A_1} b\|^2 - \|\nabla_{A_2} b\|^2| &\leq (\|\nabla_{A_1} b\| + \|\nabla_{A_2} b\|)\|(A_1-A_2) \# b\| \\
        &\leq C\left(\|b\|_{H_{A_1}^1} + \|b\|_{H_{A_2}^1}\right)\|A_1-A_2\|_n \|b\|_{\frac{2n}{n-2}}.
\end{align*} By Sobolev embedding and the Kato inequality trick, see e.g. \cite[(1.16)]{instantons},
\begin{align*}
    \|b\|_{\frac{2n}{n-2}} \leq C\min\left\{\|b\|_{H_{A_1}^1}, \|b\|_{H_{A_2}^1}\right\}.
\end{align*} Hence for \(\|A_1 - A_2\|_{H^q}\) sufficiently small such that, by Sobolev embedding, \(C^2\|A_1 - A_2\|_n \leq \varepsilon < 1\),
\begin{align}\label{lojasiewiczclaim:normequivalence}
    \frac{1-\varepsilon}{1+\varepsilon}\|b\|_{H_{A_2}^1}^2 \leq \|b\|_{H_{A_1}^1}^2 \leq \frac{1+\varepsilon}{1-\varepsilon}\|b\|_{H_{A_2}^1}^2.
\end{align} Thus we have equivalence of the two norms. \par
 We next show \begin{align}\label{lojasiewiczclaim:D*Fequivalence}
    \left|\int \langle F_{A_1}, D_{A_1} b\rangle - \int \langle F_{A_2}, D_{A_2} b\rangle\right| \leq \varepsilon\min\left\{\|b\|_{H_{A_1}^1}, \|b\|_{H_{A_2}^1}\right\}.
\end{align} We compute
\begin{align*}
    \int \langle F_{A_1}, D_{A_1} b\rangle - \int \langle F_{A_2}, D_{A_2} b\rangle & = \text{I} + \text{II} + \text{III} + \text{IV} + \text{V},
\end{align*}where  
\begin{align*}
    \text{I} &= \int \langle F_{A_2}, (A_1 - A_2)\wedge b\rangle \\
    \text{II} &= \int \langle D_{A_2}(A_1 - A_2), D_{A_2} b\rangle \\
    \text{III} &= \int \langle D_{A_2}(A_1-A_2), (A_1-A_2) \wedge b\rangle \\
    \text{IV} &= \int \langle (A_1 - A_2) \wedge (A_1-A_2), D_{A_2} b\rangle \\
    \text{V} &= \int \langle (A_1 - A_2) \wedge (A_1-A_2), (A_1 - A_2) \wedge b\rangle.
\end{align*} We estimate
\begin{align*}
    |\text{I}| &\leq C\|F_{A_2}\|\|A_1 - A_2\|_n\|b\|_{\frac{2n}{n-2}} \\
    |\text{II}| &\leq C(\|\nabla_{\text{ref}}(A_1 - A_2)\| + \|A_2\|_4\|A_1 - A_2\|_4)\|b\|_{H_{A_2}^1} \\
    |\text{III}| &\leq C(\|\nabla_{\text{ref}}(A_1 - A_2)\| + \|A_2\|_4\|A_1 - A_2\|_4)\|A_1 -A_2\|_n\|b\|_{\frac{2n}{n-2}} \\
    |\text{IV}| &\leq C\|A_1 - A_2\|_4^2\|b\|_{H_{A_2}^1} \\
    |\text{V}| &\leq C\|A_1 - A_2\|_4^2\|A_1 - A_2\|_n\|b\|_{\frac{2n}{n-2}}.
\end{align*} Since \(q \geq 1\) and \(q > \frac{n}{2}-1\), we have 
\begin{align*}
    &\|F_{A_2}\| + \|A_2\|_4 \leq C_{q, M}\|A_2\|_{H^q} \\
    &\|\nabla_{\text{ref}}(A_1 - A_2)\| + \|A_1 - A_2\|_4 + \|A_1 - A_2\|_n \leq C_{q, M} \|A_1 - A_2\|_{H^q}.
\end{align*}Hence if \(\|A_1 - A_2\|_{H^q}\) is small enough,
\begin{align*}
    |\text{I}| + |\text{II}| + |\text{III}| + |\text{IV}| + |\text{V}| \leq \varepsilon\|b\|_{H_{A_2}^1}.
\end{align*} Thus by symmetry of the preceding computations
\begin{align*}
    \left|\int \langle F_{A_1}, D_{A_1} b\rangle - \int \langle F_{A_2}, D_{A_2} b\rangle\right| \leq \varepsilon\min\left\{\|b\|_{H_{A_1}^1}, \|b\|_{H_{A_2}^1}\right\},
\end{align*} which is the bound (\ref{lojasiewiczclaim:D*Fequivalence}) that we wanted to show. \par
Continuity of \(\mathcal{V}\) now readily follows from (\ref{lojasiewiczclaim:normequivalence}) and (\ref{lojasiewiczclaim:D*Fequivalence}), so the claim holds, as desired.
\end{claimproof}
\par We further have that \(\mathcal{V}\) is invariant under the action of \(H^{q+1}\) gauge transformations, which follows from
\begin{align*}
    \frac{\left|\int \langle F_{u(A)}, D_{u(A)} b\rangle\right|}{\left(\left\|\nabla_{u(A)} b\right\|^2 + \|b\|^2\right)^{\frac{1}{2}}} = \frac{\left|\int \langle u(F_A), u(D_A u^{-1}(b))\rangle\right|}{\left(\left\|u\left(\nabla_A u^{-1}(b)\right)\right\|^2 + \|b\|^2\right)^{\frac{1}{2}}} = \frac{\left|\int \langle F_A, D_A u^{-1}(b)\rangle\right|}{\left(\left\|\nabla_A u^{-1}(b)\right\|^2 + \|u^{-1}(b)\|^2\right)^{\frac{1}{2}}}.
\end{align*} Lastly, if we also have \(F_A \in H^q\), whence \(D_A^\ast F_A \in H^{q-1} \subset L^2\), then note that
\begin{align*}
    \mathcal{V}(A) \leq \left\|D_A^\ast F_A\right\|.
\end{align*}\par
We now give our variant of the R{\aa}de-Yang-Simon {\L}ojasiewicz inequality.
\begin{lemma}\label{lemma:lojasiewicz}
Let \(q > \frac{n}{2}-1\) and \(q \geq 1\). Given a fixed Yang-Mills connection $A_1 \in H^q(\calA_E),$ there exists a gauge-invariant $H^q$-neighborhood \(U\) of \(A_1\) and a constant \(\theta \in \left(0, \frac{1}{2}\right)\) 
such that for \(A \in U\),
\begin{align}\label{lojasiewicz:lojasiewiczinequality}
    \left|\mathcal{YM}(A) - \mathcal{YM}(A_1)\right|^{1 - \theta} \leq \mathcal{V}(A).
\end{align}
\end{lemma}
\begin{proof}
This is a straightforward adaptation of the arguments of \cite[Prop. 7.2]{rade} and \cite[Lemma 12]{byanguniqueness}, based on \cite[Theorem 3]{leonsimon}, which we include for the sake of completeness.
In the sequel, \(C\) denotes a positive constant, whose dependencies include \(A_1, E, q\), and the geometry of \(M\), which may increase from line to line. \par
    
    We first reduce proving (\ref{lojasiewicz:lojasiewiczinequality}) to a simpler setting. Since the desired inequality (\ref{lojasiewicz:lojasiewiczinequality}) is gauge-invariant, and since \(H^q\) Yang-Mills connections are smooth modulo \(H^{q+1}\) gauge, we may WLOG assume that \(A_1\) is smooth. Furthermore, the standard Coulomb gauge-fixing argument, e.g. \cite[Prop. 2.3.4]{donkron}, implies that there is an \(H^q\) neighborhood \(U\) of \(A_1\) such that if \(A \in U\), then \(A\) is gauge-equivalent to a connection $B$  which satisfies \(D_{A_1}^\ast(B - A_1) = 0\). Again by gauge-invariance, it suffices to prove (\ref{lojasiewicz:lojasiewiczinequality}) for \(A\) satisfying the gauge-fixing condition. Finally, by continuity of \(\mathcal{V}\), there exists an \(H^q\)-neighborhood of \(A_1\) such that for \(A\) in the neighborhood
    \begin{align*}
        \|D_A^\ast F_A\|_{H^{-1}} \leq 2 \mathcal{V}(A),
    \end{align*} so it suffices to prove (\ref{lojasiewicz:lojasiewiczinequality}) with \(\|D_A^\ast F_A\|_{H^{-1}}\) in place of \(\mathcal{V}(A)\). \par
    Now that we have reduced the proof of (\ref{lojasiewicz:lojasiewiczinequality}) to an easier setting, we next carry out the Lyapunov-Schmidt reduction step of Simon's argument. In the sequel, we change our reference connection to be the smooth connection \(A_1\), and hence define Sobolev norms with respect to \(A_1\). Since \(M\) is compact, the norms induced by \(\nabla_{\text{ref}}\) and \(\nabla_{A_1}\) are equivalent. \par 
    By Sobolev multiplication, the mapping \(a \mapsto \mathcal{YM}(A)\), where \(\nabla_A = \nabla_{A_1} + a\), is a well-defined, polynomial function on \(H^q\). 
    We define maps
    \begin{align*}
        M &: H^q \to H^{q-2} \\
        M(a) &:= D_A^\ast F_A, \\
        \tilde{M} &: H^q \to H^{q-2} \\
        \tilde{M}(a) &:= D_A^\ast F_A + D_{A_1}D_{A_1}^\ast a \\
        L_A &: H^q \to H^{q-2} \\
        L_A(b) &:= D_A^\ast D_Ab  + D_{A_1} D_{A_1}^\ast b+ (-1)^{3n+1}\ast (b \wedge \ast F_A).
    \end{align*}  Observe that \(L_A\) is the linearization of \(\tilde{M}\) at \(a\), i.e.,
    \begin{align*}
        L_A = (\nabla \tilde{M})(a).
    \end{align*}Now, \(L_{A_1}\) is self-adjoint and elliptic, so by elliptic theory
    \begin{align*}
        \mathcal{H} := \text{ker}(L_{A_1})
    \end{align*} is of finite dimension and consists of smooth sections. Define 
    \begin{align*}
        \Pi &: H^q \to \mathcal{H} \\
        \Pi(a) &:= \text{proj}_{\mathcal{H}}^{L^2}(a) \\
        N &: H^q \to H^{q-2} \\
        N &:= \tilde{M} + \Pi.
    \end{align*} Note that the linearization of \(N\) at \(a = 0\), i.e \((\nabla N)(0)\), equals \(L_{A_1} + \Pi\), which is invertible. Thus as in 
    \cite{leonsimon}, there exist open sets \(U_1 \subset H^q, U_2 \subset H^{q-2}\), with \(0 \in U_i\), and a diffeomorphism
    \begin{align*}
        \Psi : U_2 \to U_1
    \end{align*} which is inverse to \(N\). Next, set
    \begin{align*}
        V := \text{ker}(D_{A_1}^\ast) \subset H^q,
    \end{align*} and let \(\psi_1, \dots, \psi_m\) be an \(L^2\)-orthonormal basis of \(\mathcal{H} \cap V\). For \(\delta > 0\) small enough, we have a real-analytic \cite[(2.12)]{leonsimon} function
    \begin{align*}
        F &: B_\delta(0) \subset \R^m \to \R \\
        F(\xi) &= \mathcal{YM}\left(A_1 + \Psi\left(\sum_{j = 1}^m \xi^j \psi_j\right)\right).
    \end{align*} This completes the Lyapunov-Schmidt reduction step. \par
    Now that we have the existence of the local inverse \(\Psi\) of \(N\), we next show the following smoothing estimates for \(\Psi\), which are analogues of \cite[(2.9) and (2.10)]{leonsimon}: for \(a, b \in U_1\)
    \begin{align}
        \|\Psi(a) - \Psi(b)\|_{H^1} &\leq C\|a - b\|_{H^{-1}}, \label{lojasiewicz:PsiH1Lipschitz} \\
        \|\Psi(a) - \Psi(b)\|_{H^q} &\leq C\|a - b\|_{H^{q-2}}. \label{lojasiewicz:PsiHqLipschitz}
    \end{align} We first consider (\ref{lojasiewicz:PsiH1Lipschitz}). By elliptic theory we have for \(v \in H^q\)
    \begin{align*}
        \|v\|_{H^1} \leq C\|(L_{A_1}+\Pi)v\|_{H^{-1}}.
    \end{align*} Then by Sobolev multiplication, for a given \(\varepsilon > 0\), we may shrink \(U_1, U_2\) so that for \(a \in U_1\) and \(v \in H^q\)
    \begin{align*}
        \|((\nabla \tilde{M})(a) - L_{A_1})(v)\|_{H^{-1}} &\leq C\|(\nabla_{A_1} a) \# v + a \# \nabla_{A_1} v + a \# a \# v\|_{H^{-1}} \\
        &\leq C(\|\nabla_{A_1} a\|_{H^{q-1}} + \|a \# a\|_{H^{q-1}})\|v\|_{H^1} + C\|a\|_{H^q}\|\nabla_{A_1} v\|_{H^0} \\
        &\leq C(\|a\|_{H^{q}} + \|a\|_{H^{q}}^2)\|v\|_{H^1} + C\|a\|_{H^q}\|v\|_{H^1} \\
        &\leq \varepsilon \|v\|_{H^1}.
    \end{align*} Hence
    \begin{align*}
         \|(L_{A_1}+\Pi)v\|_{H^{-1}} &\leq \|((\nabla N)(a))v\|_{H^{-1}} + \|((\nabla \tilde{M})(a) - L_{A_1})(v)\|_{H^{-1}} \\
         &\leq \|((\nabla N)(a))v\|_{H^{-1}} + \frac{1}{2C}\|v\|_{H^1}.
    \end{align*} We deduce that for \(w \in H^{q-2}\)
    \begin{align*}
        \|((\nabla N)(a))^{-1} w\|_{H^1} \leq C\|w\|_{H^{-1}}.
    \end{align*} We may further shrink \(U_1, U_2\) so that \(U_2\) is an open ball about \(0\), which is convex and hence \(\Psi\) is defined on \(a + t(b-a)\) for \(a, b \in U_2, t\in [0, 1]\). Then
    \begin{align*}
        \|\Psi(a) - \Psi(b)\|_{H^1} &\leq \int_0^1 t\|((\nabla \Psi)(a + t(b-a)))(b-a)\|_{H^1} \, dt  \\
        &= \int_0^1 t\|((\nabla N)(\Psi(a + t(b-a))))^{-1}(b-a)\|_{H^1} \, dt \\
        &\leq C\|b-a\|_{H^{-1}}.
    \end{align*} Thus (\ref{lojasiewicz:PsiH1Lipschitz}) holds. \par 
    The other smoothing estimate (\ref{lojasiewicz:PsiHqLipschitz}) follows from a similar Sobolev multiplication argument. Thus we have the desired smoothing estimates for \(\Psi\). \par 
    Now that we have smoothing estimates for \(\Psi\), we proceed with proving the {\L}ojasiewicz inequality, which 
    we only need to do for \(a \in U_1 \cap V\). We will do this, as in \cite{leonsimon}, by establishing the following two inequalities: given \(a \in U_1 \cap V\), upon which we have \(\xi \in B_{\delta}(0)\) defined by \(\Pi(a) = \sum_{j = 1}^m \xi^j \psi_j\),
    \begin{align}
        &|\mathcal{YM}(A_1 + a) - \mathcal{YM}(A_1 + \Psi(\Pi(a))| \leq C\|M(a)\|_{H^{-1}}^2 \label{lojasiewicz:upperbound} \\
        &|(\nabla F)(\xi)| \leq C\|M(a)\|_{H^{-1}} \label{lojasiewicz:lowerbound}.
    \end{align} \par
    We first consider (\ref{lojasiewicz:upperbound}). We compute as in as in \cite[(2.31)]{leonsimon}
    \begin{align*}
        &|\mathcal{YM}(A_1+\Psi(\Pi(a))) - \mathcal{YM}(A_1 + a)| \\
        = &\left| \int_0^1 \frac{d}{dt}\mathcal{YM}(A_1 + a + t(\Psi(\Pi(a)) - a)) \, dt \right| \\
        = & \left| \int_0^1 \int_M \langle M(a + t(\Psi(\Pi(a)) - a)), \Psi(\Pi(a)) - a\rangle \, dt\right| \\\leq &\int_0^1 \|M(a + t(\Psi(\Pi(a))-a)\|_{H^{-1}} \|\Psi(\Pi(a)) - a\|_{H^1} \, dt \\
        = &\int_0^1 \|M(a + t(\Psi(\Pi(a))-\Psi(N(a)))\|_{H^{-1}} \|\Psi(\Pi(a)) - \Psi(N(a))\|_{H^1} \, dt.
    \end{align*} Denoting
    \begin{align*}
        b = t(\Psi(\Pi(a)) - \Psi(N(a))),
    \end{align*} we next estimate \(\|M(a+b)\|_{H^{-1}}\). We compute
    \begin{align*}
        M(a+b) &= D_A^\ast F_A + (D_A^\ast D_A b + \nabla_A b \# b + F_A \# b + b \# b \# b) \\
        &= M(a) \\
        &+ (D_{A_1}^\ast D_{A_1} b + (\nabla_{A_1} a) \# b + a\#\nabla_{A_1} b + a \# a \# b + (\nabla_{A_1} b) \# b + a \# b \# b + F_A \# b + b \# b \# b).
    \end{align*} Since \(q > \frac{n}{2}-1\) and \(q \geq 1\), we may estimate the \(H^{-1}\) norm of the terms in the parentheses using Sobolev multiplication to obtain
    \begin{align*}
        \|D_{A_1}^\ast D_{A_1} b\|_{H^{-1}} &\leq C\|b\|_{H^1} \\
        \|(\nabla_{A_1} a) \# b \|_{H^{-1}} &\leq C\|\nabla_{A_1} a\|_{H^{q-1}}\|b\|_{H^1} \leq C\|a\|_{H^q}\|b\|_{H^1} \\
        \|a \# \nabla_{A_1} b\|_{H^{-1}} &\leq C\|a\|_{H^q} \|\nabla_{A_1} b\|_{H^0} \leq C\|a\|_{H^q}\|b\|_{H^1} \\
        \|a \# a \# b\|_{H^{-1}} &\leq C\|a \# a\|_{H^{q-1}} \|b\|_{H^1} \leq C\|a\|_{H^q}^2\|b\|_{H^1} \\
        \|(\nabla_{A_1} b) \# b\|_{H^{-1}} &\leq C\|b\|_{H^q} \|b\|_{H^1} \\
        \|a \# b \# b\|_{H^{-1}} &\leq C\|a\|_{H^q}\|b\|_{H^q}\|b\|_{H^{-1}} \\
        \|F_A \# b\|_{H^{-1}} &\leq C\|F_A\|_{H^{q-1}} \|b\|_{H^1} \leq C(1 +  \|a\|_{H^q}^2)\|b\|_{H^1} \\
        \|b \# b \# b\|_{H^{-1}} &\leq C\|b\|_{H^q}^2\|b\|_{H^1}.
    \end{align*}Thus if \(\|a\|_{H^q} + \|b\|_{H^q} \leq 1\),
    \begin{align}\label{lojasiewicz:D*FH-1continuous}
        \|M(a + b) - M(a)\|_{H^{-1}} \leq C\|b\|_{H^1}.
    \end{align} Now, we can make \(\|a\|_{H^q}\) small by shrinking \(U_1, U_2\). Moreover, since we assumed \(a \in V\), we have \(M(a) = \tilde{M}(a) = N(a) - \Pi(a)\), so the smoothing estimate (\ref{lojasiewicz:PsiHqLipschitz}) yields
    \begin{align*}
        \|b\|_{H^q} \leq C\|\Pi(a) - N(a)\|_{H^{q-2}} = C\|M(a)\|_{H^{q-2}}.
    \end{align*} Thus, since \(A_1\) is Yang-Mills, we can make \(\|M(a)\|_{H^{q-2}}\) small by shrinking \(U_1, U_2\). We may assume \(\|a\|_{H^q} + \|b\|_{H^q} \leq 1\) and consequently the bound
    \begin{align*}
        \|M(a + b)\|_{H^{-1}} \leq C(\|M(a)\|_{H^{-1}} + \|b\|_{H^1}).
    \end{align*} Therefore, the other smoothing estimate (\ref{lojasiewicz:PsiH1Lipschitz}) yields
    \begin{align*}
        |\mathcal{YM}(A_1 + \Psi(\Pi(a))) - \mathcal{YM}(A_1+a)| &\leq
        C(\|M(a)\|_{H^{-1}} + \|b\|_{H^1})\|\Psi(\Pi(a)) - \Psi(N(a))\|_{H^1} \\
        &\leq C(\|M(a)\|_{H^{-1}} + \|M(a)\|_{H^{-1}})\|M(a)\|_{H^{-1}}\\
        &= C\|M(a)\|_{H^{-1}}^2.
    \end{align*} This establishes (\ref{lojasiewicz:upperbound}). \par 
    We next establish (\ref{lojasiewicz:lowerbound}). Given \(\xi \in B_\delta(0) \subset \R^m\), set \(\iota(\xi) = \sum_{j = 1}^m \xi^j \psi_j\). Observe that
    \begin{align*}
        |(\partial_i F)(\xi)| &= \left|\int \langle M(\Psi(\iota(\xi))), \xi^i((d\Psi)(\iota(\xi)))(\psi_i)\rangle\right| \\
        &\leq |\xi|\|M(\Psi(\iota(\xi)))\|_{H^{-1}} \|((d\Psi)(\iota(\xi)))(\psi_i)\|_{H^1} \\
        &\leq C|\xi|\|M(\Psi(\iota(\xi)))\|_{H^{-1}}.
    \end{align*} Now, \(a \in V\), then \(\Pi(a) \in \mathcal{H} \cap V\), so there exists \(\xi\) such that \(\iota(\xi) = \Pi(a)\). We compute as in \cite[(2.20)]{leonsimon} that
    \begin{align*}
        |(\nabla F)(\xi)| &\leq C\|M(\Psi(\Pi(a))\|_{H^{-1}} \\
        &\leq C(\|M(a)\|_{H^{-1}}  + \|M(\Psi(\Pi(a))) - M(a)\|_{H^{-1}}).
    \end{align*} By (\ref{lojasiewicz:D*FH-1continuous}), we have that if 
    \begin{align*}
        \|a\|_{H^q} + \|\Psi(\Pi(a)) - a\|_{H^q} \leq 1,
    \end{align*} then
    \begin{align*}
        \|M(\Psi(\Pi(a))) - M(a)\|_{H^{-1}} \leq \|\Psi(\Pi(a)) - a\|_{H^1}.
    \end{align*} We already have that \(\|a\|_{H^q}\) is small. Moreover, since \(a = \Psi(N(a))\), and since \(N(a) - \Pi(a) = M(a)\) by the assumption \(a \in V\), the smoothing estimate (\ref{lojasiewicz:PsiHqLipschitz}) yields
    \begin{align*}
        \|\Psi(\Pi(a)) - a\|_{H^q} = \|\Psi(\Pi(a)) - \Psi(N(a))\|_{H^q} \leq C\|\Pi(a) - N(a)\|_{H^{q-2}} \leq C\|M(a)\|_{H^{q-2}},
    \end{align*} which is also small. Hence by the other smoothing estimate (\ref{lojasiewicz:PsiH1Lipschitz}),
    \begin{align*}
        |\nabla F(\xi)| &\leq C(\|M(a)\|_{H^{-1}}  + \|\Psi(\Pi(a))) - a\|_{H^1}) \\
        &\leq C(\|M(a)\|_{H^{-1}} + \|\Pi(a) - N(a)\|_{H^{-1}}) \\
        &= C\|M(a)\|_{H^{-1}}.
    \end{align*}This establishes (\ref{lojasiewicz:lowerbound}). \par
    Finally, now that we have (\ref{lojasiewicz:upperbound}) and (\ref{lojasiewicz:lowerbound}), the same computation as in \cite[(2.22)]{leonsimon} yields
    \begin{align*}
        |\mathcal{YM}(A_1+a) - \mathcal{YM}(A_1)|^{1-\theta} \leq C\|M(a)\|_{H^{-1}}.
    \end{align*} Since \(\mathcal{YM}(A_1 + a) \to \mathcal{YM}(A_1)\) uniformly in \(a\), we may take \(\theta\) smaller and shrink \(U_1\) to remove the constant \(C\) in the estimate. The lemma thus follows.
\end{proof}

We also have the following uniform statement for the flat connections.

\begin{lemma}\label{lemma:uniformlojasiewicz}
    Let \(q > \lceil\frac{n}{2}-1\rceil\) and \(q \geq 1\).
    Given a \( (G, \rho) \)-bundle \(E \to M\), let $\mathcal{F} \subset H^q \left(\mathcal{A}_E\right)$
    denote the set of flat connections. There exists \(\theta_0 \in (0, \frac12) \) and a gauge-invariant $H^q$-neighborhood \(U \supset \mathcal{F}\) such that for all \(A \in U\), we have
    \begin{align}\label{uniformlojasiewicz:est}
    \mathcal{YM}(A)^{1 - \theta_0} \leq \mathcal{V}(A).
\end{align} 
Moreover, the solution of (\ref{ymf}) starting at \(A \in U\) exists for all time and converges polynomially to a flat connection in $H^q$.
\end{lemma}
\begin{proof}
    Suppose for 
    contradiction that there does not exist \(\theta \in (0, \frac{1}{2})\) such that every \(B \in \mathcal{F}\) has an \(H^q\) neighborhood \(U_B\) for which the inequality (\ref{uniformlojasiewicz:est}) holds. Then there exists \(\theta_i \searrow 0\) and a sequence of flat connections \(B_i\) such that for each \(i\), the inequality (\ref{uniformlojasiewicz:est}) with \(\theta_0 = \theta_i\) fails on every neighborhood of \(B_i\). Since the \(B_i\) are all flat, strong Uhlenbeck compactness implies that we may pass to a subsequence, also labeled by \(i\), such that there exist gauge transformations \(u_i \in H^{q+1}\) for which \(u_i(B_i)\) converge in \(H^q\) to some flat connection \(B_\infty\). By Lemma \ref{lemma:lojasiewicz}, there exists \(\theta_\infty \in (0, \frac{1}{2})\) and a neighborhood \(U_{B_\infty}\) of \(B_\infty\) such that (\ref{uniformlojasiewicz:est}) holds on \(U\) with \(\theta_0 = \theta_\infty\). Thus for \(i\) sufficiently large, (\ref{uniformlojasiewicz:est}) holds on some neighborhood \(U\) of \(u_i(B_i)\). But (\ref{uniformlojasiewicz:est}) is gauge-invariant, so (\ref{uniformlojasiewicz:est}) holds on the neighborhood \(u_i^{-1}(U_{B_{\infty}})\) of \(B_i\), which is a contradiction, as desired. Thus there exists \(\theta_0 \in (0, \frac{1}{2})\) such that every \(B \in \mathcal{F}\) has a neighborhood \(U_B\) for which (\ref{uniformlojasiewicz:est}) holds with \(\theta_0\). \par 
    By Theorem \ref{thm:continuityatinfinitetime}, we may assume \(U_B\) is such that for all \(A \in U_B\), the solution of (\ref{ymf}) starting at \(A\) exists for all time and converges polynomially in \(H^q\) to a flat connection. Thus the open set
    \begin{align*}
        U := \bigcup_{B \text{ flat, } u \in H^{q+1}(\mathcal{G}_E)} u(U_B)
    \end{align*}is a neighborhood of the flat connections satisfying the assertions of the lemma, as desired. 
\end{proof}

\vspace{5mm}

\subsection{Proof of Theorem \ref{thm:continuityatinfinitetime}} \label{continuityatinftime}
\begin{proof} 
    By Theorem \ref{thm:wellposedness}(1), for $\delta$ is sufficiently small, $A'(T_0)$ will be contained in the neighborhood $B_{\eps}(A_\infty)$ for $T_0$ large enough; we can therefore reduce to the case that the given solution $A(t) \equiv A_\infty$ is static. In particular, given $\eps > 0,$ it suffices to determine $\delta > 0$ such that for $A'_0 \in B_\delta(A_\infty),$ either alternative (1) holds for $A'(t)$ or else
    \begin{equation}\label{alternate(2)}
        \| A'(t) - A_\infty \| < \eps
    \end{equation}
    for all $t$ and $A'(t)$ converges, which is equivalent to (2). 

    Suppose that for \(\delta_\infty > 0\) to be determined, we have
    \begin{align}\label{continuityatinfinitetime:energydoesntdropmuch}
        \mathcal{YM}(A'(t)) > \mathcal{YM}(A_\infty) - \delta_\infty
    \end{align} for $t \in \LB 0 , \tau \right),$ where $\tau \leq \infty$ is the maximal such time. The two cases are distinguished by (1) $\tau < \infty$ and (2) $\tau = \infty.$ 

    Let \(\kappa > 0\) be such that the {\L}ojasiewicz inequality (\ref{lojasiewicz:lojasiewiczinequality}) holds on \(B_{3\kappa}(A_\infty)\). 
We assume that $\delta < \kappa$ is small enough that $A'(t) \subset B_{\kappa}(A_\infty)$ for $0 \leq t \leq 1.$ Now, supposing that \(A'(t) \in B_{3\kappa}(A_\infty)\) for \(t\) in some interval \([a, b] \subset \LB 0, \tau \right) \), 
we compute
    \begin{align*}
        \int_a^{b} & \left\|D_{A'(t)}^\ast F_{A'(t)}\right\| \, dt \\
        &= \int_a^z \frac{-\frac{d}{dt}\left(\mathcal{YM}(A'(t)) - \mathcal{YM}(A_\infty)\right)}{\left\|D_{A'(t)}^\ast F_{A'(t)}\right\|} \, dt + \int_z^{b} \frac{\frac{d}{dt}\left(\mathcal{YM}(A_\infty) - \mathcal{YM}(A'(t))\right)}{\left\|D_{A'(t)}^\ast F_{A'(t)}\right\|} \, dt  \\
        &\leq \int_a^z \frac{-\frac{d}{dt} (\mathcal{YM}(A'(t))-\mathcal{YM}(A_\infty))}{(\mathcal{YM}(A'(t)) - \mathcal{YM}(A_\infty))^{1-\theta}} \, dt + \int_z^{b} \frac{\frac{d}{dt} (\mathcal{YM}(A_\infty) - \mathcal{YM}(A'(t)))}{(\mathcal{YM}(A_\infty) - \mathcal{YM}(A'(t)))^{1-\theta}} \, dt \\
        &\leq \frac{1}{\theta}\left((\mathcal{YM}(A'(a)) - \mathcal{YM}(A_\infty)))^\theta + (\mathcal{YM}(A_\infty) - \mathcal{YM}(A'(b)))^{\theta}\right) \\
        &< \frac{1}{\theta}\left((\mathcal{YM}(A'_0) - \mathcal{YM}(A_\infty))^\theta + \delta_{\infty}^{\theta}\right).
    \end{align*}
    Assuming that $\delta$ and $\delta_\infty$ are sufficiently small, Theorem \ref{thm:Hkestimates} shows that in fact $A'(t)$ remains inside $B_{2\kappa}(A_\infty)$ for all $t \in \LB 0, \tau \right).$ The rest of the argument is a standard application of the \L ojasiewicz-Simon inequality and Gronwall's inequality to $\mathcal{YM}(A'(t)),$ which we omit.
\end{proof}

\subsection*{Data availability and conflict of interest statement}

This paper has no associated data and the authors have no conflicts of interest.

\bibliographystyle{amsinitial}
\bibliography{biblio}

\end{document}